\documentclass[11pt]{amsart}     
\usepackage{arxiv}

\usepackage[utf8x]{inputenc}
\usepackage[T1]{fontenc}
\usepackage{lmodern}
\usepackage[english]{babel}

\usepackage[dvipsnames]{xcolor}
\usepackage{amsmath,amssymb}
\usepackage{amsthm}
\usepackage[backgroundcolor=white, linecolor=red, bordercolor=red, textwidth=3.5cm]{todonotes}
\usepackage{hyperref}
\usepackage{mathtools}
\usepackage[export]{adjustbox}
\usepackage{svg}
\usepackage{graphicx}
\usepackage{stmaryrd}
\usepackage{multicol}

\usepackage[charter, cal=cmcal]{mathdesign}

\usepackage{scalerel}
\usepackage{tikz}
\usepackage{forest}
\usepackage{tikzsymbols}
\usepackage{xfrac}
\usetikzlibrary{arrows,positioning, calc}
\usetikzlibrary{arrows.meta}
\usetikzlibrary{svg.path}
\usetikzlibrary{positioning}
\usetikzlibrary{external}

\tikzset{
    side by side/.style 2 args={
    line width=1pt,
    #1,
	transform canvas = {yshift = 0pt, xshift = 0.5pt},
    postaction={
        postaction={draw,#2, line width = 1pt, transform canvas = {yshift = 0pt, xshift = -1pt}}
        }
    }
}
\tikzstyle{treeNode}[black]=[draw=#1,fill=#1!20, circle, anchor=center,minimum size = 0.5cm]
\tikzstyle{treeNodeLabeled}[black]=[draw=#1, circle, inner sep = 1pt, anchor=center, minimum size = 0.5cm]
\tikzstyle{treeNodeMarked}[red]=[draw=#1,fill=#1!20, circle, minimum size = 0.5cm, inner sep = 1pt,anchor=center]
\tikzstyle{treeNodeBlue}[blue]=[draw=#1,fill=#1!20, circle, minimum size = 0.5cm, inner sep = 1pt,anchor=center]
\tikzstyle{treeNodeInner}=[inner sep = 0pt, circle, fill]
\tikzstyle{treeNodeRoot}=[inner sep = 2pt, circle, fill]
\tikzstyle{treeEdgeMarked}=[red, very thick]
\tikzstyle{treeEdgeBlue}=[{blue!65}, very thick]
\tikzstyle{treeEdgeDouble}=[side by side={red!80}{blue!65}, bend right]

\forestset{
  default preamble={
      for tree={l = 0cm, s sep = 0.3cm,calign=fixed edge angles}
    }
}

\usepackage{relsize}
\usepackage[shortlabels]{enumitem}
\usepackage{nicefrac}
\usepackage[linesnumbered]{algorithm2e}
\usepackage{booktabs}
\usepackage[normalem]{ulem}
\usepackage[justification=centering]{caption}
\usepackage[justification=centering]{subcaption}

\usepackage[draft, margin, index]{fixme}
\fxsetup{theme=color, mode=multiuser}
\FXRegisterAuthor{dp}{adp}{\color{purple}Diane} 
\FXRegisterAuthor{db}{adb}{\color{red}Daniel} 

\DeclareMathOperator{\tr}{tr}
\DeclareMathOperator{\im}{Im}
\DeclareMathOperator{\profile}{Profile}
\DeclarePairedDelimiter{\abs}{\lvert}{\rvert}
\newcommand{\R}{\mathbb{R}}

\newcommand{\T}{\mathcal{T}}
\newcommand{\trees}{\mathbb{T}}

\renewcommand{\L}{\mathcal{L}}
\renewcommand{\O}{\mathcal{O}}
\newcommand{\p}{\mathcal{P}}
\newcommand{\Aut}{\text{Aut}}
\newcommand{\A}{\mathcal{A}}
\newcommand{\V}{\mathbf{V}}
\newcommand{\cat}[1]{\mathrm{Cat}_{#1}}
\newcommand{\dcat}[2]{\mathrm{DCat}^{#1,#2}}
\newcommand{\subgraph}[2]{#1[#2]}
\newcommand*\circled[1]{\tikz[baseline=(char.base)]{
    \node[shape=circle,draw,inner sep=.6pt, scale=.75] (char) {#1};}}

\newcommand*{\red}[1]{
{\color{red}{#1}}
}

\newcommand{\reviseSecond}[1]{#1} 

\newcommand{\revise}[1]{#1} 
\newcommand{\revisedelete}[1]{} 
\newcommand{\revisedeletemath}[1]{}

\newcommand{\restrict}[2]{\revise{{#1}[#2]}}

\newtheorem{theorem}{Theorem}[section]
\newtheorem{claim}{Claim}[theorem]

\newtheorem{lemma}{Lemma}[theorem]
\newtheorem{remark}{Remark}[theorem]
\newtheorem{conjecture}{Conjecture}

\usepackage{authblk} 
\usepackage{orcidlink}

\begin{document}

\captionsetup[figure]{labelfont={bf},name={Fig.},labelsep=space}

\title{Getting to the Root of the Problem: \\ Sums of Squares for \revise{Limits of} Trees
	\thanks{This research was funded in part by the Austrian Science Fund (FWF) [10.55776/DOC78]. For open access purposes, the author has applied a CC BY public copyright license to any author-accepted manuscript version arising from this submission.}
}

\renewcommand{\headeright}{D.\ Brosch and D.\ Puges}
\renewcommand{\undertitle}{}
\renewcommand{\shorttitle}{Sums of Squares for Limits of Trees}

\author[1]{Daniel Brosch \orcidlink{0000-0003-3988-4336}
}
\author[1]{Diane Puges \orcidlink{0000-0002-3864-0762}
}
\affil[1]{Alpen-Adria-Universit\"at Klagenfurt,
	Universit\"atsstra{\ss}e~65--67, 9020 Klagenfurt, Austria, 
	\href{mailto:daniel.brosch@aau.at}{daniel.brosch@aau.at},
	\href{mailto:diane.puges@aau.at}{diane.puges@aau.at}}

\date{}

\maketitle

\begin{abstract}
	The inducibility of a graph represents its maximum density as an induced subgraph over all possible sequences of graphs of size growing to infinity. This invariant of graphs has been extensively studied since its introduction in $1975$ by Pippenger and Golumbic. In $2017$, Czabarka, Sz\'ekely and Wagner extended this notion to leaf-labeled rooted binary trees, which are objects widely studied in the field of phylogenetics. They obtain the first results and bounds for the densities and inducibilities of such trees. Following up on their work, we apply Razborov's flag algebra theory to this setting, introducing the flag algebra of rooted leaf-labeled binary trees. This framework allows us to use polynomial optimization methods, based on semidefinite programming, to efficiently obtain new upper bounds for the inducibility of trees and to improve existing ones. Additionally, we obtain the first outer approximations of profiles of trees, which represent all possible simultaneous densities of a pair of trees \revise{in a sequence of trees of growing sizes}. Finally, we are able to prove the non-convexity of some of these profiles.
\end{abstract}

\section{Introduction}
\label{sec:intro}


Asymptotic extremal graph theory studies the behaviour of graphs whose number of vertices grows towards infinity. It is an old and well-studied area of graph theory in which many areas of mathematics meet, and where even simple-sounding conjectures can prove extremely challenging.
A main topic in this area is the problem of finding the inducibility of a graph, i.e., its maximum density \revise{as an induced subgraph} in an arbitrarily large graph\revise{~\cite{pippenger1975InducibilityGraphs}}. It is a very hard problem that remains open even for many \revise{``easy''} graphs despite being widely studied: for instance, \revisedelete{even determining} the inducibility of the $5$-cycle \revisedelete{is extremely hard} \revise{was determined in $2018$~\cite{Balogh2016}, more than $40$ years after being conjectured by Pippenger and Golumbic~\cite{pippenger1975InducibilityGraphs}}. An adjacent problem consists in determining the profile of a set of graphs: that is, the exact relations between the densities of these graphs \revise{in a common large graph}. In $2017$,  Czabarka, Sz\'ekely and Wagner~\cite{wagner2017} \revise{extended} the notions of density and inducibility to leaf-labeled rooted binary trees, a type of trees stemming from phylogenetics. They \revise{obtained} the first results and bounds for the inducibilities of these trees.
In this paper, we follow up on their work and apply\revise{, for the first time,} flag algebra theory to the setting of leaf-labeled rooted binary trees. Flag algebras are a powerful tool \revise{in} extremal combinatorics, introduced in $2007$ by Razborov~\cite{razborov2007}. They have been successfully used to tackle and solve several very challenging problems of graph theory\revise{~\cite{RAZBOROV2008,FALGASRAVRY2012,Grzesik2012,Keevash2011,BALOGH2014,Lidicky2021,Balogh2023}}, by allowing the application of tools from polynomial optimization, computationally in the form of semidefinite programming, to extremal combinatorics. Here, we use them to develop a computer-assisted way to obtain strong rigorous bounds on inducibilities of trees. As main results, we recover all known inducibilities of small trees (up to a small epsilon), obtain over $300$ new bounds on inducibilities, and the first outer approximations of profiles of trees, for some of which we prove nonconvexity. To do so, we solve the first instance of a generalized problem of moments in flag algebras.

In what follows we describe the organization of the paper.
In Section~\ref{sec:trees}, we start by defining the setting of leaf-labeled rooted binary trees introduced by Czabarka et al.\revise{~\cite{wagner2017}} and the notions of densities, inducibilities and tree-profiles in this context. Section~\ref{sec:flagAlgebra} defines the flag algebra of rooted binary trees and explains the main tools of Razborov's flag algebra theory we use in this setting. We describe in Section~\ref{sec:computation} how we can turn it into a computational approach based on semidefinite programming and polynomial optimization tools. This allows us to compute, in a systematic manner, bounds on the inducibilities of trees. We detail in Section~\ref{sec:results} our results: we recover all known inducibilities of small trees, improve the existing bounds on one \revisedelete{specific} tree \revise{of special interest}, and obtain more than $300$ new bounds as well as the first outer approximations of tree-profiles. Furthermore, we give some exact results on parts of the tree-profiles of two types of trees: the caterpillar tree of size $k$ for $k \in \{4,5,6\}$ and the even tree of size $6$. This profile for $k=4$ is shown in Figure~\ref{fig:profileIntro}\revise{, and explained in more detail in Section~\ref{ssec:ResultsTreeProfiles}}.
\begin{figure}[ht!]
	\centering
	\includegraphics[width=.7\linewidth]{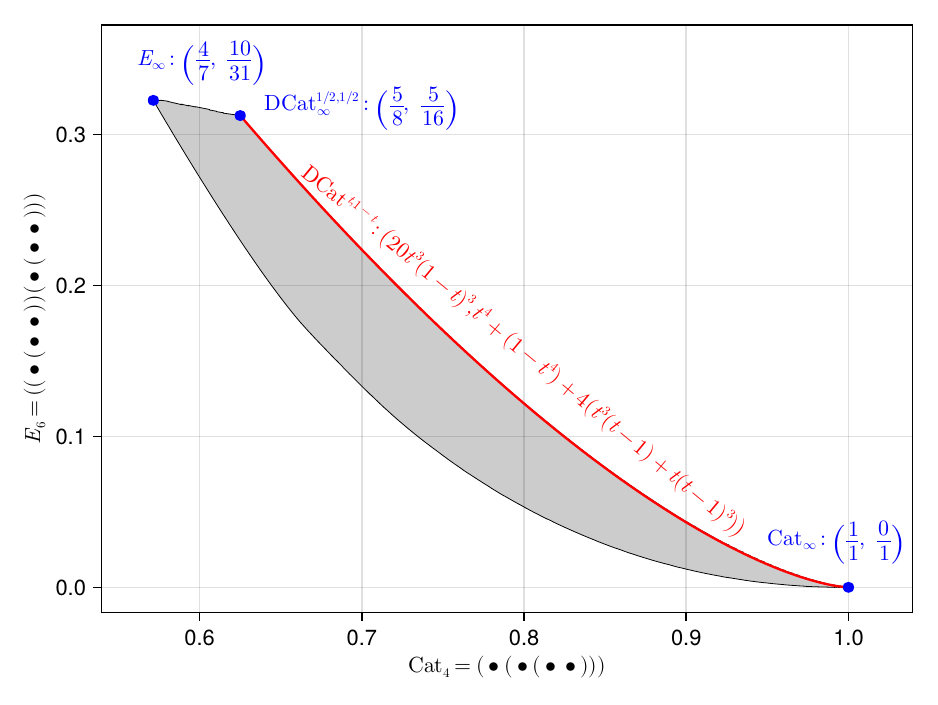}
	\caption{Tree-profile of the caterpillar tree of size $4$ and the even tree of size $6$\\ with three known points and our conjecture for the upper boundary}
	\label{fig:profileIntro}
\end{figure}
In particular, we prove one new characteristic point of the upper boundaries of these profiles using exact sum of squares certificates computed by our flag algebra-based framework. Finally, we are able to prove non-convexity of some tree-profiles, using rigorous (but not sharp) sum of squares certificates generated by our implementation of the flag algebra of rooted binary trees.

\section{Density and inducibility of rooted leaf-labeled binary trees}
\label{sec:trees}

\subsection{Rooted leaf-labeled binary trees and \revisedelete{their} subtrees}

In this article, we are working with rooted leaf-labeled binary trees. ``Leaf-labeled'' \revise{means} that only the leaves are considered ``vertices.'' \revise{This view on vertices significantly changes how we define subtrees (these trees are sometimes called ``topological trees''~\cite{DossouOlory2018a}). The vertices (in the usual graph sense) which are not leaves are called \emph{inner vertices}.} \revise{``Rooted'' means we consider trees in which one vertex, if there is one, has been designated as the tree's \emph{root}. The empty graph $\varnothing$ and the graph consisting of only one vertex are rooted trees. The vertex of the latter is a root and a leaf simultaneously.} ``Binary'' means that every \revise{inner} vertex has exactly two children.

\revise{In this paper, we want to investigate the densities of subtrees in large trees and asymptotic constructions. Typically, trees are sparse objects: As we consider larger and larger trees, the probability that a uniformly sampled induced subgraph of a fixed number of vertices is an independent set approaches one. In the leaf-labeled setting, this changes: trees are no longer sparse but dense objects. We will always obtain another tree if we sample a subtree coming from a subset of leaves, as described below. Accordingly, it is possible to define meaningful subtree densities in trees and their limits here.}

In this paper, every mention of a tree will now refer to a leaf-labeled rooted binary tree, except when explicitly stated otherwise.
For a tree $T$, we denote \revise{by} $\L(T)$ its set of leaves, and we define its size $|T| = |\L(T)|$ as its number of leaves. We consider the trees up to isomorphism; \revise{two trees $T$ and $S$ are isomorphic whenever there is a graph isomorphism between them which preserves the root. We denote this by $T \cong S$}. We call the set of all trees up to isomorphism $\trees$, and let $\trees_n\subseteq \trees$ denote the subset of trees with $n$ leaves. \revise{We denote by $\varnothing$ the empty tree and consider it to be of size $0$}. \revise{For $n \geq 1$,} the number of \revisedelete{such} trees in $\trees_n$ is the Wedderburn-Etherington number $W_n$: The first few numbers in the sequence are $1,1,1,2,3,6,11,23 \ldots$ \revise{(OEIS sequence A001190~\cite{Sloane})}. Figure~\ref{fig:treeExamples} depicts the one and only tree of size $3$, the two trees of size $4$, as well as the three trees of size $5$.
\begin{figure}[h!]
	\captionsetup[subfigure]{font=small}
	\centering
	\begin{minipage}{.25\textwidth}
		\centering
		\subcaptionbox{Tree of size $3$}[\textwidth]{%
			\centering
			\includegraphics[width=.37\linewidth]{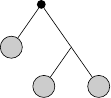}}
	\end{minipage}
	\begin{minipage}{.3\textwidth}
		\subcaptionbox{Trees of size $4$}[\textwidth]{%
			\begin{subfigure}{.45\textwidth}
				\centering
				\includegraphics[width=\linewidth]{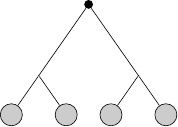}
			\end{subfigure}%
			\begin{subfigure}{.45\textwidth}
				\centering
				\includegraphics[width=.8\linewidth]{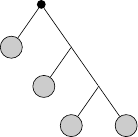}
			\end{subfigure}}
	\end{minipage}%
	\begin{minipage}{.4\textwidth}
		\hspace*{.5cm}
		\subcaptionbox{Trees of size $5$}[\textwidth]{%
			\begin{subfigure}{.3\textwidth}
				\centering
				\includegraphics[width=1.2\linewidth]{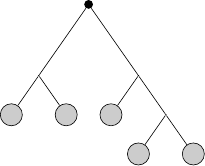}
			\end{subfigure}%
			\begin{subfigure}{.3\textwidth}
				\centering
				\includegraphics[width=1\linewidth]{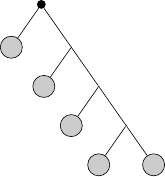}
			\end{subfigure}%
			\begin{subfigure}{.4\textwidth}
				\centering
				\includegraphics[width=.77\linewidth]{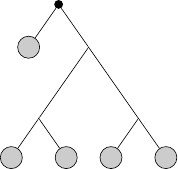}
			\end{subfigure}}
	\end{minipage}
	\caption{All trees of sizes $3$, $4$ and $5$.}
	\label{fig:treeExamples}
\end{figure}
These trees are essential to the topic of phylogenetics where they are extensively studied~\cite{Semple2009,Bienvenu2024}. \revisedelete{They are used to represent the evolutionary history of a set of species, and studying their properties is primordial to understand the evolution of species with respect to each other.}

\begin{figure}[h!]
	\captionsetup[subfigure]{font=footnotesize}
	\centering
	\begin{subfigure}{.17\textwidth}
		\resizebox{.9\textwidth}{!}{%
			\begin{forest}
				[, treeNodeRoot
						[, treeNodeInner
								[1, treeNodeLabeled]
								[2, treeNodeLabeled]
						]
						[, treeNodeInner
								[3, treeNodeLabeled]
								[, treeNodeInner
										[4, treeNodeLabeled]
										[5, treeNodeLabeled]
								]
						]
				]
			\end{forest}
		}
		\caption{}
		\label{fig:subtree_step1}
	\end{subfigure}
	\begin{subfigure}{.17\textwidth}
		\resizebox{.9\textwidth}{!}{%
			\begin{forest}
				[, treeNodeRoot
						[, treeNodeInner
								[1, treeNodeMarked]
								[2, treeNodeLabeled]
						]
						[, treeNodeInner
								[3, treeNodeMarked]
								[, treeNodeInner
										[4, treeNodeMarked]
										[5, treeNodeLabeled]
								]
						]
				]
			\end{forest}
		}
		\caption{}
		\label{fig:subtree_step2}
	\end{subfigure}%
	\centering
	\begin{subfigure}{.17\textwidth}
		\resizebox{.9\textwidth}{!}{%
			\begin{forest}
				[, treeNodeRoot
				[, treeNodeInner, edge=treeEdgeMarked
				[1, treeNodeMarked, edge=treeEdgeMarked]
				[2, treeNodeLabeled]
				]
				[, treeNodeInner, edge=treeEdgeMarked
				[3, treeNodeMarked, edge=treeEdgeMarked]
				[, treeNodeInner, edge=treeEdgeMarked
				[4, treeNodeMarked, edge=treeEdgeMarked]
				[5, treeNodeLabeled]
				]
				]
				]
			\end{forest}
		}
		\caption{}
		\label{fig:subtree_step3}
	\end{subfigure}
	\begin{subfigure}{.17\textwidth}
		\resizebox{.6\textwidth}{!}{%
			\begin{forest}
				[, treeNodeRoot
				[, treeNodeInner, edge=treeEdgeMarked
				[1, treeNodeMarked, edge=treeEdgeMarked]
				]
				[, treeNodeInner, edge=treeEdgeMarked
				[3, treeNodeMarked, edge=treeEdgeMarked]
				[, treeNodeInner, edge=treeEdgeMarked
				[4, treeNodeMarked, edge=treeEdgeMarked]
				]
				]
				]
			\end{forest}
		}
		\caption{}
		\label{fig:subtree_step4}
	\end{subfigure}
	\begin{subfigure}{.17\textwidth}
		\resizebox{.6\textwidth}{!}{%
			\begin{forest}
				[, treeNodeRoot
						[1, treeNodeLabeled]
						[, treeNodeInner
								[3, treeNodeLabeled]
								[4, treeNodeLabeled]
						]
				]
			\end{forest}
		}
		\caption{}
		\label{fig:subtree_step5}
	\end{subfigure}
	\caption{Subtree induced by the set of leaves $\{1,3,4\}$}
	\label{fig:subtreeExample}
\end{figure}
We use the notation $S \subseteq T$ to express that $S$ is a subtree of $T$.
We define the \emph{height} of a leaf as its distance (including inner vertices) to the root. For instance, in the tree in Figure~\ref{fig:leafHeight}, leaf $1$ has height $1$, leaf $2$ has height $2$, and all other leaves have height $4$.
\begin{figure}[h!]
	\centering
	\resizebox{.1\textwidth}{!}{%
		\begin{forest}
			[, treeNodeRoot
					[1, treeNodeLabeled]
					[, treeNodeInner
							[2, treeNodeLabeled]
							[, treeNodeInner
									[, treeNodeInner
											[, treeNode]
											[, treeNode]
									]
									[, treeNodeInner
											[, treeNode]
											[, treeNode]
									]
							]
					]
			]
		\end{forest}}
	\caption{Height of a leaf}
	\label{fig:leafHeight}
	\vspace{-.5cm}
\end{figure}
Some trees have particular structures that make them especially interesting to study. We call \revise{a} \emph{caterpillar tree} (or sometimes just caterpillar) of size $n$, denoted by $\cat{n}$, the tree where each inner node has as children a leaf and another caterpillar tree of size $n-1$, the caterpillar tree of size $1$ being the \revise{tree} consisting of one leaf \revise{(see Figure~\ref{fig:caterpillar})}. In other words, the inner vertices of the caterpillar tree form a path. Caterpillars form a very important set of trees with interesting and useful properties; their unrooted counterparts are also of great importance in \revise{the study of phylogenetic trees~\cite{alon2016PhyloTrees} and were crucial in the proof of $\mathcal{NP}$-completeness of the maximum quartet consistency problem~\cite{Steel1992}}.
\begin{figure}[h!]
	\centering
	\includegraphics[width=.13\linewidth]{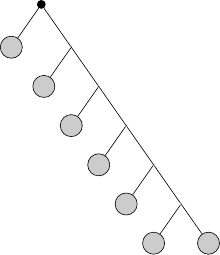}
	\caption{$\cat{7}$}
	\label{fig:caterpillar}
\end{figure}
We call \revise{an} \emph{even tree} of size $n$ and denote \revise{by} $E_n$ the tree \revise{where} the two \revise{subtrees} of each inner node have a size difference of at most $1$\revise{, as depicted in Figure~\ref{fig:eventrees}}. In particular, for every $k \geq 1$, the even tree of size $2^k$ is the complete binary tree of \revise{size} $k$.

\begin{figure}[h!]
	\captionsetup[subfigure]{font=footnotesize}
	\centering
	\begin{subfigure}{.24\textwidth}
		\centering
		\includegraphics[width=.5\linewidth]{Pictures/Tree-7.pdf}
		\caption{$E_5$}
	\end{subfigure}
	\begin{subfigure}{.24\textwidth}
		\centering
		\includegraphics[width=.5\linewidth]{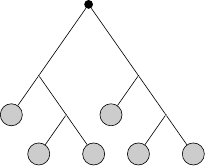}
		\caption{$E_6$}
	\end{subfigure}
	\begin{subfigure}{.24\textwidth}
		\centering
		\includegraphics[width=.5\linewidth]{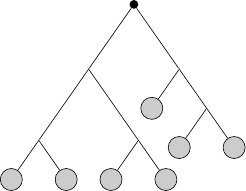}
		\caption{$E_7$}
	\end{subfigure}
	\begin{subfigure}{.24\textwidth}
		\centering
		\includegraphics[width=.75\linewidth]{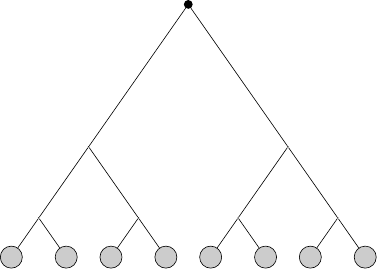}
		\caption{$E_8$}
	\end{subfigure}
	\caption{Even trees of sizes $5$ to $8$.}
	\label{fig:eventrees}
\end{figure}
\subsection{Inducibility of trees}

\paragraph{Inducibility of graphs.}
The notion of inducibility of graphs has been introduced in $1975$ by Pippenger and Golumbic~\cite{pippenger1975InducibilityGraphs}, who studied the maximum frequency with which a fixed graph on $k$ vertices can appear in another graph whose number of vertices goes to infinity. It is an extensively studied topic that has been approached in many different ways. In particular, Razborov's flag algebra theory has already been successfully applied to problems of inducibility in graphs. Indeed, it was used by Sperfeld in $2011$ to determine new bounds on the inducibility of several oriented graphs. In $2012$, Falgas-Ravry and Vaughan~\cite{FalgasRavry2012a} \revise{introduced} the software `Flagmatic' to determine exact inducibilities and upper bounds for several $3$-graphs up to $5$ vertices, as well as for oriented star graphs of size $3$ and $4$. In $2013$, Hirst~\cite{Hirst2013} \revise{determined} the inducibilities \revisedelete{in graphs} of $K_{1,1,2}$ and \revisedelete{of} the so-called paw graph.  In $2016$, Balogh, Hu, Lidický and Pfender~\cite{Balogh2016} \revise{proved}, using flag algebras, that the maximum density of the induced $5$-cycle $C_5$ is achieved by an iterated blow-up of \revise{the} 5-cycle. Following up on this, Lidický, Mattes and Pfender~\cite{Lidicky2023} \revise{determined} in $2022$ the maximum number of induced copies of $C_5$ in a graph on $n$ vertices, as well as all the \revisedelete{maximizer graphs} \revise{graphs in which this maximum is attained}.
\paragraph{Inducibility of leaf-labeled rooted binary trees.}
Czabarka et al.~\cite{wagner2017} \revise{adapted} the concept of inducibility in graphs to rooted leaf-labeled binary trees in the following way: We consider two trees $S$ and $T$, with $|S|=k$ and $|T|=n$ ($n \geq k$). $\mathcal{C}(S,T)$ is the number of \revisedelete{induced} subtrees \revise{of} $T$ that are isomorphic to $S$. The \emph{subtree density} of the tree $S$ in the tree $T$ is
\begin{align*}
	p(S;T)\coloneqq \frac{\mathcal{C}(S,T)}{\binom{n}{k}}.
\end{align*}
It is the proportion of copies of $S$ in $T$ over all subtrees of the same size of $T$. In other words, it is the probability of obtaining a subtree isomorphic to $S$ when picking uniformly at random a subtree of size $|S|$ in $T$. As such, we can reformulate the definition as
\begin{align*}
	p(S;T) & = \mathbb{P}[\revise{\subgraph{T}{\V}} \cong {S}] \in [0,1],
\end{align*}
where $\V$ is a \revise{uniformly} random subset of leaves of $T$ of size $|S|$, and $\revise{\subgraph{T}{\V}}$ is the subtree of $T$ induced by $\V$.

We now consider an increasing sequence of trees $\T=  (T_k)_{k\geq \revise{0}}$, i.e. \revise{such that} $(|T_k|)_{k \geq 0}$ is strictly increasing.
We define the subtree density of a tree $S$ in this increasing sequence of trees as
\begin{align*}
	\phi_{\T}(S) & \coloneqq \lim\limits_{n\rightarrow\infty} p(S;T_n)                                                            = \lim\limits_{n\rightarrow\infty} \mathbb{P}[\revise{\subgraph{T_n}{V_n}} \cong {S}] \in [0,1],
\end{align*}
where $\mathbf{V_n}$ is a random subset of leaves of $T_n$ of size $|S|$. Note that \revise{the vector of all densities $(p(S;T_n))_{S\in\trees}$ lies in $[0,1]^{\trees}$, which is compact by Tychonoff's theorem. Thus,} we can always find a subsequence such that all the densities converge. From now on, we always work (implicitly) with such \emph{converging increasing \revisedelete{sub}sequences}.


The \emph{inducibility} of the tree $S$ is then the maximum \revise{subtree density of $S$ in any} converging increasing sequences of trees:
\begin{align}
	\label{eq:defInducibility}
	I(S) \coloneqq \max\limits_{\T} \phi_{\T}(S).
\end{align}

This coincides with the definition of the inducibility of a tree in~\cite{wagner2017} as the limit superior of the maximum subtree density of $S$ in a tree:
\begin{align*}
	\limsup\limits_{|T| \rightarrow \infty} p(S;T)
	= \limsup\limits_{n \rightarrow \infty} \max\limits_{T:|T|=n} p(S;T)
\end{align*}
Indeed, by Tychonoff's theorem, the sequence of maximizers of $p(S;\cdot)$ of size $n$ has a converging subsequence. \revise{Thus, the definition of the inducibility of trees in~\cite{wagner2017}  matches our definition of $I(S)$, and the maximum in~\eqref{eq:defInducibility} is attained.}

\paragraph{Previous results on the inducibility of trees.}

In $2017$, Czabarka et al.~\cite{wagner2017} \revise{proved} that every binary tree has positive inducibility and that the only trees with inducibility $1$ are the caterpillar trees. They also \revise{obtained} the first inducibilities of binary trees: the complete binary tree of height $2$,
the even tree of size $r$; and \revise{provided} bounds \revise{for} other trees. They \revise{applied} these results to crossings in random tanglerams, which are pairs of binary trees of same size, whose leaves are joined by a perfect matching.
\revisedelete{These tanglegrams are related to the tree-pairs of phylogenetic trees studied in~\cite{alon2016PhyloTrees}}.

Several articles have also focused on the inducibility of $d$-ary trees, i.e., trees in which each of the non-leaf vertices of has between $2$ and $d$ children.
In $2018$, Dossou-Olory and Wagner~\cite{DossouOlory2018a} \revise{studied} the relation between inducibilities with fixed degree and with bounded degree; they also \revise{provided} a lower bound on the limit inferior on the density of the binary caterpillar. Additionally, they \revise{computed} in~\cite{DossouOlory2018} stronger bounds on the inducibility of small, but very challenging binary and ternary trees. In $2021$, they \revise{were} able to obtain more exact inducibilities and bounds on binary and $d$-ary trees, and \revise{obtained} results on the speed of convergence of the maximum density of a tree~\cite{DossouOlory2021}.
In $2020$, Czabarka et al.~\cite{wagner2020} proved the following theorem: for any $d$-ary tree $\revise{S}$ (i.e., each of the non-leaf vertices of $\revise{S}$ has between $2$ and $d$ children), $$\max\limits_{|\revise{T}| = n} p(\revise{S;T}) = I(\revise{S}) + \O(n^{-1}).$$  They showed as well that, for any binary tree $\revise{S}$, the inducibility of $\revise{S}$ in binary trees is equal to its inducibility in $d$-ary trees and in strictly $d$-ary trees (trees in which every vertex has exactly $0$ or $d$ children).

In $2016$, Alon, Naves and Sudakov~\cite{alon2016PhyloTrees} \revise{studied} the density of some patterns of size $4$ (``quartets'') in tree-pairs of trivalent trees, which are unrooted trees in which every non-leaf vertex has exactly three neighbors.  To do so, they \revise{used} flag algebra calculus that they adapt to this particular setting. They then \revise{solved} a semidefinite program to obtain bounds and various results on the density of such patterns. They \revise{obtained} as well some further results on (unrooted) caterpillar trees, by seeing them as permutations and using \revisedelete{flag algebras applied to permutations theory} \revise{the flag algebras of permutations}.

\subsection{Tree profiles}

The \emph{Tree-profile} of trees $T$ and $S$ is the set of pairs of densities $(\phi_{\T}(T),\phi_{\T}(S))$ that can be attained simultaneously from the same increasing sequence of trees $\T$.
\begin{align*}\profile (T,S) \coloneqq \{(\phi_{\T}(T),\phi_{\T}(S))\mid \T \text{ is an increasing sequence of trees}\} \subseteq [0,1]^2
\end{align*}

Graph profiles were first introduced in $1979$ by Erd\H{o}s, Lov\'asz, and Spencer~\cite{Erdoes1979}, who investigated how the densities of two graphs behaved with respect to each other. In particular, they showed that the profile of any $m$ connected graphs is full-dimensional in $\R^m$. However, graph profiles have proven \revise{to be} very challenging to study; they are not necessarily convex or even semialgebraic sets, and still very little is known about them.
A major breakthrough on the topic was obtained in $2008$ by Razborov~\cite{RAZBOROV2008}, who used his flag algebra theory to obtain the exact relation between the density of triangles in a graph and its edge density. In other words, he gave the first full description of a graph profile, the profile of the triangle and the edge graph. This problem was introduced by Tur\'an in $1941$~\cite{Zarankiewicz1955} and was until then still (mostly) unsolved. The case of edge versus the complete graph $K_4$ was then solved by Nikiforov in $2010$~\cite{Nikiforov2010}. The general case of $K_n$ and the edge was solved in $2016$ by Reiher~\cite{Reiher2016}, proving a conjecture of Lov\'asz and Simonovits~\cite{Lovasz1983}.

There is no known full profile of three or more connected graphs, though there exist partial results. In $2013$, Huang et al.~\cite{Huang2013} \revise{studied} profiles of four induced graphs with up to three vertices, and \revise{described} exactly the profile of 3-cliques and 3-anticliques on three vertices. They also \revise{gave} an exact description in the triangle-free case, with the use of flag algebras. Glebov et al.~\cite{Glebov2016} \revise{followed} on this work in $2016$ by determining all profiles of two induced graphs of size $3$.
Graph profiles of the edge and the $k$-edge path have been determined in $2016$ by Nagy for $n=4$~\cite{nagy2016}, notably using results from Ahlswede and Katona who proved the upper boundary in the case $k=2$ in $1978$~\cite{Ahlswede1978}.
In $2018$, Reiher and Wagner~\cite{Reiher2018} obtained an upper bound on the profile of the edge and the $k$-star for all integers $k\leq2$.
Recently, Cairncross and Muyabi~\cite{Cairncross2025} obtained exact profiles of ordered graphs, and the first results on profiles of colored graphs.
In $2022$, Blekherman et al.~\cite{Blekherman2022} computed and studied the tropicalization of graph and hypergraph profiles and, through it, exhibited limitations for the sums of squares method to prove graph density inequalities.

In $2016$, Bubeck and Linial~\cite{bubeck2016TreeProfiles} \revise{defined} the density of a tree $S$ in another tree $T$ as the number of copies of $S$ divided by the number of connected subgraphs of size $\abs{S}$ inducing a tree in $T$. In this context, they \revise{investigated} the limit sets of $k$-profiles of trees, that is, the profile of all trees on $k$ vertices. They \revise{showed} that these profiles are always convex. Following up on this paper, Chan et al.~\cite{Chan2022} proved further results on the densities and inducibilities of such trees.

\section{The flag algebra of rooted binary trees}
\label{sec:flagAlgebra}

Flag algebras, first introduced by Alexander Razborov in 2007 \cite{razborov2007}, are one of the most powerful and promising tools in extremal combinatorics. They provide a way to formulate extremal problems in graphs analogously to polynomial optimization problems. These can then be relaxed into semidefinite programs that can efficiently be solved by a computer, thus providing automated certificates for (often tight) bounds for such problems.
The similarities to polynomial optimization are not coincidental: recent developments have brought to light that sums of squares in flag algebras correspond to solving the limit of a sequence of polynomial optimization problems. The problems in the sequence grow in the number of variables (but not degree), counterbalanced by exhibiting more and more symmetries \cite{Raymond2017}. Interpreted this way, flag algebras can be recovered from the \emph{representation stability} of the problem, relating to \emph{dimension free descriptions} of the relevant cones, a more general concept recently investigated in the setting of \revise{convex} optimization in~\cite{Levin2023}.

Flag algebras have been used for a wide variety of problems, from densities in \\ graphs~\cite{FALGASRAVRY2012,Grzesik2012} to hypergraphs~\cite{Keevash2011}, permutations~\cite{BALOGH2014}, Ramsey numbers~\cite{Lidicky2021}, crossing numbers of graphs~\cite{Balogh2023}, and more.
Flag algebras have been used by Alon, Naves and Sudakov in~\cite{alon2016PhyloTrees} to study tree-pairs of trivalent phylogenetic trees. To the best of our knowledge, they have so far not been used on other types of trees, in particular not on rooted binary trees. A thorough overview of the theory of flag algebras is given in~\cite{oliveira2016FlagAlgebras}, another one in~\cite{brosch2022symmetry}, although more focused on non-induced flag algebras. We explain it here in the setting of trees.

\subsection{Model theoretic view}

Razborov introduced flag algebras in the setting of first-order model theory~\cite{razborov2007}. \revise{He works with a language $L$ containing predicate symbols (to be interpreted in models as \emph{relations} between tuples of elements), but no functions or distinguished elements. The theory $T$ is assumed to be universal, i.e., all its axioms can be written in the form
	\begin{equation*}
		\forall x_1\forall x_2\ldots \forall x_n\colon \phi(x_1,\ldots, x_n),
	\end{equation*}
	where $\phi$ is a quantifier free formula in $L$. Furthermore, he assumes existence of an infinite model of $T$, implying that models of each finite size exist.}
To avoid having to reprove \revise{Razborov's} results in our setting, we provide a model theoretic description of leaf-labeled trees. For this, \revise{we give in Theorem \ref{th:predicate} a theory with one three-argument predicate $\p(\cdot;\cdot, \cdot)$ and three axioms such that the models of the theory are our trees.}

\begin{theorem}
	\label{th:predicate}
	Let $n \geq 3$, and let $\p\subseteq [n]^3$ \revise{be a relation. We write $\p(i;j,k)$ for $(i,j,k)\in\p$.} The two following statements are equivalent:
	\begin{enumerate}[(i)]
		\item\label{it:pred1}There exists a unique tree $T$ such that, for any three distinct leaves $\revise{i,j,k}\in\L(T)$,
		      \begin{equation}
			      \p(i;j,k) \revise{ \Leftrightarrow } \text{ in the subtree of T induced by $\{i,j,k\}$, $i$ is at height $1$}\label{eq:defPredicate}
		      \end{equation}
		      \revise{No relations between non-distinct $i,j,k$ hold.}
		\item\label{it:pred2} $\p$ satisfies the following axioms \revise{for all distinct $i,j,k,\ell\in[n]$}:
		      \begin{enumerate}[(a)]
			      \item $\p(i;j,k) \revise{\ \Leftrightarrow \ } \p(i;k,j)$, \label{item:axiom1}
			      \item exactly one of $\p(i;j,k), \p(j;i,k), \p(k;i,j)$ is \revise{true},\label{item:axiom2}
			      \item $\p(i;j,k)\Rightarrow \p(\ell;j,k) \lor \p(i;j,\ell)$.\label{item:axiom3}
		      \end{enumerate}
	\end{enumerate}
\end{theorem}

\begin{proof}

	We start by proving that~\ref{it:pred1} implies~\ref{it:pred2}.
	Let us assume that $\p\revise{\subseteq} [n]^3$ satisfies~\ref{it:pred1}: $\p$ then represents a unique tree $T$ of size $n$. Axioms~\ref{item:axiom1} and~\ref{item:axiom2} follow directly from the definition of a tree. To prove axiom~\ref{item:axiom3}, let us consider \revise{distinct} $i,j,k,\ell \in [n]$. We \revisedelete{can} assume \revisedelete{without loss of generality} that $\p(i;j,k)$ \revise{is true}. We then consider the subtree induced by $\{\ell,j,k\}$ in $T$. If $\ell$ is at height $1$ in this subtree, $\p(\ell;j,k)$ \revise{is true} and~\ref{item:axiom3} holds. If not, we can assume without loss of generality that $j$ is at height $1$, i.e., $\p(j;\ell,k)$ \revise{holds}. Then, since $i$ is at height $1$ in the subtree induced by $\{i,j,k\}$, $i$ has to be at height $1$ as well in the subtree induced by $\{i,j,\ell\}$, hence $\p(i;j,\ell)$ \revise{is true} and~\ref{item:axiom3} holds as well. \\

	We now prove that~\ref{it:pred2} implies \ref{it:pred1}. To do so, we proceed by induction.
	For $n=3$, let $\p\revise{\subseteq}\{1, 2, 3\}^3$ fulfill the axioms of~\ref{it:pred2}. Then, there exists $i \in \{1,2,3\}$ such that $\p(i;j,k)$ and $\p(i;k,j)$ \revise{hold}, where $\{i,j,k\} = \{1,2,3\}$. By assumption, $\p$ is \revise{false on the remaining tuples in} $\{1, 2, 3\}^3$. Then, with $T$ the tree of size $3$ with leaf $i$ at height $1$ and leaves $j$ and $k$ at height $2$, represented in Figure~\ref{fig:treeIJK}, ~\ref{it:pred1} holds.

	\begin{figure}[ht!]
		\centering
		\resizebox{.12\textwidth}{!}{%
			\begin{forest}
				[, treeNodeRoot
						[\scriptsize$i$, treeNodeLabeled]
						[, treeNodeInner
								[j, treeNodeLabeled]
								[k, treeNodeLabeled]
						]
				]
			\end{forest}}
		\caption{The unique tree on $\{i,j,k\}$ with $\p(i;j,k)$ \revise{being true}}
		\label{fig:treeIJK}
	\end{figure}

	Let $n\geq 4$, and suppose that for any $3 \leq k \leq n$, any $\p\revise{\subseteq} [k]^3$ satisfying ~\ref{it:pred2} also satisfies \ref{it:pred1}. Let $\p\revise{\subseteq} [n+1]^3$ fulfill the axioms of~\ref{it:pred2}. We denote by $\p|_n$ be the restriction of $\p$ to $[n]^3$. By induction, $\p|_n$ uniquely describes a tree $T_n$ of size $n$. We will now show that $\p$ uniquely describes a tree $T_{n+1}$ obtained by attaching to $T_n$ the leaf $(n+1)$ in a position uniquely defined by $\p$.

	We define a graph $G$ with vertices $[n]=\{1, \ldots, n\}$ and an edge for every pair $\{i,j\}$ such that $\p(n+1;i,j)$ \revise{holds}. We denote by $T_1$ and $T_2$ the two subtrees joining at the root of $T_n$, and by $\L_1$ and $\L_2$ their respective sets of leaves: $\L_1 = \L(T_1)$ and $\L_2 = \L(T_2)$.
	Note that for every $i$ in $\L_1$ (resp. $\L_2$), for every $j,k$ in $\L_2$ (resp. $\L_1$), $\p(i;j,k)$ holds. Indeed, leaves from the same branch will have a common inner node higher than the root, but two leaves in separate branches will only join at the root node. By the same reasoning, $\p(i;j,k)$ \revise{holding} for some $\{i,j,k\}$ implies that $j$ and $k$ have to be in the same branch $T_1$ or $T_2$.

	We make the following claims:

	\begin{claim}
		At least one of $\revisedeletemath{\restrict{G}{T_1}}$ $\restrict{G}{\revise{\L_1}}$ and $\revisedeletemath{\restrict{G}{T_2}}$ $\restrict{G}{\revise{\L_2}}$ is a complete graph.
	\end{claim}
	\begin{proof}
		Let us assume that neither $\restrict{G}{\revise{\L_1}}$ nor $\restrict{G}{\revise{\L_2}}$ are complete. This means that there exist \revise{$i,j$ in $\L_1$} and \revise{$i',j'$ in $\L_2$} such that $\p(n+1;i,j)$ and $\p(n+1;i',j')$ \revise{are false}.
		Since $i$ is in $\L_1$ and both $i'$ and $j'$ are in $\L_2$, $\p(i;i',j')$ \revise{holds}. As $\p$ fulfills~\ref{item:axiom3}, this implies $\p(n+1;i',j') \revise{\lor} \p(i;i',n+1)$. Because we assumed $\p(n+1;i',j')$ \revise{does not hold},  $\p(i;i',n+1)$ \revise{is true}. Similarly, $\p(i';i,j)$ \revise{holds}, which implies $\p(n+1;i,j)\revise{\lor}\p(i';i,n+1)$, hence $\p(i';i,n+1)$ \revise{holds} by assumption. Finally, we obtain that both $\p(i;i',n+1)$ and $\p(i';i,n+1)$ \revise{hold}, which contradicts axiom~\ref{item:axiom2}.\qed
	\end{proof}

	\begin{claim}
		$\p(n+1;\cdot,\cdot)$ is constant on $\L_1 \times \L_2$, i.e.,
		for every $i,j \in \L_1$, for every $i',j' \in \L_2$, $\p(n+1;i,i') \revise{\Leftrightarrow} \p(n+1;j,j')$.
	\end{claim}
	\begin{proof}
		Let $i,j \in \L_1$ and $i',j' \in \L_2$.
		Let's assume, without loss of generality, that $\restrict{G}{\revise{\L_1}}$ is complete. Then, $\p(n+1;i,j)$ \revise{holds}. If $\p(n+1;i,i')$ \revise{is true}, then~\ref{item:axiom3} implies $\p(j';i,i') \revise{\lor} \p(n+1;i,j')$. Since $i$ is in $\L_1$ and $i',j'$ are in $\L_2$, $\p(i;i',j')$ \revise{is true}, hence \revise{also} $\p(n+1;i,j')$ \revise{and} $\p(n+1;j',i)$. This implies in turn $\p(j;j',i) \revise{\lor} \p(n+1;j',j)$. Since $(i,j)$ and $j'$ are in separate branches, we thus obtain that $\p(n+1;j',j)$, $\p(n+1;j,j')$, and $\p(n+1;i,i')$ \revise{all hold}.
		If $\p(n+1;i,i')$ \revise{does not hold}, \revise{as} $\p(j';i,i')$ \revise{is also false}, the contraposition of axiom~\ref{item:axiom3} implies \revise{neither} $\p(n+1;i,j')$ \revise{nor} $\p(n+1;j',i)$ \revise{hold}. These results, combined with the fact that $\p(i;j',j)$ \revise{is false}, imply in turn that \revise{none of} $\p(n+1;j',j)$, $\p(n+1;j,j')$, and $\p(n+1;i,i')$ \revise{can hold}.\qed
	\end{proof}

	The first claim allows us to split the problem in five cases, depending on the completeness of $G$ and its restrictions.
	We start by considering the case where both $\restrict{G}{\revise{\L_1}}$ and $\restrict{G}{\revise{\L_2}}$ are complete. This directly implies that the vertex $(n+1)$ has to be attached at height $1$ or $2$ of the tree. Three different configurations are possible, illustrated in Figure~\ref{fig:proofPredicate}. First, if $G$ is complete, $\p(n+1;i,j)$ \revise{holds} for all vertices $\{i,j\}$ of $T_n$. The tree with $n+1$ at height $1$, and the lowest common root of $T_1$ and $T_2$ at height $2$ then fulfills~\eqref{eq:defPredicate} for $\p$, and it is the only one (Fig.~\ref{fig:case3}). If $G$ is not complete, $(n+1)$ has to be attached at a height of exactly $2$ of the tree, that is, at the root of either $T_1$ or $T_2$. Since $\restrict{G}{\revise{\L_1}}$ and $\restrict{G}{\revise{\L_2}}$ are complete, $G$ misses an edge between \revisedelete{$T_1$} \revise{$\L_1$} and \revisedelete{$T_2$} \revise{$\L_2$}. There are then two possibilities. If there exists $i \in \L_1$ and $j\in \L_2$, such that $\p(i;n+1,j)$ \revise{is true}, $(n+1)$ has to be attached in $T_2$. Then, the tree where the root of $T_1$ is at height $1$, and $n+1$ and the root of $T_2$ are at height $2$ is the unique tree fulfilling~\eqref{eq:defPredicate} for $\p$ (Fig.~\ref{fig:case4}). Otherwise, there exists $i \in \L_1$ and $j\in \L_2$, such that $\p(j;n+1, i)$ \revise{holds}. Then the unique tree fulfilling~\eqref{eq:defPredicate} for $\p$ has the root of $T_2$ at height $1$, and $n+1$ and the root of $T_1$ at height $2$ (Fig.~\ref{fig:case5}).

	We now consider the case where $\restrict{G}{\revise{\L_1}}$ is not complete: there exist $(i,j)$ in $\L_1^2$ such that $\p(i;n+1,j)$ \revise{holds}. It implies that vertex $(n+1)$ has to be attached somewhere in $T_1$. We then repeat the process recursively in $T_1$: our second claim ensures that there will be no contradiction with the predicate on $T_2$. This provides us with one unique way of attaching $(n+1)$ to $T_1$ and thus to $T_n$. We proceed analogously if $\restrict{G}{\revise{\L_2}}$ is not complete.	\qed
\end{proof}

\vspace*{-.5cm}
\begin{figure}[ht!]
	\captionsetup[subfigure]{font=footnotesize}
	\centering
	\subcaptionbox{$G$ is complete \label{fig:case3}}[.28\textwidth]{
		\resizebox{.12\textwidth}{!}{%
			\begin{forest}
				[, treeNodeRoot
						[\scriptsize$n\!+\!1$, treeNodeLabeled]
						[, treeNodeInner
								[$T_1$, treeNodeLabeled]
								[$T_2$, treeNodeLabeled]
						]
				]
			\end{forest}}}
	\subcaptionbox{$\p(i;n+1, j)$ \revise{holds} \\ with $i \in \L_1$ and $j\in \L_2$\label{fig:case4}}[.28\textwidth]{
		\resizebox{.12\textwidth}{!}{%
			\begin{forest}
				[, treeNodeRoot
						[$T_1$, treeNodeLabeled]
						[, treeNodeInner
								[\scriptsize$n\!+\!1$, treeNodeLabeled]
								[$T_2$, treeNodeLabeled]
						]
				]
			\end{forest}}}
	\subcaptionbox{$\p(j;n+1, i)$ \revise{holds} \\ with $i \in \L_1$ and $j\in \L_2$\label{fig:case5}}[.28\textwidth]{
		\resizebox{.12\textwidth}{!}{%
			\begin{forest}
				[, treeNodeRoot
						[$T_2$, treeNodeLabeled]
						[, treeNodeInner
								[\scriptsize$n\!+\!1$, treeNodeLabeled]
								[$T_1$, treeNodeLabeled]
						]
				]
			\end{forest}}}
	\caption{Possible configurations with $\restrict{G}{\revise{\L_1}}$ and $\restrict{G}{\revise{\L_2}}$ complete}
	\label{fig:proofPredicate}
\end{figure}

This allows us to directly apply the theory of flag algebras introduced in~\cite{razborov2007} in our setting. \revise{Note that we can use this predicate to formulate the problem of the inducibility of a tree as an instance of polynomial optimization with infinitely many variables, analogously to~\cite{Raymond2017}.
	To do so, we would introduce binary variables $x_{ijk}=1\Leftrightarrow \p(i;j,k)$ for any $i,j,k$ in $\mathbb{N}$, and enforce the axioms of Theorem~\ref{th:predicate} as constraints using these variables. The objective function would need to be modelled as a limit analogously to the definition of subtree density in an increasing sequence of trees.}



\subsection{Types and flags}

Using the notations of Razborov \cite{razborov2007} (see also~\cite{oliveira2016FlagAlgebras}) adapted to our setting, a \emph{type} of size $k$ is a tree $\sigma$ of size $k$, with $\L(\sigma) =\{1, \ldots, k\}$. The empty type (of size $0$) \revise{is the empty tree, again} denoted by $\varnothing$.

An \emph{embedding} of a type $\sigma$ of size $k$ in a tree $T$ (with $k \leq |T|$) is an injective function $\theta\colon\{1, \ldots, k\} \rightarrow \L(T)$ that defines an isomorphism between $\sigma$ and the subtree of $T$ induced by $\im(\theta)$.

A $\sigma$\emph{-flag} $(T,\theta)$ is a tree $T$ together with an embedding of $\sigma$. Put simply, it is a tree whose set of leaves is partially labeled (\emph{flagged}), the labeled leaves inducing a subtree isomorphic to $\sigma$. In practice, we only specify the embedding if it is relevant to the context; the type can also be omitted when it is either obvious or irrelevant. \revise{We show two examples of flags of different types in Figure~\ref{fig:2treeflags}.} 

\begin{figure}[h!]
	\begin{subfigure}{.4\textwidth}
		\centering
		\vspace*{.4cm}
		\includegraphics[width=.22\linewidth]{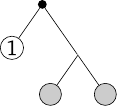}
		\vspace*{.5cm}
		\caption{A $\circled{1}$-tree-flag}
	\end{subfigure}%
	\centering
	\begin{subfigure}{.35\textwidth}
		\resizebox{.35\textwidth}{!}{%
			\centering
			\begin{forest}
				[, treeNodeRoot
						[, treeNodeInner,
							[1, treeNodeLabeled]
								[, treeNodeInner
										[, treeNode]
										[, treeNode]
								]
						]
						[, treeNodeInner
								[, treeNode]
								[, treeNodeInner
										[2, treeNodeLabeled]
										[, treeNode]
								]
						]
				]
			\end{forest}}
		\vspace*{.2cm}
		\centering
		\caption{A \begin{minipage}{0.1\textwidth}
				\centering
				\resizebox{\textwidth}{!}{%
					\begin{forest}
						[, treeNodeRoot
								[1, treeNodeLabeled]
								[2, treeNodeLabeled]
						]
					\end{forest}
				}
			\end{minipage}-tree-flag}
	\end{subfigure}
	\caption{Two tree-flags on different types}
	\label{fig:2treeflags}
\end{figure}
The automorphisms of the $\sigma$-tree-flag $(T, \theta)$ are the automorphisms of the unlabeled tree $T$ that leave the labeled leaves $\im\theta$ in place. We denote by $\Aut(T, \theta)$ the set of automorphisms of the tree-flag $(T, \theta)$.

Isomorphism between flags is the same as between trees, with the added condition that the labeling has to be preserved. More precisely, two $\mathbf{\sigma}$-tree-flags $(T, \theta)$ and $(S, \eta)$ are isomorphic if there is a graph isomorphism $\rho$ between $T$ and $S$ such that $\rho(\theta(i)) = \eta(i)$ for every leaf $i$ of the type $\sigma$. For instance, in Figure~\ref{fig:treeFlagsIso} are depicted three \begin{minipage}{0.06\textwidth}
	\centering
	\resizebox{\textwidth}{!}{%
		\begin{forest}
			[, treeNodeRoot
					[1, treeNodeLabeled]
					[, treeNodeInner
							[2, treeNodeLabeled]
							[3, treeNodeLabeled]
					]
			]
		\end{forest}
	}
\end{minipage}
\hspace*{-.1cm}-tree-flags of size $4$: \ref{fig:treeFlagIso1} and~\ref{fig:treeFlagIso2} are isomorphic to each other, but not to~\ref{fig:treeFlagIso3} or~\ref{fig:treeFlagIso4}.
\begin{figure}[h!]
	\begin{subfigure}{.2\textwidth}
		\resizebox{.45\textwidth}{!}{%
			\begin{forest}
				[, treeNodeRoot
						[1, treeNodeLabeled]
						[, treeNodeInner
								[2, treeNodeLabeled]
								[, treeNodeInner
										[3, treeNodeLabeled]
										[, treeNode]
								]
						]
				]
			\end{forest}
		}
		\centering
		\caption{}
		\label{fig:treeFlagIso1}
	\end{subfigure}%
	\centering
	\begin{subfigure}{.2\textwidth}
		\resizebox{.45\textwidth}{!}{%
			\begin{forest}
				[, treeNodeRoot
						[1, treeNodeLabeled]
						[, treeNodeInner
								[2, treeNodeLabeled]
								[, treeNodeInner
										[, treeNode]
										[3, treeNodeLabeled]
								]
						]
				]
			\end{forest}
		}\centering
		\caption{}
		\label{fig:treeFlagIso2}
	\end{subfigure}
	\begin{subfigure}{.2\textwidth}
		\resizebox{.45\textwidth}{!}{%
			\begin{forest}
				[, treeNodeRoot
						[1, treeNodeLabeled]
						[, treeNodeInner
								[, treeNode]
								[, treeNodeInner
										[2, treeNodeLabeled]
										[3, treeNodeLabeled]
								]
						]
				]
			\end{forest}
		}\centering
		\caption{}
		\label{fig:treeFlagIso3}
	\end{subfigure}
	\begin{subfigure}{.2\textwidth}
		\resizebox{.5\textwidth}{!}{%
			\begin{forest}
				[, treeNodeRoot
						[, treeNodeInner
								[1, treeNodeLabeled]
								[, treeNode]
						]
						[, treeNodeInner
								[2, treeNodeLabeled]
								[3, treeNodeLabeled]
						]
				]
			\end{forest}
		}\centering
		\caption{}
		\label{fig:treeFlagIso4}
	\end{subfigure}
	\caption{Four \begin{minipage}{0.06\textwidth}
			\centering
			\resizebox{\textwidth}{!}{%
				\begin{forest}
					[, treeNodeRoot
							[1, treeNodeLabeled]
							[, treeNodeInner
									[2, treeNodeLabeled]
									[3, treeNodeLabeled]
							]
					]
				\end{forest}
			}
		\end{minipage}\hspace{-.1cm}-tree-flags of size $4$}
	\label{fig:treeFlagsIso}
\end{figure}

The size $|(T,\theta)|$ of a flag we still define as the number of leaves $|\L(T)|$ of the underlying tree $T$. We denote by $\trees_n^\sigma$ the set of all $\sigma$-tree-flags of size $n$ (up to isomorphism), and by $\trees^\sigma$ the set of all $\sigma$-tree-flags (up to isomorphism). We say that two $\sigma$-tree-flags $(T,\theta)$ and $(S, \eta)$ are \emph{disjoint} if the sets of their leaves differ outside the embedding of $\sigma$, i.e.\ when $(\L(T)\setminus \im\theta)\cap (\L(S)\setminus \im\eta) = \emptyset$. Subflags we define analogously to subtrees by $\revise{\subgraph{(T,\theta)}{\V}} \coloneqq (\revise{\subgraph{T}{\V}}, \theta)$, under the condition that the subset of leaves contains the labeled leaves $\im\theta \subseteq \V\subseteq \L(T)$. \revise{We identify the set of trees $\trees$ with the set of $\varnothing$-tree-flags $\trees^\varnothing$.}


\subsection{Densities of flags}
\label{ssec:densitiesFlags}

We can extend the notion of density in trees to flags. We define the density of a $\sigma$-tree-flag $(S, \theta)$ in another $\sigma$-tree-flag $(T, \eta)$ as the probability of obtaining a tree-flag isomorphic to $(S, \theta)$ when choosing uniformly at random a $\mathbf{\sigma}$-subtree-flag of size $|S|$ in $T$. i.e.\ it is the probability
\begin{align*}
	p((S, \theta);(T, \eta)) & \coloneqq \mathbb{P}[\revise{\subgraph{(T,\eta)}{\V\cup\im\eta}} \cong (S,\theta)] \in [0,1],
\end{align*}
where $\V$ is a random subset of $\L(T) \setminus \im(\eta)$ of size $|S|-|\sigma|$.

Densities of flags follow the \emph{chain rule} in Lemma~\ref{lemma:chain_rule}.
\begin{lemma}[Chain rule for flags~\cite{razborov2007}]\label{lemma:chain_rule}
	For $S,T\in\trees^\sigma$ and an integer $n$ such that $\abs{S} \leq n \leq \abs{T}$ we have
	\begin{align}
		p(S; T) & = \sum\limits_{S' \in \trees^{\sigma}_n} p(S; S')p(S'; T).
	\end{align}
\end{lemma}


Analogously to the unlabeled case, we define the density of a flag $(S,\theta)$ in an increasing sequence of flags $\T=  (T_\revise{n}, \eta_\revise{n})_{\revise{n}\geq 1}$ ($(|T_n|)_{n \geq 0}$ is strictly increasing) as
\begin{align*}
	\reviseSecond{\psi}_{\T}(S, \theta) & = \lim\limits_{n\rightarrow\infty} p((S, \theta);(T_n, \eta_n))\in [0,1].
\end{align*}
\reviseSecond{To differentiate flag densities from (unlabeled) densities, we denote them by $\psi$ instead of $\phi$. If $\T$ is a sequence of flags, but we write $\phi_\T$, then we use $\T$ for the corresponding sequence of (unlabeled) trees $(T_n)_{n\geq 1}$.}
As before, we assume the limits exist; we are working with \emph{converging sequences of flags}, again obtainable by choosing appropriate subsequences of flags.

We can linearly extend densities of tree-flags to formal (real) linear combinations of tree-flags
, which we call \emph{quantum trees}, following Lov\'asz's quantum graphs~\cite{Lovasz2012}. \revise{We denote by $\R\trees^\sigma$ the space of all quantum trees, i.e., the space of all formal real linear combinations of tree flags.}

\subsection{Products of flags}

Our goal is now to understand products of densities of flags. Given two $\sigma$-flags $S_1, S_2\in \trees^\sigma$ we want to find a quantum tree $S_1\cdot S_2\in\R\trees^\sigma$ such that
\begin{equation*}
	\reviseSecond{\psi}_{\T}(S_1)\reviseSecond{\psi}_{\T}(S_2) = \reviseSecond{\psi}_{\T}(S_1\cdot S_2)
\end{equation*}
for all increasing sequences of flags $\T$.

To compute these, we need to define the \emph{sunflower density} of two flags $S_1, S_2\in \trees^\sigma$ in another flag $T=(t,\theta)\in\trees^\sigma$. It is given by
\begin{align*}
	p(S_1,S_2; T) \coloneqq \mathbb{P}[\revise{\subgraph{T}{\V_1}} \cong S_1 \wedge \revise{\subgraph{T}{\V_2}} \cong S_2] \in [0,1],
\end{align*}
where $\V_1, \V_2\in \mathcal{P}(\L(T))$ is a uniformly random \emph{sunflower} with center $\V_1\cap\V_2=\im\theta$ and petals of sizes $\abs{\V_i\setminus \im \theta} = \abs{S_i}-\abs{\sigma}$. It is the probability that $S_1$ and $S_2$ fit in $T$ simultaneously, matching only on the type $\sigma$.

The chain rule in Lemma~\ref{lemma:chain_rule} can be generalized to compute the sunflower density of several tree-flags. For every $\abs{S_1}+ \abs{S_2} - \abs{\sigma} \leq n \leq \abs{T}$, the identity
\begin{align}\label{eq:chain_rule_generalized}
	p(S_1, S_2; T) & = \sum\limits_{\revise{S'} \in \trees^{\sigma}_n} p(S_1, S_2; \revise{S'})p(\revise{S'}; T)
\end{align}
holds~(Lemma 2.2 in \cite{razborov2007}).

The sunflower densities of flags exactly describe the products of densities in the limit, captured by the following Theorem~\ref{th:productlim} (adapted from Theorem $2$ in~\cite{oliveira2016FlagAlgebras}, see also Lemma 2.3 in~\cite{razborov2007}):

\begin{theorem}
	\label{th:productlim}
	If $S_1$ and $S_2$ are fixed $\sigma$-tree-flags, then for any $\sigma$-tree-flag $T$ such that $S_1, S_2$ fit in $T$,
	\begin{equation*}
		p(S_1,S_2;T) = p(S_1;T)p(S_2;T) + \O\left(1/\abs{T}\right).
	\end{equation*}
\end{theorem}
Indeed, the probability that two (independently) random $\sigma$-subtree-flags in $T$ of sizes $\abs{S_1}$ and $\abs{S_2}$ are disjoint approaches $1$ as $T$ gets larger.


Theorem~\ref{th:productlim} tells us that the product $\reviseSecond{\psi}_{\T}(S_1)\reviseSecond{\psi}_{\T}(S_2)$ of densities of flags in an increasing sequence $\T$ behaves, asymptotically, like the sunflower density of $S_1$ and $S_2$. And the chain rule~\eqref{eq:chain_rule_generalized} tells us how to compute it:
\begin{align*}
	\reviseSecond{\psi}_{\T}(S_1)\reviseSecond{\psi}_{\T}(S_2) =  \sum\limits_{T \in \trees_n^\sigma} p(S_1,S_2;T)\reviseSecond{\psi}_{\T}(T)
\end{align*}
for any $n\geq \abs{S_1} + \abs{S_2}- \abs{\sigma}$.

Thus, we find a natural way to define the \emph{gluing product} of tree-flags as
\begin{align*}
	S_1\cdot S_2 \coloneqq \sum\limits_{T \in \trees_n^\sigma} p(S_1,S_2;T) T \in \R\trees^\sigma.
\end{align*}
Note that\revise{, when $n=\abs{S_1} + \abs{S_2}- \abs{\sigma}$,} every tree-flag $T$ appearing in the sum with nonzero coefficients can be obtained by ``gluing'' the leaves of $S_1$ and $S_2$ with same labels on top of each other, and sending the unlabeled leaves to distinct leaves of $T$.

This product is not yet entirely well-defined; it depends on the choice of $n$. To solve this, we define the \emph{flag algebra of trees of type $\sigma$} as the quotient
\begin{equation*}
	\A^\sigma \coloneqq \sfrac{\R\trees^\sigma}{\mathcal{K}^\sigma},
\end{equation*}
where $\mathcal{K}^\sigma$ is the linear span of all elements
\begin{equation}\label{eq:quotient}
	S - \sum\limits_{S' \in \trees^{\sigma}_n} p(S; S')S',
\end{equation}
where $S\in\trees^\sigma$ and $n\geq \abs{S}$. The elements of $\mathcal{\sigma}$ are exactly the zeroes implied by the chain rule Lemma~\ref{lemma:chain_rule}. Quotienting them out turns $\A^\sigma$ into an algebra and the product $S_1\cdot S_2$ well-defined (Lemma 2.4 in \cite{razborov2007}).
\reviseSecond{
	Observe that the density functions associated with flags and trees define
	\emph{positive homomorphisms} from the corresponding flag algebras to $\mathbb{R}$.
	Specifically, for every convergent sequence $\T$ of $\sigma$-flags,
	\begin{equation*}
		\psi_\T \in \mathrm{Hom}^+(\A^\sigma, \mathbb{R})
		\quad \text{and} \quad
		\phi_\T \in \mathrm{Hom}^+(\A^\varnothing, \mathbb{R}),
	\end{equation*}
	because the density operators are multiplicative and nonnegative on $\trees^\sigma$.
}

We say that a quantum tree $T\in\A^\sigma$ is \emph{nonnegative}, denoted by $T\geq 0$, if
\begin{equation*}
	\reviseSecond{\psi}_{\T}(T) \geq 0
\end{equation*}
for all increasing sequences $\T$.





To illustrate this, we express the square of the following $\circled{1}$-tree-flag of size $3$ \begin{minipage}{0.05\textwidth}
	\centering
	\includegraphics[width=\linewidth]{Pictures/LatexPic-21.pdf}
\end{minipage} \revisedelete{in function of $\circled{1}$-tree-flags of size $5$ (the minimum size possible)}. We can construct the tree-flags that will appear in the product by the gluing operation described above, and obtain the three following tree-flags:
\begin{minipage}{0.06\textwidth}
	\centering
	\includegraphics[width=\linewidth]{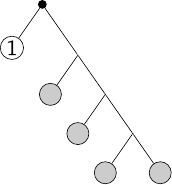}
\end{minipage},
\begin{minipage}{0.06\textwidth}
	\centering
	\includegraphics[width=\linewidth]{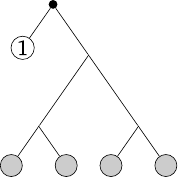}
\end{minipage}
and \begin{minipage}{0.06\textwidth}
	\centering
	\includegraphics[width=\linewidth]{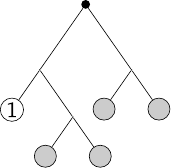}
\end{minipage}.

\vspace*{.2cm}

It is straightforward to check that they all can simultaneously contain $2$ disjoint subtree-flags isomorphic to \begin{minipage}{0.05\textwidth}
	\centering
	\includegraphics[width=\linewidth]{Pictures/LatexPic-21.pdf}
\end{minipage}, and that no other $\circled{1}$-tree-flag of size $5$ does.

\vspace{.3cm}
To determine the coefficient of each tree-flag in the product, we compute the probability of two randomly chosen disjoint $\circled{1}$-subtree-flags of size $3$ in this tree-flag to be both isomorphic to  \begin{minipage}{0.05\textwidth}
	\centering
	\includegraphics[width=\linewidth]{Pictures/LatexPic-21.pdf}
\end{minipage}.

All the possible ways (up to isomorphism) of obtaining two disjoint subtree-flags isomorphic to \ \begin{minipage}{0.05\textwidth}
	\centering
	\includegraphics[width=\linewidth]{Pictures/LatexPic-21.pdf}
\end{minipage} in each tree present in the product are represented in Figure~\ref{fig:treeProducts}. In each figure one subtree is represented in red and the other one in blue.

\begin{figure}[h!]
	\captionsetup[subfigure]{justification=centering}
	\centering
	\begin{subfigure}{.18\textwidth}
		\hspace*{.4cm}
		\resizebox{.55\textwidth}{!}{%
			\begin{forest}
				[, treeNodeRoot
				[1, treeNodeLabeled, edge=treeEdgeDouble]
				[, treeNodeInner, edge=treeEdgeDouble
				[, treeNodeMarked, edge=treeEdgeMarked]
				[, treeNodeInner, edge=treeEdgeDouble
				[, treeNodeMarked, edge=treeEdgeMarked]
				[, treeNodeInner, edge=treeEdgeDouble
				[, treeNodeBlue, edge=treeEdgeBlue]
				[, treeNodeBlue, edge=treeEdgeBlue]
				]
				]
				]
				]
			\end{forest}}
		\caption{}
		\label{fig:treeProd1}
	\end{subfigure}
	\centering
	\begin{subfigure}{.18\textwidth}
		\hspace*{.3cm}
		\resizebox{.55\textwidth}{!}{%
			\begin{forest}
				[, treeNodeRoot
				[1, treeNodeLabeled, edge=treeEdgeDouble]
				[, treeNodeInner, edge=treeEdgeDouble
				[, treeNodeMarked, edge=treeEdgeMarked]
				[, treeNodeInner, edge=treeEdgeDouble
				[, treeNodeBlue, edge=treeEdgeBlue]
				[, treeNodeInner, edge=treeEdgeDouble
				[, treeNodeMarked, edge=treeEdgeMarked]
				[, treeNodeBlue, edge=treeEdgeBlue]
				]
				]
				]
				]
			\end{forest}}
		\caption{}
		\label{fig:treeProd2}
	\end{subfigure}
	\centering
	\begin{subfigure}{.18\textwidth}
		\hspace*{.4cm}
		\resizebox{.55\textwidth}{!}{%
			\begin{forest}
				[, treeNodeRoot
				[1, treeNodeLabeled, edge=treeEdgeDouble]
				[, treeNodeInner, edge=treeEdgeDouble
				[, treeNodeInner, edge=treeEdgeMarked
				[, treeNodeMarked, edge=treeEdgeMarked]
				[, treeNodeMarked, edge=treeEdgeMarked]
				]
				[, treeNodeInner, edge=treeEdgeBlue
				[, treeNodeBlue, edge=treeEdgeBlue]
				[, treeNodeBlue, edge=treeEdgeBlue]
				]
				]
				]
			\end{forest}}
		\caption{}
		\label{fig:treeProd3}
	\end{subfigure}
	\begin{subfigure}{.18\textwidth}
		\hspace*{.4cm}
		\resizebox{.55\textwidth}{!}{%
			\begin{forest}
				[, treeNodeRoot
				[1, treeNodeLabeled, edge=treeEdgeDouble]
				[, treeNodeInner, edge=treeEdgeDouble
				[, treeNodeInner, edge=treeEdgeDouble
				[, treeNodeMarked, edge=treeEdgeMarked]
				[, treeNodeBlue, edge=treeEdgeBlue]
				]
				[, treeNodeInner, edge=treeEdgeDouble
				[, treeNodeMarked, edge=treeEdgeMarked]
				[, treeNodeBlue, edge=treeEdgeBlue]
				]
				]
				]
			\end{forest}}
		\caption{}
		\label{fig:treeProd4}
	\end{subfigure}
	\begin{subfigure}{.18\textwidth}
		\hspace*{.4cm}
		\resizebox{.55\textwidth}{!}{%
			\begin{forest}
				[, treeNodeRoot
				[, treeNodeInner, edge=treeEdgeBlue
				[1, treeNodeLabeled, edge=treeEdgeDouble]
				[, treeNodeInner, edge=treeEdgeMarked
				[, treeNodeMarked, edge=treeEdgeMarked]
				[, treeNodeMarked, edge=treeEdgeMarked]
				]
				]
				[, treeNodeInner, edge=treeEdgeBlue
				[, treeNodeBlue, edge=treeEdgeBlue]
				[, treeNodeBlue, edge=treeEdgeBlue]
				]
				]
			\end{forest}}
		\caption{}
		\label{fig:treeProd5}
	\end{subfigure}
	\caption{Every way (up to isomorphism) to obtain two disjoint subtree-flags\\ isomorphic to \begin{minipage}{0.04\textwidth}
			\centering
			\includegraphics[width=\linewidth]{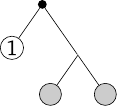}
		\end{minipage} in trees of size $5$ }
	\label{fig:treeProducts}
\end{figure}

We then see that, for \begin{minipage}{0.06\textwidth}
	\centering
	\includegraphics[width=\linewidth]{Pictures/LatexPic-22.pdf}
\end{minipage}
and\begin{minipage}{0.06\textwidth}
	\centering
	\includegraphics[width=\linewidth]{Pictures/LatexPic-23.pdf}
\end{minipage},
this probability is equal to $1$: any $\circled{1}$-subtree-flag of size $3$ in these tree-flags is isomorphic to \begin{minipage}{0.05\textwidth}
	\centering
	\includegraphics[width=\linewidth]{Pictures/LatexPic-21.pdf}
\end{minipage}.
For \begin{minipage}{0.06\textwidth}
	\centering
	\includegraphics[width=\linewidth]{Pictures/LatexPic-24.pdf}
\end{minipage},
this probability is equal to~$\frac{1}{3}$. Indeed, the only  $\circled{1}$-subtree-flags isomorphic to \begin{minipage}{0.05\textwidth}
	\centering
	\includegraphics[width=\linewidth]{Pictures/LatexPic-21.pdf}
\end{minipage} in this tree are induced by pairs of leaves of the same height, who account for $\frac{1}{3}$ of all the pairs of leaves.

Finally, we obtain the following expression for the square of \ \begin{minipage}{0.05\textwidth}
	\centering
	\includegraphics[width=\linewidth]{Pictures/LatexPic-21.pdf}
\end{minipage}:

\begin{equation*}
	\begin{minipage}{0.2\textwidth}
		\centering
		\includegraphics[width=.35\linewidth]{Pictures/LatexPic-21.pdf}
	\end{minipage}
	\hspace{-0.06\textwidth}
	\cdot
	\hspace{-.04\textwidth}
	\begin{minipage}{0.2\textwidth}
		\centering
		\includegraphics[width=.35\linewidth]{Pictures/LatexPic-21.pdf}
	\end{minipage}
	\hspace{-0.05\textwidth}
	=
	\begin{minipage}{0.13\textwidth}
		\centering
		\includegraphics[width=.75\linewidth]{Pictures/LatexPic-22.pdf}
	\end{minipage}
	+ \hspace{-0.025\textwidth}
	\begin{minipage}{0.17\textwidth}
		\centering
		\includegraphics[width=.65\linewidth]{Pictures/LatexPic-23.pdf}
	\end{minipage}
	\hspace{-.05\textwidth}
	+ \ \frac{1}{3}
	\hspace{-.01\textwidth}
	\begin{minipage}{0.17\textwidth}
		\centering
		\includegraphics[width=.65\linewidth]{Pictures/LatexPic-24.pdf}
	\end{minipage}
\end{equation*}






\subsection{Downward operator }

To go back from the space of linear combinations of tree-flags to the set of quantum trees, we use the \emph{downward operator} adapted to rooted binary trees.
The downward operator~$\llbracket \cdot\rrbracket$ averages flags over all choices of labels: it unlabels the flags. For a $\sigma$-tree-flag $T$, we denote \revise{by} $T|_0$ the tree obtained simply by forgetting the labels of $T$. We then have
\begin{align*}\llbracket T\rrbracket_\sigma= q_\sigma(T)\cdot T|_0,
\end{align*} where \begin{align*}
	q_\sigma(T) & = \frac{(n-\abs{\sigma})!}{n!}\cdot \frac{\abs{\Aut(T|_0)}}{\abs{\Aut(T)}}
\end{align*}
is a normalizing factor equal to the probability that a random injective map $\theta:V(\sigma) \rightarrow V(T)$ is such that $(T|_0,\theta)$ is a $\sigma$-tree-flag isomorphic to $T$. We can extend this operator to quantum graphs: this then provides us with a linear map from $\A^{\sigma}$ to the space of (unlabeled) \revise{quantum} trees $\revise{\R\trees}$. This is exceedingly useful, as Theorem~\ref{th:downwards} makes it possible to prove statements in $\A^\varnothing$ using true statements in $\A^\sigma$.

\begin{theorem}[Theorem 3.1 in \cite{razborov2007}]\label{th:downwards}
	Let $T\in\A^\sigma$ be a nonnegative quantum tree, i.e.\ $\reviseSecond{\psi}_\T(T)\geq 0$ for all increasing sequences $\T$. Then $\llbracket T\rrbracket \geq 0$ is also nonnegative.
\end{theorem}

We give a few examples of applications of the downwards operator below.
\vspace{-.75cm}
\begin{center}
	\begin{equation*}
		\Bigl\llbracket
		\hspace{-2.5cm}
		\begin{minipage}{0.49\textwidth}
			\centering
			\resizebox{.15\textwidth}{!}{%
				\begin{forest}
					[, treeNodeRoot
							[1, treeNodeLabeled]
							[, treeNodeInner
									[, treeNode]
									[, treeNode]
							]
					]
				\end{forest}
			}
		\end{minipage}
		\hspace{-2.5cm}
		\Bigr\rrbracket
		=  \frac{1}{3}
		\hspace{-2.2cm}
		\begin{minipage}{0.49\textwidth}
			\centering
			\includegraphics[width=.15\linewidth]{Pictures/Tree-4.pdf}
		\end{minipage}
	\end{equation*}
	\begin{equation*}
		\Bigl\llbracket
		\hspace{-2.5cm}
		\begin{minipage}{0.49\textwidth}
			\centering
			\resizebox{.15\textwidth}{!}{%
				\begin{forest}
					[, treeNodeRoot
							[, treeNode]
							[, treeNodeInner
									[1, treeNodeLabeled]
									[, treeNode]
							]
					]
				\end{forest}
			}
		\end{minipage}
		\hspace{-2.5cm}
		\Bigr\rrbracket
		=  \frac{2}{3}
		\hspace{-2.2cm}
		\begin{minipage}{0.49\textwidth}
			\centering
			\includegraphics[width=.15\linewidth]{Pictures/Tree-4.pdf}
		\end{minipage}
	\end{equation*}
	\vspace*{-.2cm}
	\begin{equation*}
		\Bigl\llbracket \ \Bigl(
		\hspace{-1cm}
		\begin{minipage}{0.25\textwidth}
			\centering
			\resizebox{.3\textwidth}{!}{%
				\begin{forest}
					[, treeNodeRoot
							[1, treeNodeLabeled]
							[, treeNodeInner
									[, treeNode]
									[, treeNode]
							]
					]
				\end{forest}
			}
		\end{minipage}
		\hspace{-.9cm}
		\Bigr)^2 \  \Bigr\rrbracket
		=
		\ \frac{1}{5}
		\hspace{-.45cm}
		\begin{minipage}{0.18\textwidth}
			\centering
			\includegraphics[width=.55\linewidth]{Pictures/Tree-8.pdf}
		\end{minipage}
		\hspace{-.8cm}
		+ \ \frac{1}{5}
		\hspace{-.45cm}
		\begin{minipage}{0.18\textwidth}
			\centering
			\includegraphics[width=.6\linewidth]{Pictures/Tree-9.pdf}
		\end{minipage}
		\hspace{-.05\textwidth}
		+ \ \frac{1}{15} \
		\begin{minipage}{0.1\textwidth}
			\centering
			\includegraphics[width=1.2\linewidth]{Pictures/Tree-7.pdf}
		\end{minipage}
	\end{equation*}
\end{center}

All averaged products of trees-flags needed for up to the $6$th level of the hierarchy (defined in the next section) are given in Appendix~\ref{app:products}.

\section{Flag sums of squares and semidefinite programming}
\label{sec:computation}

We now explain how we use and implement the flag algebra of trees to obtain computer-assisted bounds on inducibilities of trees.

Quantum trees $S\in \A^\sigma$ correspond to functions which send tree sequences $\T$ to real numbers $\reviseSecond{\psi}_\T(S)$. Thus, analogously to polynomial optimization~\cite{Laurent2008a}, we can use the sums of squares method to compute bounds on extremal problems in trees. The idea is simple: If $f = c_1T_1 + \cdots + c_n T_n$ is a quantum flag in the algebra $\A^{\sigma}$, then both its square $f^2$ and, more importantly, its unlabeled square $\llbracket f^2\rrbracket$ are nonnegative functions on tree sequences (see Theorem 3.14 in \cite{razborov2007}). Thus, every squared and unlabeled quantum flag proves an inequality of the form $\llbracket f^2\rrbracket\geq 0$ in $\A^\varnothing$.

Let $\mathcal{F}_\sigma$ be a vector of $n$ tree-flags of type $\sigma$.
Quantum flags involving the flags of $\mathcal{F}_{\sigma}$ are of the form
\begin{equation*}
	f = c^\top \mathcal{F}_\sigma,
\end{equation*}
where $c\in\revisedeletemath{\R^{\mathcal{F}_\sigma}} \revise{\R^n}$ is the vector of coefficients of the quantum flag.
Unlabeled squares of quantum flags can be written as
\begin{equation*}
	\llbracket f^2 \rrbracket = \llbracket (c^\top\mathcal{F}_\sigma)^2 \rrbracket = \llbracket\langle cc^\top, \mathcal{F}_\sigma\mathcal{F}_\sigma^\top \rangle\rrbracket=\langle cc^\top, \llbracket\mathcal{F}_\sigma\mathcal{F}_\sigma^\top \rrbracket\rangle,
\end{equation*}
where
\begin{equation*}
	\langle A,B\rangle \coloneqq \tr(A^\top B) = \sum_{i,j=1}^n A_{ij}B_{ij}
\end{equation*}
denotes the \emph{trace inner product} of (symmetric) matrices $A,B\in S^n$.

Here, $cc^\top$ is a rank-one \emph{positive semidefinite matrix}, i.e.\ $v^\top (cc^\top) v = (c^\top v)^2\geq 0$ for all vectors $v\in \R^n$. Positive semidefinite matrices form a convex cone $S^n_{\succcurlyeq 0}$. We denote $X\in S^n_{\succcurlyeq 0}$ by $X\succcurlyeq 0$. Hence, flag sums of squares $\llbracket\sum_{i=1}^{k}f_i^2\rrbracket$, where the $f_i$ are quantum flags in $\A^\sigma$, are of the form
\begin{equation*}
	\langle M_\sigma, \llbracket\mathcal{F}_\sigma\mathcal{F}_\sigma^\top \rrbracket\rangle
\end{equation*}
where $M_\sigma\in S^n_{\succcurlyeq 0}$ is a positive semidefinite matrix.

Of course, we can combine sums of squares coming from algebras $\A^\sigma$ of different types $\sigma$. In general, we can work with a family $(\mathcal{F}_\sigma)_\sigma$, where $\mathcal{F}_\sigma$ $\revisedeletemath{\subseteq \A^\sigma}$ are vectors of flags \revise{of type $\sigma$}. We call
\begin{equation*}
	f = \sum_\sigma \langle M_\sigma, \llbracket\mathcal{F}_\sigma\mathcal{F}_\sigma^\top \rrbracket\rangle,
\end{equation*}
where the $M_\sigma$ are positive semidefinite matrices of appropriate sizes, a \emph{sum of squares certificate} for the nonnegativity of a quantum tree $f\in\A^\varnothing$.

\subsection{A hierarchy of SDPs }

Comparing the coefficients of a quantum flag $f\in\A^\varnothing$ with the coefficients in the sum of squares $\sum_\sigma \langle M_\sigma, \llbracket\mathcal{F}_\sigma\mathcal{F}_\sigma^\top \rrbracket\rangle$ (up to quotienting out $\mathcal{K}^\varnothing$) leads to linear constraints on the coefficients of the positive semidefinite matrices $M_\sigma$. To implement the quotient algebra $\A^\varnothing\coloneqq \sfrac{\R\revise{\trees^\varnothing}}{\mathcal{K}^\varnothing}$ we can add free variables corresponding to a basis of $\mathcal{K}^\varnothing$.

What remains to decide is which vectors of flags $\mathcal{F}_\sigma$ to use. We chose to implement the analogue of the SDP hierarchy used by the software Flagmatic, as it is described in Section 2.3 of \cite{FalgasRavry2012a}. For each natural number $L\in \mathbb{N}$, we define a level $L$ of the hierarchy, which we denote by $\mathrm{SOS}_L\subseteq \A^\varnothing$. For level $L$ of the hierarchy, we consider flags of types $\sigma$ with $|\sigma| \leq L$ and $|\sigma|\equiv L\mod 2$. We then form the vector $\mathcal{F}_\sigma$ to consist of all $\sigma$-flags with exactly $(L-|\sigma|)/2$ unlabeled leaves up to (label preserving) isomorphism. This way, all products of flags in each $\mathcal{F}_\sigma\mathcal{F}_\sigma^\top$ result in a tree with exactly $L$ vertices. We would not gain anything here by considering types with $|\sigma| + 1\equiv L\mod 2$, or flags with fewer unlabeled leaves due to the quotient relations \eqref{eq:quotient}.

With this choice of types and vectors of flags, we define the $L$th level of the hierarchy as
\begin{equation*}
	\mathrm{SOS}_L \coloneqq \left\lbrace\sum_\sigma \langle M_\sigma, \llbracket\mathcal{F}_\sigma\mathcal{F}_\sigma^\top \rrbracket\rangle \mid M_\sigma \succcurlyeq 0\right\rbrace \subseteq \A^\varnothing.
\end{equation*}
We give a list of sizes of the semidefinite blocks of the hierarchy for the first few levels in Table~\ref{tab:blockSizes}.

\begin{table}[ht!]
	\centering
	\begin{tabular}{@{}ccc@{}}
    \toprule
    Level & Block sizes  & Sum \\\midrule
    $4$   & $3_{1}1_{3}$ & 6   \\$5$ & $5_{1}2_{1}1_{3}$ & 10\\$6$ & $9_{1}7_{2}1_{7}$ & 30\\$7$ & $20_{1}9_{3}4_{1}1_{11}$ & 62\\$8$ & $35_{2}25_{1}11_{6}2_{1}1_{23}$ & 186\\$9$ & $70_{1}54_{3}13_{11}9_{1}1_{46}$ & 430\\$10$ & $147_{2}77_{6}69_{1}15_{23}3_{1}1_{98}$ & 1271\\$11$ & $264_{3}230_{1}104_{11}20_{1}17_{46}1_{207}$ & 3175\\\bottomrule
\end{tabular}

	\caption{The block sizes of the hierarchy $\mathrm{SOS}_\text{Level}$. They are given in the form $(\text{size of block})_{\text{multiplicity of block}}$.}
	\label{tab:blockSizes}
\end{table}


\subsection{Computational approach}\label{sec:comp}

We implemented the flag algebra of trees as part of the Julia package \texttt{FlagSOS.jl}\footnote{\url{https://github.com/DanielBrosch/FlagSOS.jl}} introduced in the thesis~\cite{brosch2022symmetry}. All code to recover the results of the paper is available as ancillary file.

\paragraph{Computing products of tree flags.} We compute all relevant products between flags simultaneously in a preprocessing step: We first generate all trees up to isomorphism (up to \revisedelete{a $T$ number of} \revise{$L$} leaves, where $L$ is the level of the hierarchy). For each such tree \revise{$T$} we then compute all pairs of (possibly overlapping) subtrees \revisedelete{$(t_1,t_2)$} \revise{$(S_1,S_2)$} of \revise{$T$} such that their union of leaves is the set of leaves of \revise{$T$}. Adding labels to the vertices in the overlap gives us a contribution to the product of the \revise{corresponding} flags \revisedelete{$t_1\cdot t_2$}.

\paragraph{Rounding.} The matrices $M_\sigma$ returned by the solver are only approximately positive semi-definite, i.e.\ the matrices may have eigenvalues slightly below zero. To work around this, we do the following: We (numerically) compute the eigenvalues and eigenvectors of the \revise{matrices} $M_\sigma$, and fix the negative eigenvalues to zero. We then round the eigenvectors, scaled by the square roots of the positive eigenvalues, to rationals, to compute a decomposition $M_\sigma \approx V_\sigma V_\sigma^\top\succcurlyeq 0$, where $V_\sigma$ is a rational rectangular matrix. As $V_\sigma V_\sigma^\top$ is positive semidefinite, we obtain a rational certificate:
\begin{equation*}
	\sum_\sigma \langle M_\sigma, \llbracket\mathcal{F}_\sigma\mathcal{F}_\sigma^\top \rrbracket\rangle + f_{\text{err}} = \sum_\sigma \langle V_\sigma V_\sigma^\top, \llbracket\mathcal{F}_\sigma\mathcal{F}_\sigma^\top \rrbracket\rangle \geq 0.
\end{equation*}
Here $f_{\text{err}} = \sum_T c_T T\in \A^\varnothing$ is an error term resulting from the rounding. By definition of trees $T$ as density functionals, we know \revise{each} $T$ takes values in the interval $[0,1]$. This way we can bound the error term to take values in the interval
\begin{equation*}
	-\sum_t |c_T| \leq f_{\text{err}} \leq \sum_t |c_T|.
\end{equation*}
\revisedelete{In practice, this error is of order $10^{-6}$.} We add this error with the appropriate sign to computed bounds to obtain rigorous bounds.

\paragraph{Solvers.}
We \revise{computed all numerical bounds} in this paper with Mosek~\cite{mosek} on a server equipped with an AMD EPYC 7532 32-Core Processor @ 3.30GHz and 1024GB of RAM. The loss in objective after applying the above rounding procedure was, in all cases, of order $10^{-6}$. \revise{The exact bounds in Section~\ref{sec:InducExact} were computed using the ClusteredLowRankSolver~\cite{leijenhorstSolvingClusteredLowrank2024}.}

\section{Results}
\label{sec:results}

In this section, we show and explain some of the more interesting results we obtained thanks to the flag algebra of trees on the inducibility of trees and on \revisedelete{their} tree-profiles.


\subsection{Inducibility}
\label{sec:resultsInd}

We can build an SDP based hierarchy for approximating the inducibility of a tree $S$ the following way:
\begin{align*}
	I(S) \coloneqq \max_{\mathcal{T}} \phi_{\mathcal{T}}(S) = \max S = \min \{t \mid t-S \geq 0\} \leq \min \{t \mid t-S \in \mathrm{SOS}_L\} \eqqcolon I_L(S),
\end{align*}
where $L\geq |S|$ is the level of the hierarchy. We are able to compute $I_{11}(S)$ numerically for all trees with up to $11$ leaves in at most $10$ seconds each.
All these results are given in Appendix~\ref{app:bounds}.
We present here our results for some selected trees. \revise{We start by giving a full example, the inducibility of the even tree $E_5$. We detail how one can use known optimal constructions (i.e.\ tree limits attaining the inducibility) to find the exact certificate. Then, we give other exact results we managed to obtain in a similar fashion, recovering a known inducibility proven in~\cite{wagner2017} and determining two new inducibilities. Finally, we detail the progression of the bounds for three small trees where the bound does not appear to have converged yet. These include a tree with \reviseSecond{known strong, but not tight,} upper and lower bounds, obtained in~\cite{DossouOlory2018}.}
\revisedelete{all the non-trivial trees up to size $6$, the ones with a known inducibility proven in~\cite{wagner2017}; and the one with known tight upper and lower bounds, obtained in~\cite{DossouOlory2018}, for which we detail the evolution with the levels of the hierarchy.}

\subsubsection{Example: The inducibility of $E_5$ and a phantom branch}\label{sec:inducibilityEFive}
To both give a full SDP-example and some of the ideas and consequences of the exact solutions of the primal and dual problem, we detail here how we recover the exact inducibility of the balanced tree with five leaves $E_5$: $I(E_5) = \frac{2}{3}$, first proven in~\cite{wagner2017}. A sharp asymptotic construction is the \emph{infinite even tree} $E_\infty = (E_1, E_2, E_3, E_4,\ldots)$.
\newcommand{\TIrr}{T}
\newcommand{\TOne}{\textcircled{\small 1}}
\newcommand{\FA}{
	\mathcal{F}_{\small\TOne}
}
\newcommand{\FAT}{
	\mathcal{F}_{\small\TOne}^\top
}
\newcommand{\TThree}{\begin{minipage}{0.04\textwidth}
		\centering
		\resizebox{\textwidth}{!}{%
			\begin{forest}
				[, treeNodeRoot
						[1, treeNodeLabeled]
						[, treeNodeInner
								[2, treeNodeLabeled],
							[3, treeNodeLabeled]
						]
				]
			\end{forest}
		}
	\end{minipage}}
\newcommand{\FB}{
	\mathcal{F}_{\TThree}
}

We here need the fifth level of the hierarchy to obtain the sharp bound: $I_5(E_5) = I(E_5)$. The flags we need are \vspace*{-.2cm}
\begin{align*}
	\FA           & =\left(
	\begin{minipage}{0.05\textwidth}
		\centering
		\resizebox{\textwidth}{!}{%
			\begin{forest}
				[, treeNodeRoot
						[, treeNode]
						[, treeNodeInner
								[1, treeNodeLabeled],
							[, treeNode]
						]
				]
			\end{forest}
		}
	\end{minipage},
	\begin{minipage}{0.05\textwidth}
		\centering
		\resizebox{\textwidth}{!}{%
			\begin{forest}
				[, treeNodeRoot
						[1, treeNodeLabeled]
						[, treeNodeInner
								[, treeNode],
							[, treeNode]
						]
				]
			\end{forest}
		}
	\end{minipage}
	\right)^\top, &
	\FB           & = \left(
	\begin{minipage}{0.07\textwidth}
		\centering
		\resizebox{\textwidth}{!}{%
			\begin{forest}
				[, treeNodeRoot
						[1, treeNodeLabeled]
						[, treeNodeInner
								[3, treeNodeLabeled],
							[, treeNodeInner
										[2, treeNodeLabeled],
									[, treeNode],
								]
						]
				]
			\end{forest}
		}
	\end{minipage},
	\begin{minipage}{0.07\textwidth}
		\centering
		\resizebox{\textwidth}{!}{%
			\begin{forest}
				[, treeNodeRoot
						[, treeNodeInner
								[1, treeNodeLabeled]
								[, treeNode],
						]
						[, treeNodeInner
								[2, treeNodeLabeled],
							[3, treeNodeLabeled],
						]
				]
			\end{forest}
		}
	\end{minipage},
	\begin{minipage}{0.07\textwidth}
		\centering
		\resizebox{\textwidth}{!}{%
			\begin{forest}
				[, treeNodeRoot
						[, treeNode]
						[, treeNodeInner
								[1, treeNodeLabeled],
							[, treeNodeInner
										[2, treeNodeLabeled],
									[3, treeNodeLabeled],
								]
						]
				]
			\end{forest}
		}
	\end{minipage},
	\begin{minipage}{0.07\textwidth}
		\centering
		\resizebox{\textwidth}{!}{%
			\begin{forest}
				[, treeNodeRoot
						[1, treeNodeLabeled]
						[, treeNodeInner
								[2, treeNodeLabeled],
							[, treeNodeInner
										[3, treeNodeLabeled],
									[, treeNode],
								]
						]
				]
			\end{forest}
		}
	\end{minipage},
	\begin{minipage}{0.07\textwidth}
		\centering
		\resizebox{\textwidth}{!}{%
			\begin{forest}
				[, treeNodeRoot
						[1, treeNodeLabeled]
						[, treeNodeInner
								[, treeNode],
							[, treeNodeInner
										[2, treeNodeLabeled],
									[3, treeNodeLabeled],
								]
						]
				]
			\end{forest}
		}
	\end{minipage}
	\right)^\top,
\end{align*}
as well as the three trees with five leaves which are fully labeled.

The unlabeled products of these flags (see also Appendix~\ref{app:products}) give the symmetric matrices
\renewcommand{\arraystretch}{1.25}
\begin{align*}
	\left\llbracket \FA\FAT\right\rrbracket      & = \begin{pmatrix}
		                                                 \frac{8}{15} \cat{5} + \frac{8}{15} \TIrr + \frac{2}{5} E_5
		                                                  & \quad
		                                                 \frac{2}{15} \cat{5} + \frac{2}{15} \TIrr + \frac{4}{15} E_5 \\
		                                                  & \quad
		                                                 \frac{1}{5} \cat{5} + \frac{1}{5} \TIrr + \frac{1}{15} E_5
	                                                 \end{pmatrix},                                                                                \\
	\left\llbracket \FB\FB^\top \right\rrbracket & = \begin{pmatrix}
		                                                 \frac{1}{20}\cat{5} & \frac{1}{30}E_5 & \frac{1}{60}\cat{5}                 & \frac{1}{15}\TIrr   & \frac{1}{60}\cat{5}                     \\
		                                                                     & \frac{1}{10}E_5 & \frac{1}{15}\TIrr                   & \frac{1}{30}E_5     & \frac{1}{30}E_5                         \\
		                                                                     &                 & \frac{1}{30}\cat{5}+\frac{1}{30}E_5 & \frac{1}{60}\cat{5} & \frac{1}{60}\cat{5}                     \\
		                                                                     &                 &                                     & \frac{1}{20}\cat{5} & \frac{1}{60}\cat{5}                     \\
		                                                                     &                 &                                     &                     & \frac{1}{30}\cat{5} + \frac{1}{15}\TIrr
	                                                 \end{pmatrix},
\end{align*}
where $\TIrr = \begin{minipage}{0.05\textwidth}\centering\resizebox{\textwidth}{!}{\begin{forest}
				[, treeNodeRoot
						[, treeNode]
						[, treeNodeInner
								[, treeNodeInner
										[, treeNode],
									[, treeNode],
								],
							[, treeNodeInner
										[, treeNode],
									[, treeNode],
								],
						]
				]
			\end{forest}}\end{minipage}$. The remaining three $1\times 1$-blocks correspond to slack variables for the three trees with five leaves.

The bound $I_5(E_5)$ is obtained by solving the SDP
\begin{equation}
	\begin{aligned}\label{eq:IndE5_Opt}
		\min_{\substack{X_1,X_2\succcurlyeq 0 \\x_{\cat{5}},x_\TIrr,x_{E_5},t\geq 0}} t \quad\text{s.t.}\quad   t\varnothing - E_5 = &\left\langle X_1, \left\llbracket \FA\FAT\right\rrbracket\right\rangle + \left\langle X_2, \left\llbracket \FB\FB^\top \right\rrbracket\right\rangle \\ &+ x_{\cat{5}} \cat{5} + x_\TIrr \TIrr + x_{E_5} E_5
	\end{aligned}
\end{equation}
and can easily be written in standard form by comparing the coefficients of both sides of the constraint, remembering that $\varnothing = \cat{5} + T + E_5$ by Lemma~\ref{lemma:chain_rule}.

\paragraph{Numerical and rigorous bound.} Solving this SDP numerically yields a bound of approximately $0.6666667$. After applying the rounding procedure described in Section~\ref{sec:comp}, we obtain the (slightly weaker) rigorous bound $\frac{262812752561575991588071}{394219100000000000000000}$. To derive the exact bound of $\frac{2}{3}$, we can apply more advanced rounding methods as, for example, developed in~\cite{cohn2024optimalitysphericalcodesexact}. Here, we follow the combinatorial
approach to obtain an exact certificate, introduced by Razborov~\cite{Razborov2010}, and built upon to include \emph{phantom edges} by Pikhurko and Vaughan~\cite{PIKHURKO_VAUGHAN_2013}.

\paragraph{Complementary slackness.} Any primal optimal solution of~\eqref{eq:IndE5_Opt} has to be \emph{complementary} to any dual optimal solution. \reviseSecond{If the bound is sharp, as is indicated by the numerical bound here,} we can construct optimal dual solutions from optimal constructions. Here the optimal construction is the infinite even tree $E_\infty$. Indeed, for any optimal construction $\T$, evaluating $\phi_\T$ on both sides of \eqref{eq:IndE5_Opt} implies
\reviseSecond{
	\begin{equation}\label{eq:compl_mat_unlabeled}
		\left\langle X_1, \phi_\T\left(\left\llbracket \FA\FAT\right\rrbracket\right)\right\rangle = \left\langle X_2, \phi_\T\left(\left\llbracket \FB\FB^\top \right\rrbracket\right)\right\rangle = 0,
	\end{equation}
}
and
\begin{equation}\label{eq:compl_scal}
	x_{\cat{5}} \phi_{\T}(\cat{5}) = x_\TIrr \phi_{\T}(\TIrr) = x_{E_5} \phi_{\T}(E_5) = 0.
\end{equation}
\reviseSecond{We want to use equation~\eqref{eq:compl_mat_unlabeled} to find explicit kernel vectors of the matrices $X_i$. To do so, we will need to remove the unlabeling operator $\llbracket \cdot\rrbracket$, and move from densities of (unlabeled) trees to densities of flags. Recall here that $\phi_\T\in \mathrm{Hom}^+(\A^\varnothing, \mathbb{R})$ refers to the homomorphism corresponding to the unlabeled densities, and $\psi_\T\in \mathrm{Hom}^+(\A^\sigma, \mathbb{R})$ to labeled densities in a labeling of $\T$.}

\reviseSecond{
	In contrast to unlabeling flags, moving from unlabeled trees to flags is more involved, as it is not unique. However, by Razborov~\cite{razborov2007}, Theorem~3.5 (and Definition~10), there exists a unique \emph{ensemble of random homomorphisms} $(\psi_{\mathbf{T}}^\sigma)_{\sigma \in \trees}$ that is \emph{rooted at} $\phi_\T$. Each $\psi_{\mathbf{T}}^\sigma$ is a random variable taking values in $\mathrm{Hom}^+(\A^\sigma, \mathbb{R})$ (i.e., it represents the vector of densities of $\sigma$-flags) arising from the different ways of assigning labels to the trees in the increasing sequence $\T$ so as to obtain convergent sequences of $\sigma$-flags. The index~$\mathbf{T}$ serves to indicate that the random variable is obtained by the limits of flags generated from the sequence~$\T$.

	Here, being \emph{rooted} means that for every quantum flag $f$ of type~$\sigma$,
	\begin{equation*}
		\mathbb{E}[\psi_{\mathbf{T}}^\sigma(f)]
		= \frac{\phi_\T(\llbracket f\rrbracket)}{q_\sigma(\sigma)\,\phi_\T(\sigma)},
	\end{equation*}
	and consequently,
	\begin{equation}
		\mathbb{E}\!\left[\left\langle X_1, \psi_{\mathbf{T}}^{\small\TOne} \FA\FAT\right\rangle\right]
		= \mathbb{E}\!\left[\left\langle X_2, \psi_{\mathbf{T}}^{\TThree} \FB\FB^\top\right\rangle\right]
		= 0.
	\end{equation}
	Since these expressions are expectations of nonnegative quantities, we find that, for almost every realization $\T$ of the labeling process,
	\begin{equation}\label{eq:compl_mat}
		\left\langle X_1, \reviseSecond{\psi}_{\T}(\FA)\reviseSecond{\psi}_{\T}(\FA)^\top \right\rangle
		= \left\langle X_2, \reviseSecond{\psi}_{\T}(\FB)\reviseSecond{\psi}_{\T}(\FB)^\top \right\rangle
		= 0.
	\end{equation}

}

\paragraph{Primal kernel.} The first kind of complementarity comes from the densities of graphs in $E_\infty$: for $S\in \{\cat{5}, \TIrr, E_5\}$, we have $x_S\phi_{E_\infty}(S)=0$ by equation~\eqref{eq:compl_scal}. All trees have a positive density in $E_\infty$, thus all three $x_S$ variables have to be zero. Such trees we call \emph{sharp}, following Lemma 3.2. of~\cite{PIKHURKO_VAUGHAN_2013}.

The second form of complementarity, known as forced zero eigenvectors, proceeds analogously to Lemma 3.1 of~\cite{PIKHURKO_VAUGHAN_2013}. For each type $\sigma \in \left\lbrace\TOne, \TThree\right\rbrace$, we can transform the sequence $E_\infty$ into an increasing sequence of $\sigma$-flags $(E_k, \eta_k)_{k \geq 0}$ by embedding $\sigma$ into each tree of $E_\infty$. There are multiple valid choices for where to place $\sigma$ within the trees so that the resulting sequence converges (i.e., the corresponding flag density sequences converge). Each such sequence yields a vector $\reviseSecond{\psi}_{\T}(\mathcal{F}_\sigma)\coloneqq(\reviseSecond{\psi}_{\T}(F))_{F \in \mathcal{F}_\sigma}$. By~\eqref{eq:compl_mat} we have $\langle X_i, \reviseSecond{\psi}_{\T}(\mathcal{F}_\sigma)\reviseSecond{\psi}_{\T}(\mathcal{F}_\sigma)^\top\rangle = 0$ \reviseSecond{almost surely, especially whenever we have a positive chance of obtaining the sequence of flags by placing the labels at random}. For any positive semidefinite matrix $X$ and any vector $v$ we have that $\langle X, vv^\top\rangle=v^\top Xv = 0$ implies $v\in\ker X$. Thus $\reviseSecond{\psi}_{\T}(\mathcal{F}_\sigma)$ must lie in the kernel of $X_1$ or $X_2$, respectively.

\newcommand{\MakeFlag}[3]{{#1}_{#2,{\small #3}}}

For the type $\TOne$, all such sequences of $\TOne$-flags created from  $E_\infty$ are asymptotically equal. We call this limit $\MakeFlag{E}{\infty}{\TOne} = (\MakeFlag{E}{n}{\TOne})_{n\geq 0}$.

Here we find three different ways to place labels in $E_\infty$ to turn it into a sequence of $\TThree$-flags:
\begin{equation*}
	\begin{minipage}{0.2\textwidth}\centering\resizebox{\textwidth}{!}{\begin{forest}
				[, treeNodeRoot
						[$\MakeFlag{E}{\tfrac n 2}{\textcircled{\small 1}}$, treeNodeLabeled],
					[, treeNodeInner
								[$\MakeFlag{E}{\tfrac n 4}{\textcircled{\small 2}}$, treeNodeLabeled],
							[$\MakeFlag{E}{\tfrac n 4}{\textcircled{\small 3}}$, treeNodeLabeled]
						]
				]
			\end{forest}}\end{minipage}, \hspace*{.15cm}
	\begin{minipage}{0.2\textwidth}\centering\resizebox{\textwidth}{!}{\reviseSecond{\begin{forest}
					[, treeNodeRoot
							[$\MakeFlag{E}{\tfrac n 2}{\textcircled{\small 1}}$, treeNodeLabeled],
						[, treeNodeInner
									[
										$E_{\tfrac{n}{4}}$, treeNodeLabeled
									],
								[
										, treeNodeInner
											[
												$\MakeFlag{E}{\tfrac n 8}{\textcircled{\small 2}}$, treeNodeLabeled
											],
										[
												$\MakeFlag{E}{\tfrac n 8}{\textcircled{\small 3}}$, treeNodeLabeled
											]
									]
							]
						\hspace*{-.1cm}
					]
				\end{forest}}}\end{minipage}, \hspace*{.2cm}
	\begin{minipage}{0.2\textwidth}\centering\resizebox{\textwidth}{!}{\reviseSecond{\begin{forest}
					[, treeNodeRoot
							[$E_{\tfrac{n}{2}}$, treeNodeLabeled],
						[, treeNodeInner
									[
										$\MakeFlag{E}{\tfrac n 4}{\textcircled{\small 1}}$,, treeNodeLabeled
									],
								[
										, treeNodeInner
											[
												$\MakeFlag{E}{\tfrac n 8}{\textcircled{\small 2}}$, treeNodeLabeled
											],
										[
												$\MakeFlag{E}{\tfrac n 8}{\textcircled{\small 3}}$, treeNodeLabeled
											]
									]
							]
						\hspace*{-.1cm}
					]
				\end{forest}}}\end{minipage}.
\end{equation*}
We call the three corresponding \reviseSecond{(flag)} limits $E_{\infty}^1, E_{\infty}^2, E_{\infty}^3$, each of which is the infinite even tree $E_\infty$ when unlabeled. \reviseSecond{Asymptotically, each has a  positive probability of being obtained by placing the labels at random. Thus, equation~\eqref{eq:compl_mat} holds for these constructions, and }evaluating the vectors of flags provides the following kernel vectors
\begin{align*}
	\reviseSecond{\psi}_{\MakeFlag{E}{\infty}{\resizebox{0.7em}{!}{\TOne}}}(\FA) & = (\tfrac 2 3, \tfrac 1 3)^\top,                                             &
	\reviseSecond{\psi}_{E_{\infty}^1}(\FB)                                      & = (\tfrac 1 4, \tfrac 1 2, 0, \tfrac{1}{4}, 0)^\top,                           \\
	\reviseSecond{\psi}_{E_{\infty}^2}(\FB)                                      & = \reviseSecond{(\tfrac 1 8, \tfrac 1 2, 0, \tfrac 1 8, \tfrac 1 4)^\top},   &
	\reviseSecond{\psi}_{E_{\infty}^3}(\FB)                                      & = \reviseSecond{(\tfrac 1 8, \tfrac 1 4, \tfrac{1}{2}, \tfrac 1 8, 0)^\top}.   \\
\end{align*}
This restricts the row spaces of $X_1$ and $X_2$ to subspaces of dimensions one and two respectively.

\paragraph{Phantom Kernel.} Numerically, the sums of squares certificate appears to have rank one, not two. However, there is no fourth way to place labels in $E_\infty$ to obtain a kernel vector linearly independent of the ones we already found. Such ``missing'' extremal constructions have been observed before, and \reviseSecond{in some cases can be explained through} \emph{phantom edges}, first introduced through Lemma 3.3. in~\cite{PIKHURKO_VAUGHAN_2013}. Adding finitely many edges to a graph limit does not modify it, as none of its subgraph densities change. But such additional edges may change the densities of flags in the constructions for specific placements of labels. Here we observe something similar: Adding a single \emph{phantom branch} at the root of $E_\infty$ does result in the same densities of subtrees in the limit, but gives us an additional option to place labels
\begin{equation*}
	\begin{minipage}{0.2\textwidth}\centering\resizebox{\textwidth}{!}{\begin{forest}
				[, treeNodeRoot
						[$\MakeFlag{E}{\tfrac n 2}{\textcircled{\small 1}}$, treeNodeLabeled],
					[, treeNodeInner
								[3, treeNodeLabeled],
							[$\MakeFlag{E}{\tfrac n 2}{\textcircled{\small 2}}$, treeNodeLabeled]
						]
				]
			\end{forest}}\end{minipage}\reviseSecond{,}
\end{equation*}
\reviseSecond{where $\MakeFlag{E}{\tfrac n 2}{\TOne}$ and $\MakeFlag{E}{\tfrac n 2}{\textcircled{\small 2}}$ are even balanced trees as defined above and $\textcircled{\small 3}$ is a single leaf with label $3$.}
We call its limit $\tilde{E}_\infty$. Note that for each (large enough) finite $n$ this tree is not an even tree. \reviseSecond{This limit of flags gives us the last kernel vector of $X_2$}
\begin{equation*}
	\reviseSecond{\psi}_{\tilde{E}_{\infty}}(\FB) = (\tfrac{1}{2}, \tfrac{1}{2}, 0, 0, 0)^\top.
\end{equation*}
\reviseSecond{
	\begin{remark}
		In contrast to the vectors obtained previously, we cannot simply assume here that this vector \emph{has to be} a kernel vector of any certificate $X_2$. The probability of placing the label $\textcircled{\small 3}$ on the new leaf is zero, and as such equation~\eqref{eq:compl_mat} does not have to hold. In some cases more elaborate arguments are possible~\cite{PIKHURKO_VAUGHAN_2013,falgas2015codegree} to get around this issue. The idea is as follows: we blow up the additional leaf such that it contains an $\varepsilon > 0$ of the vertices (e.g.\ by replacing it with an $E_{\varepsilon n}$), giving us a new close-to-optimal construction $\T_\varepsilon$. We then have probability $\varepsilon$ of picking the label $\textcircled{\small 3}$ in this blown up new leaf. If we apply $\phi_{\T_\varepsilon}$ on both sides of equation~\eqref{eq:IndE5_Opt} and follow steps analogous to before, we will end with an equation of the form
		\begin{equation*}
			\varepsilon \left\langle X_2, \reviseSecond{\psi}_{\mathbf{T}_\varepsilon}(\FB)\reviseSecond{\psi}_{\mathbf{T}_\varepsilon}(\FB)^\top \right\rangle = \mathcal{O}(\varepsilon^k),
		\end{equation*}
		where $k$ is the order of change in objective function, the inducibility of $E_5$ in $\T_\varepsilon$, as $\varepsilon$ approaches zero. If $k>1$ (i.e.\ the inducibility changes \emph{less} than expected for small $\varepsilon$), then $\reviseSecond{\psi}_{\tilde{E}_{\infty}}(\FB)$ has to be another kernel vector of $X_2$.
		Sadly, here we have~$k=1$, for this construction and all others we could think of. In such cases the vectors found this way can still be used to reduce the size of the SDP, but may potentially weaken the resulting bound.
	\end{remark}
}

\paragraph{The exact certificate.} We have already reduced the row spaces of both $X_1$ and $X_2$ to one-dimensional subspaces, and eliminated the slack variables $x_S$. It is straightforward to find an exact certificate
\begin{equation*}
	\frac{2}{3}\varnothing - E_5 = \frac{1}{3} \left\llbracket\left(\begin{minipage}{0.05\textwidth}
			\centering
			\resizebox{\textwidth}{!}{%
				\begin{forest}
					[, treeNodeRoot
							[, treeNode]
							[, treeNodeInner
									[1, treeNodeLabeled],
								[, treeNode]
							]
					]
				\end{forest}
			}
		\end{minipage}
	-2\cdot
	\begin{minipage}{0.05\textwidth}
			\centering
			\resizebox{\textwidth}{!}{%
				\begin{forest}
					[, treeNodeRoot
							[1, treeNodeLabeled]
							[, treeNodeInner
									[, treeNode],
								[, treeNode]
							]
					]
				\end{forest}
			}
		\end{minipage}\right)^2\right\rrbracket
	+ 2 \left\llbracket\left(
	\begin{minipage}{0.05\textwidth}
			\centering
			\resizebox{\textwidth}{!}{%
				\begin{forest}
					[, treeNodeRoot
							[1, treeNodeLabeled]
							[, treeNodeInner
									[3, treeNodeLabeled],
								[, treeNodeInner
											[2, treeNodeLabeled],
										[, treeNode],
									]
							]
					]
				\end{forest}
			}
		\end{minipage}
	-\begin{minipage}{0.05\textwidth}
			\centering
			\resizebox{\textwidth}{!}{%
				\begin{forest}
					[, treeNodeRoot
							[, treeNodeInner
									[1, treeNodeLabeled]
									[, treeNode],
							]
							[, treeNodeInner
									[2, treeNodeLabeled],
								[3, treeNodeLabeled],
							]
					]
				\end{forest}
			}
		\end{minipage}
	+\begin{minipage}{0.05\textwidth}
			\centering
			\resizebox{\textwidth}{!}{%
				\begin{forest}
					[, treeNodeRoot
							[1, treeNodeLabeled]
							[, treeNodeInner
									[2, treeNodeLabeled],
								[, treeNodeInner
											[3, treeNodeLabeled],
										[, treeNode],
									]
							]
					]
				\end{forest}
			}
		\end{minipage}
	+\begin{minipage}{0.05\textwidth}
			\centering
			\resizebox{\textwidth}{!}{%
				\begin{forest}
					[, treeNodeRoot
							[1, treeNodeLabeled]
							[, treeNodeInner
									[, treeNode],
								[, treeNodeInner
											[2, treeNodeLabeled],
										[3, treeNodeLabeled],
									]
							]
					]
				\end{forest}
			}
		\end{minipage}
	\right)^2\right\rrbracket.
\end{equation*}
This proves $I(E_5) = \frac{2}{3}$.

\paragraph{Dual Kernel.} The exact certificate also gives us information on all tree limits with $\phi_\mathcal{T}(E_5)=\frac{2}{3}$. By complementary slackness, with the same argument as above, we find that, for each such $\T$, we have
\begin{equation*}
	\begin{minipage}{0.05\textwidth}
		\centering
		\resizebox{\textwidth}{!}{%
			\begin{forest}
				[, treeNodeRoot
						[, treeNode]
						[, treeNodeInner
								[1, treeNodeLabeled],
							[, treeNode]
						]
				]
			\end{forest}
		}
	\end{minipage}
	=2\cdot
	\begin{minipage}{0.05\textwidth}
		\centering
		\resizebox{\textwidth}{!}{%
			\begin{forest}
				[, treeNodeRoot
						[1, treeNodeLabeled]
						[, treeNodeInner
								[, treeNode],
							[, treeNode]
						]
				]
			\end{forest}
		}
	\end{minipage}\quad\text{and}\quad
	\begin{minipage}{0.05\textwidth}
		\centering
		\resizebox{\textwidth}{!}{%
			\begin{forest}
				[, treeNodeRoot
						[1, treeNodeLabeled]
						[, treeNodeInner
								[3, treeNodeLabeled],
							[, treeNodeInner
										[2, treeNodeLabeled],
									[, treeNode],
								]
						]
				]
			\end{forest}
		}
	\end{minipage}
	+\begin{minipage}{0.05\textwidth}
		\centering
		\resizebox{\textwidth}{!}{%
			\begin{forest}
				[, treeNodeRoot
						[1, treeNodeLabeled]
						[, treeNodeInner
								[2, treeNodeLabeled],
							[, treeNodeInner
										[3, treeNodeLabeled],
									[, treeNode],
								]
						]
				]
			\end{forest}
		}
	\end{minipage}
	+\begin{minipage}{0.05\textwidth}
		\centering
		\resizebox{\textwidth}{!}{%
			\begin{forest}
				[, treeNodeRoot
						[1, treeNodeLabeled]
						[, treeNodeInner
								[, treeNode],
							[, treeNodeInner
										[2, treeNodeLabeled],
									[3, treeNodeLabeled],
								]
						]
				]
			\end{forest}
		}
	\end{minipage}
	=\begin{minipage}{0.05\textwidth}
		\centering
		\resizebox{\textwidth}{!}{%
			\begin{forest}
				[, treeNodeRoot
						[, treeNodeInner
								[1, treeNodeLabeled]
								[, treeNode],
						]
						[, treeNodeInner
								[2, treeNodeLabeled],
							[3, treeNodeLabeled],
						]
				]
			\end{forest}
		}
	\end{minipage}
\end{equation*}
\reviseSecond{almost surely for random placements of labels in $\T$.}


While we do not attempt a formal proof here, it appears reasonable to expect that these equations imply any optimal construction must be an even tree, up to a vanishing error as the number of leaves tends to infinity. A similar approach was recently employed to verify the uniqueness of the $D_4$-root system~\cite{Laat2024}.
Moreover, one might aim to establish a stability result: any nearly optimal construction $\T$ must be \emph{close} to the even tree, in a manner analogous to the stability phenomena first studied by Simonovits~\cite{simonovits1968method} and later extended to the flag algebra framework by Pikhurko~\cite{PIKHURKO20111142}. A full formalization and detailed proof of these ideas are left for future work.

\subsubsection{\revise{Exact bounds}}\label{sec:InducExact}

\revise{

	\begin{table}[ht!]
		\revise{
			\centering
			\begin{tabular}{@{}cccc@{}}
				\toprule
				Tree                                                           & Inducibility     & Maximizer                                  & Sharp bound \\ \midrule
				\vspace{.25cm}
				$\cat{4}$                                                      & $1$              & $\cat{\infty}$                             & $I_4$       \\\vspace{.25cm}
				$E_4$                                                          & $\frac{3}{7}$    & $E_\infty$                                 & $I_4$       \\\vspace{.25cm}
				$E_5$                                                          & $\frac{2}{3}$    & $E_\infty$                                 & $I_5$       \\\vspace{.25cm}
				$E_6$                                                          & $\frac{10}{31}$  & $E_\infty$                                 & $I_6$       \\\vspace{.25cm}
				$E_7$                                                          & $\frac{5}{21}$   & $E_\infty$                                 & $I_7$       \\\vspace{.25cm}
				$E_8$                                                          & $\frac{45}{889}$ & $E_\infty$                                 & $I_8$       \\\vspace{.25cm}
				$E_9$                                                          & $\frac{12}{85}$  & $E_\infty$                                 & $I_9$       \\\vspace{.25cm}
				$E_{10}$                                                       & $\frac{8}{73}$   & $E_\infty$                                 & $I_{10}$    \\
				\midrule
				\begin{minipage}{0.1\textwidth}
					\centering
					\resizebox{.7\textwidth}{!}{%
						\begin{forest}
							[, treeNodeRoot
									[, treeNodeInner[, treeNode][, treeNode]]
									[, treeNodeInner[, treeNode]
											[, treeNodeInner[, treeNode][, treeNodeInner[, treeNode],[, treeNode]]]]]
						\end{forest}}\end{minipage} & $\frac{15}{32}$  & $\dcat{\frac{1}{2}}{\frac{1}{2}}_{\infty}$ & $I_7$                           \\\vspace{.25cm}
				\begin{minipage}{0.1\textwidth}
					\centering
					\resizebox{.7\textwidth}{!}{%
						\begin{forest}
							[, treeNodeRoot
									[, treeNodeInner[, treeNode][, treeNode]]
									[, treeNodeInner[, treeNodeInner[, treeNode][, treeNode]][, treeNodeInner[, treeNode][, treeNode]]]]
						\end{forest}}\end{minipage} & $\frac{45}{217}$ & $E_\infty$                                 & $I_6$                           \\
				\bottomrule
			\end{tabular}
			\caption{Exact certificates: Recovered results from~\cite{wagner2017}, and two new inducibilities.}
			\label{tab:ind_bounds_exact}
		}
	\end{table}
}

We recall the exact inducibilities obtained by Czabarka et al.~\cite{wagner2017}: all caterpillar trees have inducibility $1$, and each even tree \revise{$E_k$} of size $k$ has inducibility $I(E_k)=k!\cdot c_k$, where
\begin{align*}
	c_{2s}   & = \frac{c_s^2}{2^{2s}-2},      &
	c_{2s+1} & = \frac{c_s c_{s+1}}{2^{2s}-1}
\end{align*}
\revise{and $c_0=c_1=1$.}
\revise{Note that this maximum density of $E_k$ is attained for the \emph{infinite even tree} $E_\infty = (E_1, E_2, E_3, E_4,\ldots)$}.

\revise{Table~\ref{tab:ind_bounds_exact} lists the known inducibilities of trees that we recovered (excluding the larger caterpillar trees), along with the hierarchy level at which each bound is attained. The certificates were rounded analogously to Section~\ref{sec:inducibilityEFive} semi-automatically using the ClusteredLowRankSolver~\cite{leijenhorstSolvingClusteredLowrank2024}, which implements the rounding procedure developed in~\cite{cohn2024optimalitysphericalcodesexact}. The last two inducibilities of trees of size six are new: we give the exact sums-of-squares certificates as ancillary files. It is straightfoward to verify the optimal constructions $E_\infty$ and the infinite balanced double caterpillar $\dcat{1/2}{1/2}_{\infty}$ (defined in detail in Section~\ref{ssec:ResultsTreeProfiles}), so we will omit the computations here.}

\subsubsection{\revise{The remaining three trees with up to six leaves}}

The tree \begin{minipage}{0.05\textwidth}
	\centering
	\resizebox{\textwidth}{!}{%
		\begin{forest}
			[, treeNodeRoot
					[, treeNode]
					[, treeNodeInner[, treeNodeInner[, treeNode][, treeNode]]
							[, treeNodeInner[, treeNode][, treeNode]]]]
		\end{forest}}\end{minipage}
has proven very challenging to study. Its inducibility is not yet known, but some bounds have been obtained algorithmically by Dossou-Olory and Wagner~\cite{DossouOlory2018}, improving results from~\cite{wagner2017}. They are as follows
\begin{align}
	0.247071 \leq I\left(\begin{minipage}{0.05\textwidth}
			                     \resizebox{\textwidth}{!}{%
				                     \begin{forest}
					[, treeNodeRoot
							[, treeNode]
							[, treeNodeInner[, treeNodeInner[, treeNode][, treeNode]]
									[, treeNodeInner[, treeNode][, treeNode]]]]
				\end{forest}} \end{minipage}\right) \leq \frac{32828685715097}{132667832500200} \approx 0.247450\,.
	\label{eq:boundStephanIrrTree}
\end{align}

In Table~\ref{tab:irrTree_bounds}, we present the upper bounds we obtain for the inducibility of \begin{minipage}{0.05\textwidth}
	\centering
	\resizebox{\textwidth}{!}{%
		\begin{forest}
			[, treeNodeRoot
					[, treeNode]
					[, treeNodeInner[, treeNodeInner[, treeNode][, treeNode]]
							[, treeNodeInner[, treeNode][, treeNode]]]]
		\end{forest}}\end{minipage}
at each level of the SDP hierarchy between $5$ and $11$, where the ones improving~\eqref{eq:boundStephanIrrTree} are marked in bold. \revise{We also give bounds for the remaining two trees with six leaves with unknown inducibility. For each of the three trees, the bounds decrease strictly with each level of the hierarchy. It seems unlikely that our strongest bounds are exact.}
\begin{table}[ht!]
	\centering
	\begin{tabular}{@{}cccc@{}}
		\toprule
		      & \multicolumn{3}{c}{Upper bounds}                         \\
		Level &
		\begin{minipage}{0.05\textwidth}
			\centering
			\resizebox{\textwidth}{!}{%
				\begin{forest}
					[, treeNodeRoot
							[, treeNode]
							[, treeNodeInner[, treeNodeInner[, treeNode][, treeNode]]
									[, treeNodeInner[, treeNode][, treeNode]]]]
				\end{forest}}\end{minipage}
		      &
		\begin{minipage}{0.1\textwidth}
			\resizebox{.7\textwidth}{!}{%
				\begin{forest}
					[, treeNodeRoot
							[, treeNode]
							[, treeNodeInner[, treeNodeInner[, treeNode],[, treeNode]]
									[, treeNodeInner[, treeNode],[, treeNodeInner[, treeNode],[, treeNode]]]]]
				\end{forest}}\end{minipage}
		      &
		\begin{minipage}{0.1\textwidth}
			\resizebox{.6\textwidth}{!}{%
				\begin{forest}
					[, treeNodeRoot
							[, treeNode]
							[, treeNodeInner[, treeNode][, treeNodeInner[, treeNodeInner[, treeNode],[, treeNode]][, treeNodeInner[, treeNode],[, treeNode]]]]] \end{forest}}\end{minipage}
		\\ \midrule
		$5$   & $0.3333335$                      & -         & -         \\
		$6$   & $0.2602938$                      & 0.3760761 & 0.2336372 \\
		$7$   & $0.2506628$                      & 0.3492771 & 0.1987351 \\
		$8$   & $0.2476918$                      & 0.3422886 & 0.1928592 \\
		$9$   & ${\mathbf{0.2471867}}$         & 0.3412286 & 0.1916752 \\
		$10$  & ${\mathbf{0.2471585}}$         & 0.3411696 & 0.1914929 \\
		$11$  & ${\mathbf{0.2471566}}$         & 0.3411657 & 0.1914539 \\ \bottomrule
	\end{tabular}
	\caption{Bounds up to $I_{11}$ for the three open cases with up to six leaves
	}
	\label{tab:irrTree_bounds}
\end{table}

We see that we are able to improve the previous best bound from level $9$ upwards. With level $11$, the bound we obtain allows us to reduce substantially the optimality gap in~\eqref{eq:boundStephanIrrTree}: from $37.9\cdot 10^{-5}$ to $8.5\cdot 10^{-5}$.

In both~\cite{wagner2017} and~\cite{DossouOlory2018}, the authors \revise{raised} the possibility that this tree may have an irrational inducibility - if rational, it would have to have a denominator of at least $89$. Proving this could be possible using our flag algebra software, and would answer the still open question of the possible irrationality of the inducibility of a tree.

\subsection{Tree-profiles}
\label{ssec:ResultsTreeProfiles}

We can compute a hierarchy of outer approximations of profiles of trees $\profile (T,S)\subseteq [0,1]^2$. The idea is the following: We partition the interval $[\min T, \max T]$ into smaller intervals
\begin{equation*}
	[\min T, \max T] = [a_1=\min T,a_2] \cup [a_2, a_3] \cup\ldots\cup [a_{n-1}, a_n=\max T].
\end{equation*}
On each interval $[a_i, a_{i+1}]$ we can then compute functions $f_{\mathrm{\scriptsize{lower}}}$ and $f_{\mathrm{\scriptsize{upper}}}$ which lower- (resp.\ upper-) bound the slice
\begin{equation*}
	\profile (T,S) \cap \left([a_i, a_{i+1}] \times \R\right),
\end{equation*}
in the sense that $S - f_{\mathrm{\scriptsize{lower}}}(T) \geq 0$ and $f_{\mathrm{\scriptsize{upper}}}(T) - S\geq 0$ if $T\in[a_i,a_{i+1}]$. We can optimize over (a suitable family of) integrable functions, trying to match them as closely as possible to the profile:
\begin{align*}
	\max_{f_{\mathrm{\scriptsize{lower}}}} \enspace & \int_{a_i}^{a_{i+1}} f_{\mathrm{\scriptsize{lower}}}(x) \,dx                 & \min_{f_{\mathrm{\scriptsize{upper}}}} \enspace & \int_{a_i}^{a_{i+1}} f_{\mathrm{\scriptsize{upper}}}(x) \,dx                \\
	\text{s.t.}\enspace                             & S - f_{\mathrm{\scriptsize{lower}}}(T) \geq 0 \text{ if } T\in[a_i,a_{i+1}]. & \text{s.t.}\enspace                             & f_{\mathrm{\scriptsize{upper}}}(T) - S\geq 0 \text{ if } T\in[a_i,a_{i+1}].
\end{align*}

Note that $f(T)\in \A^\varnothing$ if $f$ is a polynomial in $\R[x]$. While we can compute a high enough level of the hierarchy to use quadratic functions if $T$ has less than $6$ leaves, in practice it does not seem to result in much better approximations when compared to linear functions. Thus, we compute linear upper and lower bounds $f_{\mathrm{\scriptsize{lower}}}(x) = ax+b,f_{\mathrm{\scriptsize{upper}}}(x)=cx+d$ on the (sliced) profile by approximating
\begin{align*} 
	\max_{a,b} \enspace & \frac{1}{2}a(a_{i+1}^2-a_i^2) + b(a_{i+1}-a_i)  & \min_{c,d} \enspace & \frac{1}{2}c(a_{i+1}^2-a_i^2) + d(a_{i+1}-a_i) \\
	\text{s.t.}\enspace & S - aT - b \geq 0\text{ if } T\in[a_i,a_{i+1}]. & \text{s.t.}\enspace & cT+d-S \geq 0\text{ if } T\in[a_i,a_{i+1}].
\end{align*}

This kind of problem is an instance of the flag algebraic analogue to the \emph{generalized problem of moments}, see for example the dual formulation (equation 4) in the survey \cite{Klerk2019}.

We can relax this problem (following the outer approximation in \cite{Klerk2019}) to an SDP by replacing the nonnegativity constraint by being an element of the \emph{truncated quadratic module} generated by $T - a_i\geq 0$ and $a_{i+1}-T\geq 0$
\begin{equation*}
	M_L(T - a_i, a_{i+1}-T)	\coloneqq \left \lbrace s_0 + (T - a_i)s_1 + (a_{i+1}-T)s_2\mid s_0\in\mathrm{SOS}_L,  \ s_1,s_2\in \mathrm{SOS}_{L-|T|} \right\rbrace.
\end{equation*}
Note that, by definition, the elements of $M_L(T - a_i, a_{i+1}-T)\subseteq \A^\varnothing$ are nonnegative on sequences of trees whose $T$ density lies in $[a_i, a_{i+1}]$. Optimizing over the quadratic module amounts to solving an SDP. Here $M_L(T - a_i, a_{i+1}-T)$ was truncated such that it only contains quantum trees with at most $L$ leaves.

What we call level $L$ of the outer approximation of the $(T,S)$-profile is the piece-wise linear bound obtained from solving all pairs of SDPs for each slice of the profile:
\begin{align*} 
	\max_{a,b} \enspace & \frac{1}{2}a(a_{i+1}^2-a_i^2) + b(a_{i+1}-a_i) & \min_{c,d} \enspace & \frac{1}{2}c(a_{i+1}^2-a_i^2) + d(a_{i+1}-a_i) \\
	\text{s.t.}\enspace & S - aT - b \in M_L(T - a_i, a_{i+1}-T).        & \text{s.t.}\enspace & cT+d-S \in M_L(T - a_i, a_{i+1}-T).
\end{align*}
In all profile approximations pictured in this paper we sliced them into 100 sections.

We provide in Appendix~\ref{app:profiles} eight particularly interesting outer approximations of tree-profiles, where several level of the hierarchy are represented. We will now study three of them more thoroughly.

\paragraph{Outer approximations of the tree-profiles of the caterpillar trees of sizes 4 to 6 versus the even tree of size 6.}

We turn our attention to the tree-profiles of $\cat{k}$, the caterpillar tree of size $k$, and $E_6$, the even tree of size $6$, for $k \in \{4,5,6\}$. The four trees involved are recalled in Figure~\ref{fig:CatAndE}.

\begin{figure}[ht!]
	\centering
	\begin{subfigure}{.2\textwidth}
		\centering
		\includegraphics[width=.35\linewidth]{Pictures/Tree-6.pdf}
		\caption{$\cat{4}$}
	\end{subfigure}%
	\begin{subfigure}{.2\textwidth}
		\centering
		\includegraphics[width=.35\linewidth]{Pictures/Tree-8.pdf}
		\caption{$\cat{5}$}
	\end{subfigure}%
	\begin{subfigure}{.2\textwidth}
		\centering
		\includegraphics[width=.35\linewidth]{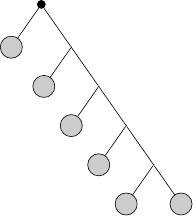}
		\caption{$\cat{6}$}
	\end{subfigure}%
	\begin{subfigure}{.2\textwidth}
		\centering
		\includegraphics[width=.45\linewidth]{Pictures/Tree-12.pdf}
		\caption{$E_6$}
	\end{subfigure}
	\caption{The caterpillar trees of sizes $4$, $5$ and $6$ and the even tree of size $6$}
	\label{fig:CatAndE}
\end{figure}

The three outer approximations of the tree-profiles of $\cat{k}$ and $E_6$ for $k \in \{4,5,6\}$ are represented in the first row of Figure~\ref{fig:profilesCatE}. Each level of the SDP hierarchy from $6$ to $11$ is drawn, the approximation becoming tighter as the level increases. We can thus clearly visualize the improvement brought the levels of the hierarchy, and note that they all seem to converge to tree-profiles with similar characteristics. The upper boundary appears to have two parts: we will give a conjecture for the expression of the right part. The right part of the upper bound turns non-convex and appears (numerically) sharp at levels $8$, $9$ and $10$, respectively. The lower boundary, however, seems more complex: the approximations are composed of multiple different sections (\revise{some of which} may not be visible yet), and each level of the hierarchy is able to substantially improve the approximation. This behavior is made obvious when plotting separately the improvement of the lower approximation with each level of the hierarchy. In each figure of the second row of Figure~\ref{fig:profilesCatE} are represented the respective gaps between the lower boundary obtained at levels $7$, $8$ and $8$ respectively of the hierarchy and those obtained at each level from $6$ to $11$. The leftmost section of each lowerbound appears (close to) linear, and is followed by (at least) two "scallops". As soon as the level reaches $7$,$8$ and $8$ respectively, the bound appears to stabilize at the point where the two scallops meet, suggesting existence of a "nice" extremal configuration. In the scallops themselves, the bounds improve strictly at every level, but \revise{numerical issues become more visible}, making it hard to claim anything specific. The apparent complexity of approximating these lower boundaries suggests that more advanced strategies (such as the \emph{variational methods} Razborov used for the triangle-edge profile~\cite{RAZBOROV2008}) may be necessary to find exact bounds.

\begin{figure}[ht!]
	\begin{center}
		\includegraphics[width=1.0\linewidth]{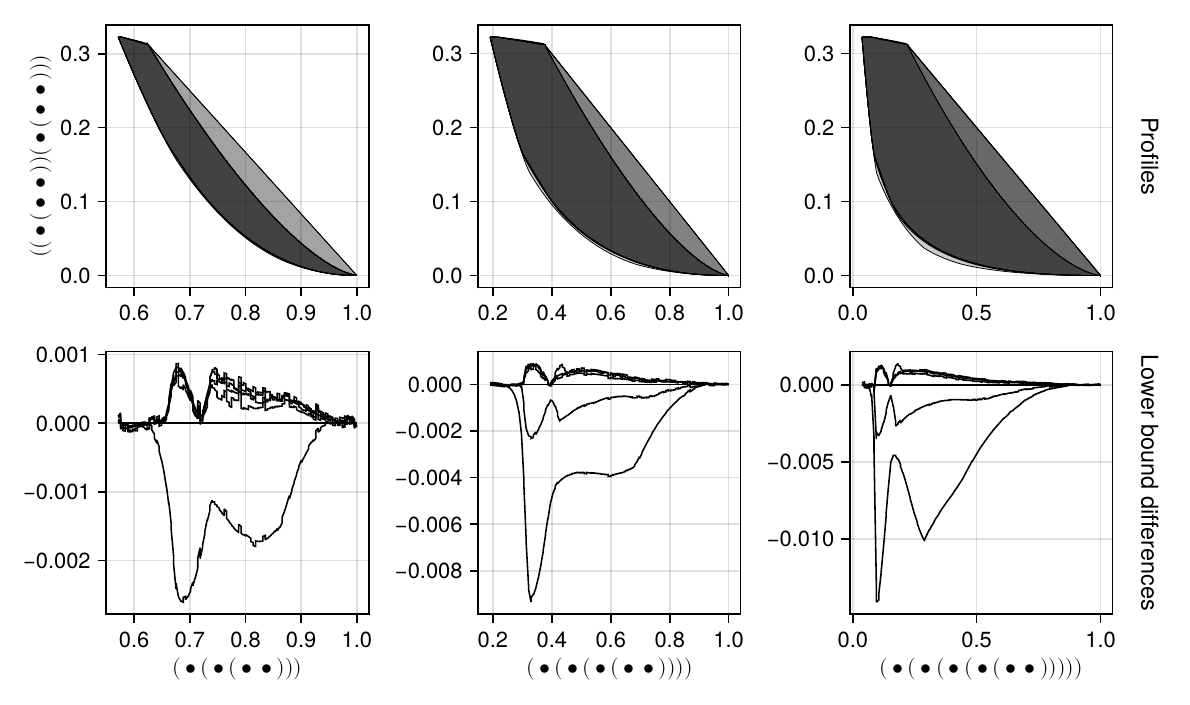}
		\vspace{-.5cm}
		\caption{Tree-profiles and lower bound differences of $\cat{4}$, $\cat{5}$ and $\cat{6}$ versus $E_6$}
		\label{fig:profilesCatE}
	\end{center}
\end{figure}

We will now formulate a conjecture for the right part of the upper boundary of each of these tree-profiles and prove some points of interest on these boundaries, depicted in Figure~\ref{fig:profileMarked} for $\cat{4}$.

\begin{figure}[ht!]
	\centering
	\includegraphics[width=.7\linewidth]{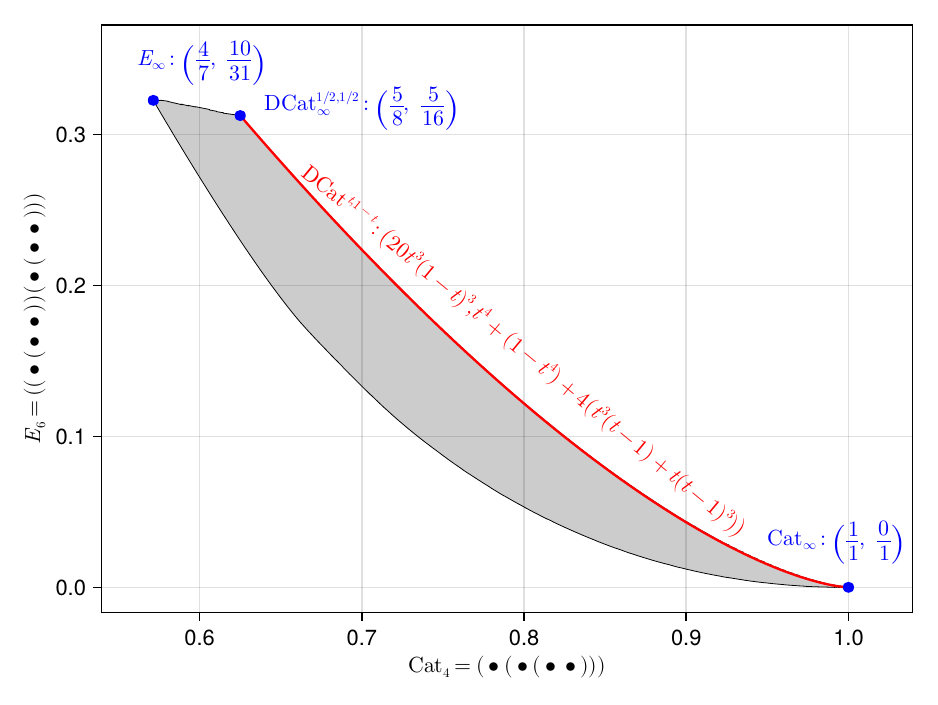}
	\caption{Tree-profile of $\cat{4}$ and $E_6$}
	\label{fig:profileMarked}
\end{figure}

\paragraph{Point 1: Right-most corner of the upper boundary.}
In each of these profiles, we can easily identify the known point $(1,0)$ corresponding to the infinite caterpillar $\cat{\infty}$, in which every finite caterpillar has density $1$ and all other trees have density $0$.

\paragraph{Right part of the upper boundary: the infinite double caterpillar.}

Let us denote \revise{by} $\dcat{p}{q}_n$ the double caterpillar of size $n$, composed of two caterpillars of size $n$ joined at the root with ratio $(p,q)$, i.e., one of the sides is a caterpillar of size $\lfloor pn\rfloor$ and the other is a caterpillar of size $\lceil qn \rceil$, with $p+q =1$.
\begin{figure}[ht!]
	\centering
	\includegraphics[width=.15\linewidth]{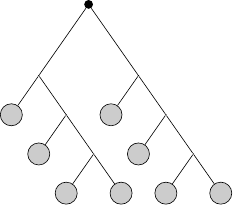}
	\caption{$\dcat{\frac{1}{2}}{\frac{1}{2}}_8$: the double caterpillar of size $8$ with ratio $(\frac{1}{2},\frac{1}{2})$}
\end{figure}

For any $k \geq 1$, the subtree density of $\cat{k}$ in $\dcat{p}{q}_n$ is
\begin{equation*}p(\cat{k}, \dcat{p}{q}_n) = \frac{\binom{\lfloor pn\rfloor}{k} + \binom{\lceil qn \rceil}{k} + \lfloor pn\rfloor\cdot \binom{\lceil qn \rceil}{k-1} + \lceil qn \rceil \cdot \binom{\lfloor pn\rfloor}{k-1}}{\binom{n}{k}}.\end{equation*}

Indeed, among the $\binom{n}{k}$ ways to pick a subset of leaves of size $k$ of $\dcat{p}{q}_n$, the ones inducing a tree isomorphic to $\cat{k}$ are: \begin{itemize}
	\item the $\binom{\lfloor pn\rfloor}{k} + \binom{\lceil qn \rceil}{k}$ ways of picking all $k$ leaves from the same side of the double caterpillar;
	\item and the $\lfloor pn\rfloor\cdot \binom{\lceil qn \rceil}{k-1} + \lceil qn \rceil \cdot \binom{\lfloor pn\rfloor}{k-1}$ ways of picking $(k-1)$ leaves on the same side of the double caterpillar, and the last leaf on the opposite side.
\end{itemize}

As $n$ tends to infinity, we obtain that \begin{align}
	\label{eq:UB_catk}
	\begin{split}
		p(\cat{k}, \dcat{p}{q}_{\infty}) & = \revise{\lim\limits_{n \to \infty}}\frac{\frac{p^kn^k}{k!} + \frac{q^kn^k}{k!} + \frac{pq^{k-1}n^k}{(k-1)!}+ \frac{p^{k-1}qn^k}{(k-1)!}}{\frac{n^k}{k!}} \\
		                                 & = p^k + q^k + k(pq^{k-1} + p^{k-1}q),
	\end{split}
\end{align}
where $\dcat{p}{q}_{\infty} = (\dcat{p}{q}_n)_{n\geq 2}$ is the increasing sequence of double caterpillars with ratio $(p,q)$.

The subtree density of $E_6$ in $\dcat{p}{q}_n$ is
\begin{equation*}p(E_6, \dcat{p}{q}_n) =  \frac{\binom{\lfloor pn\rfloor}{3}\binom{\lceil qn \rceil}{3}}{\binom{n}{6}}.
\end{equation*}
Indeed, the only way to pick a subset of $6$ leaves in $\dcat{p}{q}_n$ inducing a tree isomorphic to $E_6$, is to pick $3$ leaves on each side of $\dcat{p}{q}_n$.
We can thus compute the density of $E_6$ in $\dcat{p}{q}_{\infty}$:\begin{align}
	\label{eq:UB_E6}
	\begin{split}
		p(E_6, \dcat{p}{q}_{\infty}) & = \revise{\lim\limits_{n \to \infty}}\frac{\frac{p^3q^3n^6}{(3!)^2}}{\frac{n^6}{6!}} \\
		                             & = 20p^3q^3.
	\end{split}
\end{align}

\begin{conjecture}
	\label{conj:upperBoundary}
	$\dcat{p}{1-p}_\infty$ lies on the upper boundary of the tree-profile of $\cat{k}$ and $E_6$ for all values of $k\geq 4$ and $p\in [0,1]$.
\end{conjecture}

For every point of this curve, we are able to obtain a (numerically exact) sum of squares certificate that it indeed lies on the boundary of the profile. However, in the absence of information about the algebraic degree of the curve, this is not enough to prove Conjecture~\ref{conj:upperBoundary}, for we would need a parametrized family of sum of squares certificates for each point.

\paragraph{Point 2: Middle corner of the upper boundary.}

For each $k \in \{4,5,6\}$, we will now prove one point on the upper boundary of the tree-profile of $\cat{k}$ and $E_6$: the points $(\frac{5}{8}, \frac{5}{16})$ for $k=4$, $(\frac{3}{8},\frac{5}{16})$ for $k=5$ and $(\frac{7}{32},\frac{5}{16})$ for $k=6$. To do so, we will first show that each point is part of the profile corresponding to the densities of $\cat{k}$ and $E_6$ in the infinite double caterpillar; then we will provide a sum of squares certificate proving that they indeed lie on the upper boundary.

We can apply our previous results to compute the densities of $\cat{4}$, $\cat{5}$, $\cat{6}$ and $E_6$ in the ``balanced'' infinite double caterpillar, i.e. the infinite double caterpillar with ratio $(\frac{1}{2}, \frac{1}{2})$. In this case, we obtain from~\eqref{eq:UB_catk}
\begin{align*}
	p(\cat{k}, \dcat{\frac{1}{2}}{\frac{1}{2}}_{\infty}) & = \frac{k+1}{2^{k-1}}.
\end{align*}

We thus have
\begin{align*}
	p(\cat{4},\dcat{\frac{1}{2}}{\frac{1}{2}}_{\infty}) & = \frac{5}{8}  = 0.625,    \\
	p(\cat{5},\dcat{\frac{1}{2}}{\frac{1}{2}}_{\infty}) & = \frac{3}{8}  = 0.375,    \\
	p(\cat{6},\dcat{\frac{1}{2}}{\frac{1}{2}}_{\infty}) & = \frac{7}{32}  = 0.21875,
\end{align*}
and, following from~\eqref{eq:UB_E6} \begin{align*}
	p(E_6,\dcat{\frac{1}{2}}{\frac{1}{2}}_{\infty}) & = \frac{5}{16} = 0.3125.
\end{align*}

Our flag algebra software then gives us for each of these points a sum of squares certificate proving that they actually lie on the upper boundary of the tree-profiles. We detail below the certificate for $k=4$, and include (very similar) certificates for $k\in\{5,6\}$ as ancillary files.

We have
\vspace*{-.75cm}
\begin{center}
\begin{equation*}
	\Bigl\llbracket \ \Bigl( 
		2\!\begin{minipage}{0.13\textwidth}
			\centering
			\resizebox{.8\textwidth}{!}{%
			\begin{forest}
				[, treeNodeRoot
						[, treeNodeInner
								[3, treeNodeLabeled]
								[4, treeNodeLabeled]
						]
						[, treeNodeInner
							[2, treeNodeLabeled]
							[, treeNodeInner
								[1, treeNodeLabeled]
								[, treeNode]
							]
						]
				]
			\end{forest}
			}
		\end{minipage}
		\hspace{-.5cm}
		+ 2 \begin{minipage}{0.13\textwidth}
			\centering
			\resizebox{.8\textwidth}{!}{%
			\begin{forest}
				[, treeNodeRoot
						[, treeNodeInner
								[3, treeNodeLabeled]
								[4, treeNodeLabeled]
						]
						[, treeNodeInner
							[, treeNode]
							[, treeNodeInner
								[1, treeNodeLabeled]
								[2, treeNodeLabeled]
							]
						]
				]
			\end{forest}
			}
		\end{minipage}
		\hspace{-.5cm}
		+ 2	\begin{minipage}{0.13\textwidth}
			\centering
			\resizebox{.8\textwidth}{!}{%
			\begin{forest}
				[, treeNodeRoot
						[, treeNodeInner
								[3, treeNodeLabeled]
								[4, treeNodeLabeled]
						]
						[, treeNodeInner
							[1, treeNodeLabeled]
							[, treeNodeInner
								[2, treeNodeLabeled]
								[, treeNode]
							]
						]
				]
			\end{forest}
			}
		\end{minipage}
		\hspace{-.3cm}
		- 
		\hspace{-.25cm}
		\begin{minipage}{0.13\textwidth}
			\centering
			\resizebox{.8\textwidth}{!}{%
			\begin{forest}
				[, treeNodeRoot
						[, treeNodeInner
								[3, treeNodeLabeled]
								[4, treeNodeLabeled]
						]
						[, treeNodeInner
							[1, treeNodeLabeled]
							[2, treeNodeLabeled]
						]
				]
			\end{forest}
			}
		\end{minipage}
		\hspace{-.2cm}
		\Bigr)^2 \  \Bigr\rrbracket
\end{equation*}	
\begin{equation*}
	 = \ \frac{4}{15}
	\begin{minipage}{0.13\textwidth}
		\centering
		\resizebox{.8\textwidth}{!}{%
		\begin{forest}
			[, treeNodeRoot
					[, treeNodeInner
							[, treeNode]
							[, treeNode]
					]
					[, treeNodeInner
						[, treeNode]
						[, treeNodeInner
							[, treeNode]
							[, treeNodeInner
								[, treeNode]
								[, treeNode]
							]
						]
					]
			]
		\end{forest}
		}
	\end{minipage}
	\hspace{-.4cm}
	+ \ \frac{4}{15} \
	\begin{minipage}{0.12\textwidth}
		\centering
		\resizebox{.8\textwidth}{!}{%
		\begin{forest}
			[, treeNodeRoot
					[, treeNodeInner
							[, treeNode]
							[, treeNode]
					]
					[, treeNodeInner
						[, treeNodeInner
							[, treeNode]
							[, treeNode]
						]
						[, treeNodeInner
							[, treeNode]
							[, treeNode]
						]
					]
			]
		\end{forest}
		}
	\end{minipage}
	\hspace{-.3cm}
	- \ \frac{2}{5}
	\begin{minipage}{0.12\textwidth}
		\centering
		\resizebox{.8\textwidth}{!}{%
		\begin{forest}
			[, treeNodeRoot
					[, treeNodeInner
							[, treeNode]
							[, treeNode]
					]
					[, treeNodeInner
						[, treeNode]
						[, treeNodeInner
							[, treeNode]
							[, treeNode]
						]
					]
			]
		\end{forest}
		}
	\end{minipage}
	\hspace{-.3cm}
	+ \ \frac{1}{3}
	\begin{minipage}{0.1\textwidth}
		\centering
		\resizebox{.8\textwidth}{!}{%
		\begin{forest}
			[, treeNodeRoot
					[, treeNodeInner
							[, treeNode]
							[, treeNode]
					]
					[, treeNodeInner
						[, treeNode]
						[, treeNode]
					]
			]
		\end{forest}
		}
	\end{minipage}
\end{equation*}
\end{center}

We recall that all averaged products of trees-flags with up to $6$ vertices are given in Appendix~\ref{app:products}. We thus obtain that
\vspace*{-1cm}
\begin{center}
\begin{equation}
	\label{eq:sos1}
	4 \begin{minipage}{0.13\textwidth}
	   \centering
	   \resizebox{.8\textwidth}{!}{%
	   \begin{forest}
		   [, treeNodeRoot
				   [, treeNodeInner
						   [, treeNode]
						   [, treeNode]
				   ]
				   [, treeNodeInner
					   [, treeNode]
					   [, treeNodeInner
						   [, treeNode]
						   [, treeNodeInner
							   [, treeNode]
							   [, treeNode]
						   ]
					   ]
				   ]
		   ]
	   \end{forest}
	   }
   \end{minipage}
   \hspace{-.4cm}
   + \ 4 \
   \begin{minipage}{0.12\textwidth}
	   \centering
	   \resizebox{.8\textwidth}{!}{%
	   \begin{forest}
		   [, treeNodeRoot
				   [, treeNodeInner
						   [, treeNode]
						   [, treeNode]
				   ]
				   [, treeNodeInner
					   [, treeNodeInner
						   [, treeNode]
						   [, treeNode]
					   ]
					   [, treeNodeInner
						   [, treeNode]
						   [, treeNode]
					   ]
				   ]
		   ]
	   \end{forest}
	   }
   \end{minipage}
   \hspace{-.4cm}
   - \ 6
   \begin{minipage}{0.12\textwidth}
	   \centering
	   \resizebox{.8\textwidth}{!}{%
	   \begin{forest}
		   [, treeNodeRoot
				   [, treeNodeInner
						   [, treeNode]
						   [, treeNode]
				   ]
				   [, treeNodeInner
					   [, treeNode]
					   [, treeNodeInner
						   [, treeNode]
						   [, treeNode]
					   ]
				   ]
		   ]
	   \end{forest}
	   }
   \end{minipage}
   \hspace{-.3cm}
   + \ 5
   \begin{minipage}{0.1\textwidth}
	   \centering
	   \resizebox{.8\textwidth}{!}{%
	   \begin{forest}
		   [, treeNodeRoot
				   [, treeNodeInner
						   [, treeNode]
						   [, treeNode]
				   ]
				   [, treeNodeInner
					   [, treeNode]
					   [, treeNode]
				   ]
		   ]
	   \end{forest}
	   }
   \end{minipage}
   \ \geq \ 0.
\end{equation}
\end{center}

Using the following quotient relations
\vspace*{-.75cm}
\begin{center}
	\begin{equation*}
		\label{eq:quotient1}
		\begin{minipage}{0.1\textwidth}
			\centering
			\resizebox{.6\textwidth}{!}{%
			\begin{forest}
				[, treeNodeRoot
					[, treeNode]
					[, treeNodeInner
						[, treeNode]
						[, treeNode]
					]
				]
			\end{forest}
			}
		\end{minipage}
		= \begin{minipage}{0.1\textwidth}
			\centering
			\resizebox{.9\textwidth}{!}{%
			\begin{forest}
				[, treeNodeRoot
					[, treeNodeInner
						[, treeNode]
						[, treeNode]
					]
					[, treeNodeInner
						[, treeNode]
						[, treeNode]
					]
				]
			\end{forest}
			}
		\end{minipage}
		+ \begin{minipage}{0.1\textwidth}
			\centering
			\resizebox{.7\textwidth}{!}{%
			\begin{forest}
				[, treeNodeRoot
					[, treeNode]
					[, treeNodeInner
						[, treeNode]
						[, treeNodeInner
							[, treeNode]
							[, treeNode]
						]
					]
				]
			\end{forest}
			}
		\end{minipage}
	\end{equation*}
\end{center}

and 
\vspace*{-1cm}
\begin{center}
	\begin{equation*}
		\label{eq:quotient2}
		\begin{minipage}{0.12\textwidth}
			\centering
			\resizebox{.8\textwidth}{!}{%
			\begin{forest}
				[, treeNodeRoot
						[, treeNodeInner
								[, treeNode]
								[, treeNode]
						]
						[, treeNodeInner
							[, treeNode]
							[, treeNodeInner
								[, treeNode]
								[, treeNode]
							]
						]
				]
			\end{forest}
			}
		\end{minipage}
		= \frac{2}{3}
		\begin{minipage}{0.13\textwidth}
			\centering
			\resizebox{.8\textwidth}{!}{%
			\begin{forest}
				[, treeNodeRoot
						[, treeNodeInner
								[, treeNode]
								[, treeNode]
						]
						[, treeNodeInner
							[,treeNode]
							[, treeNodeInner
								[,treeNode]
								[, treeNodeInner
									[, treeNode]
									[, treeNode]
								]
							]
						]
				]
			\end{forest}
			}
		\end{minipage}
		+ \begin{minipage}{0.12\textwidth}
			\centering
			\resizebox{.8\textwidth}{!}{%
			\begin{forest}
				[, treeNodeRoot
					[, treeNodeInner
						[, treeNode]
						[, treeNodeInner
							[, treeNode]
							[, treeNode]
						]	
					]
					[, treeNodeInner
						[,treeNode]
						[, treeNodeInner
							[, treeNode]
							[, treeNode]
						]
					]
				]
			\end{forest}
			}
		\end{minipage}
		+ \frac{1}{6}\begin{minipage}{0.12\textwidth}
			\centering
			\resizebox{.8\textwidth}{!}{%
			\begin{forest}
				[, treeNodeRoot
					[, treeNode]
					[, treeNodeInner
						[,treeNodeInner
							[, treeNode]
							[, treeNode]
						]
						[, treeNodeInner
							[, treeNode]
							[, treeNodeInner
								[, treeNode]
								[, treeNode]
							]
						]
					]
				]
			\end{forest}
			}
		\end{minipage}
		+ \frac{2}{3}\begin{minipage}{0.12\textwidth}
			\centering
			\resizebox{.8\textwidth}{!}{%
			\begin{forest}
				[, treeNodeRoot
					[, treeNodeInner
							[, treeNode]
							[, treeNode]
					]
					[, treeNodeInner
						[, treeNodeInner
							[, treeNode]
							[, treeNode]
						]
						[, treeNodeInner
							[, treeNode]
							[, treeNode]
						]
					]
				]
			\end{forest}
			}
		\end{minipage}
	\end{equation*}
\end{center}

we can rewrite~\eqref{eq:sos1} as
\vspace*{-1cm}
\begin{center}
	\begin{equation*}
		5 - 5\begin{minipage}{0.1\textwidth}
			\centering
			\resizebox{.7\textwidth}{!}{%
			\begin{forest}
				[, treeNodeRoot
					[, treeNode]
					[, treeNodeInner
						[, treeNode]
						[, treeNodeInner
							[, treeNode]
							[, treeNode]
						]
					]
				]
			\end{forest}
			}
		\end{minipage}
		\hspace*{-.3cm}
		- \ 6 \begin{minipage}{0.12\textwidth}
			\centering
			\resizebox{.8\textwidth}{!}{%
			\begin{forest}
				[, treeNodeRoot
					[, treeNodeInner
						[, treeNode]
						[, treeNodeInner
							[, treeNode]
							[, treeNode]
						]	
					]
					[, treeNodeInner
						[,treeNode]
						[, treeNodeInner
							[, treeNode]
							[, treeNode]
						]
					]
				]
			\end{forest}
			}
		\end{minipage}
		\hspace*{-.2cm}
		\geq \begin{minipage}{0.12\textwidth}
			\centering
			\resizebox{.8\textwidth}{!}{%
			\begin{forest}
				[, treeNodeRoot
					[, treeNode]
					[, treeNodeInner
						[,treeNodeInner
							[, treeNode]
							[, treeNode]
						]
						[, treeNodeInner
							[, treeNode]
							[, treeNodeInner
								[, treeNode]
								[, treeNode]
							]
						]
					]
				]
			\end{forest}
			}
		\end{minipage}
		\hspace*{-.2cm}
		\geq 0,
	\end{equation*}
\end{center}
i.e.,
\vspace*{-1cm}
\begin{center}
	\begin{equation}
		\label{eq:finalSOS}
		 5\begin{minipage}{0.1\textwidth}
			\centering
			\resizebox{.7\textwidth}{!}{%
			\begin{forest}
				[, treeNodeRoot
					[, treeNode]
					[, treeNodeInner
						[, treeNode]
						[, treeNodeInner
							[, treeNode]
							[, treeNode]
						]
					]
				]
			\end{forest}
			}
		\end{minipage}
		\hspace*{-.3cm}
		+ \ 6 \begin{minipage}{0.12\textwidth}
			\centering
			\resizebox{.8\textwidth}{!}{%
			\begin{forest}
				[, treeNodeRoot
					[, treeNodeInner
						[, treeNode]
						[, treeNodeInner
							[, treeNode]
							[, treeNode]
						]	
					]
					[, treeNodeInner
						[,treeNode]
						[, treeNodeInner
							[, treeNode]
							[, treeNode]
						]
					]
				]
			\end{forest}
			}
		\end{minipage}
		\leq 5.
	\end{equation}	
\end{center}

The point $(\frac{5}{8}, \frac{5}{16})$ maximizes~\eqref{eq:finalSOS} in \revisedelete{both} direction \revise{of both axes}. Thus, it lies on a corner of upper boundary of the tree-profile of $\cat{4}$ and $E_6$.
\revisedelete{The certificates for $\cat{5}$ and $\cat{6}$ are very similar, and are given as ancillary files.}

\paragraph{Point 3: Left-most corner of the upper boundary.}

The works of Dossou-Olory and Wagner provide us with the left-most corners of the upper boundaries: the points $(\frac{4}{7}, \frac{10}{31})$, $(\frac{4}{21}, \frac{10}{31})$, and $(\frac{8}{217}, \frac{10}{31})$, for $k=4$, $k=5$, and $k=6$, respectively.
Indeed, they showed that they correspond to the densities of $\cat{k}$ and $E_6$ in the sequence of even trees $E_{\infty}= (E_n)_{n \to \infty}$, who simulateously minimize the density of $\cat{k}$ and maximize the one of $E_6$ .

Dossou-Olory \revise{proved} in~\cite{DossouOlory2018b} that the minimum asymptotic density of caterpillar trees is attained for $E_{\infty}$, and gives an explicit expression of this density. In the case of a binary caterpillar of size $k$, the formula is the following \begin{align*}
	p(\cat{k}, E_{\infty}) = \frac{k!}{2}\cdot \prod_{j=1}^{k-1}(2^j - 1)^{-1}.
\end{align*}
Applying this for $k=4,5,6$, we directly obtain
\begin{align*}
	p(\cat{4}, E_{\infty}) & = \frac{4}{7}, \\
	p(\cat{5}, E_{\infty}) & = \frac{4}{21}
\end{align*} and \begin{align*} p(\cat{6}, E_{\infty}) & = \frac{8}{217}.
\end{align*}

\revise{As previously mentioned, } we know from~\cite{wagner2017} that for every $r \geq 1$, the inducibility of $E_r$ is attained in the sequence $E_{\infty}$, and that $i(E_r) = p(E_r, E_{\infty})= r!\cdot c_r$, where
\begin{align*}
	c_{2s}   & = \frac{c_s^2}{2^{2s}-2},       \\
	c_{2s+1} & = \frac{c_s c_{s+1}}{2^{2s}-1}.
\end{align*}
This then gives us, with $r=6$,
\begin{align*} \label{eq:indE6}
	p(E_6, E_{\infty}) = \frac{10}{31}.
\end{align*}

\paragraph{Proving nonconvexity of the profiles.}
In contrast to the (local) profiles of trees considered by Bubeck and Linial \cite{bubeck2016TreeProfiles}, Figures~\ref{fig:profilesCatE} and~\ref{fig:profileMarked} show clearly that the tree-profiles considered in this paper can be nonconvex. We computed a rounded (and as such rigorous) sum of squares certificate for the upper bound $\frac{3322279122457465127}{25200000000000000000}\approx 0.1318365 \approx \frac{135}{1024} = p(E_6, \dcat{1/4}{3/4})$ for the $E_6$ density when $\cat{4} = \frac{101}{128} = p(\cat{4}, \dcat{1/4}{3/4})$. The certificate is given \revise{as} ancillary file, as it is unfortunately too large to be detailed here.

\section{Concluding remarks and open problems}

Applying the flag algebra theory to rooted binary trees has proved extremely efficient to obtain bounds on the densities and inducibilities of these trees. These very promising results give us many directions in which to continue researching on this topic: we list here the main ones.

\paragraph{Proving open conjectures on the inducibilities of trees.}
There are several conjectures about the maximizers of the density of a graph over all graphs of a fixed size. For example, in~\cite{wagner2017}, Czabarka et al. \revise{made} the following conjecture: for every $n\geq k$, $E_n$ has the largest number of copies of $E_k$ among all binary trees with $n$ leaves. In a similar fashion to what has been done in~\cite{Lidicky2023}, it may be possible to prove these conjectures using flag algebras.
In~\cite{DossouOlory2018}, Dossou-Olory and Wagner make the hypothesis that some trees may have an irrational inducibility, following on a question stated in~\cite{wagner2017}. As we mentioned in Section~\ref{sec:resultsInd}, they consider the
\revise{tree \begin{minipage}{0.05\textwidth}
		\centering
		\resizebox{\textwidth}{!}{%
			\begin{forest}
				[, treeNodeRoot
						[, treeNode]
						[, treeNodeInner[, treeNodeInner[, treeNode][, treeNode]]
								[, treeNodeInner[, treeNode][, treeNode]]]]
			\end{forest}}\end{minipage}}
as a possible tree with irrational inducibility. If the inducibility of this tree is indeed irrational, proving it could be possible using our framework, and would answer this open question.

\paragraph{Further investigation on the tree-profiles.}
There is still a lot to be discovered about tree-profiles. Our next goal would be to compute the exact description of the tree-profiles in Figure~\ref{fig:profilesCatE}. This means, on the one hand proving Conjecture~\ref{conj:upperBoundary}, and on the other hand determining the lower boundary which seems, as shown before, much more challenging to approach. Studying other tree-profiles in detail would also be of great interest and could lead us to learn more about the behavior and characteristics and the densities of these trees. In particular, we notice in Figure~\ref{fig:treeProfile8} \revise{in Appendix~\ref{app:profiles}} a sharp angle present on the $x$-axis at every level of the hierarchy above $1$. This leads us to think that we could obtain an exact Tur\'an number of \begin{minipage}{.04\textwidth}
	\resizebox{\textwidth}{!}{\begin{forest}
			[, treeNodeRoot [, treeNode] [, treeNodeInner [, treeNodeInner [, treeNode] [, treeNode] ] [, treeNodeInner [, treeNode] [, treeNodeInner [, treeNode] [, treeNode] ] ] ] ]
		\end{forest}
	}
\end{minipage} excluding \begin{minipage}{.04\textwidth}\resizebox{\textwidth}{!}{\begin{forest}
			[, treeNodeRoot [, treeNodeInner [, treeNode] [, treeNodeInner [, treeNodeInner [, treeNode] [, treeNode] ] [, treeNodeInner [, treeNode] [, treeNode] ] ] ] [, treeNode] ]
		\end{forest}
	}\end{minipage}.

\paragraph{Adapting to other types of trees.}
The settings to which this can be extended are numerous: $d$-ary and strictly $d$-ary trees, unrooted trees, colored trees, tree-pairs of phylogenetic trees as in~\cite{alon2016PhyloTrees}...

\paragraph{Acknowledgements}
	We are very grateful to Stephan Wagner for introducing us to this topic, and for his very valuable help throughout the project. We thank as well Luis Felipe Vargas for his useful comments. \revise{We also thank the reviewers for their insightful remarks, which have significantly improved the paper.}

\subsection*{Statements and Declarations}
\paragraph{Availability of data and materials.}
The code to recover the results obtained in this article will be provided as ArXiv ancillary files at \href{www.arxiv.org/abs/2404.12838}{arXiv:2404.12838}. No datasets are required.
\paragraph{Competing interests.} The authors have no relevant financial or non-financial interests to disclose.


\bibliographystyle{plain}
\bibliography{inducibilityTrees}

\newpage
\appendix
\section{Table of averaged products of trees}
\label{app:products}
We give here all the averaged products of trees-flags with up to $11$ leaves, and the equations obtained from quotienting out $\mathcal{K}^\varnothing$, as explained in Section~\ref{ssec:densitiesFlags}. In the interest of space, the trees are written in a compact way: unlabeled leaves are written as $\bullet$ and labeled leaves simply as their label (\circled{1} is written as $1$). The structure of the tree is given by the brackets: a tree with two branches $T_1$ and $T_2$ joining at the root node is written $(T_1T_2)$, and we proceed recursively to write the whole tree. For example, $(\bullet(\bullet\bullet)) = $ \begin{minipage}{0.04\textwidth}
	\centering
	\resizebox{\textwidth}{!}{%
		\begin{forest}
			[, treeNodeRoot
					[, treeNode]
					[, treeNodeInner[, treeNode][, treeNode]]]
		\end{forest}}\end{minipage}
and $((12)(3(4\bullet)))= $\begin{minipage}{0.1\textwidth}
	\centering
	\resizebox{\textwidth}{!}{%
		\begin{forest}
			[, treeNodeRoot
					[, treeNodeInner[1, treeNodeLabeled][2, treeNodeLabeled]]
					[, treeNodeInner[3, treeNodeLabeled][, treeNodeInner[4, treeNodeLabeled][, treeNode]]]]
		\end{forest}}\end{minipage}.
\begin{multicols}{2}
	{\allowdisplaybreaks\tiny
Type $\varnothing$
\begin{align*}
\llbracket \varnothing \cdot \varnothing\rrbracket &= \varnothing\\
\llbracket \varnothing \cdot \bullet\rrbracket &= \bullet\\
\llbracket \varnothing \cdot (\bullet\bullet)\rrbracket &= (\bullet\bullet)\\
\llbracket \varnothing \cdot (\bullet(\bullet\bullet))\rrbracket &= (\bullet(\bullet\bullet))\\
\llbracket \bullet \cdot \bullet\rrbracket &= (\bullet\bullet)\\
\llbracket \bullet \cdot (\bullet\bullet)\rrbracket &= (\bullet(\bullet\bullet))\\
\llbracket \bullet \cdot (\bullet(\bullet\bullet))\rrbracket &= (\bullet(\bullet(\bullet\bullet)))+ ((\bullet\bullet)(\bullet\bullet))\\
\llbracket (\bullet\bullet) \cdot (\bullet\bullet)\rrbracket &= (\bullet(\bullet(\bullet\bullet)))+ ((\bullet\bullet)(\bullet\bullet))\\
\llbracket (\bullet\bullet) \cdot (\bullet(\bullet\bullet))\rrbracket &= (\bullet(\bullet(\bullet(\bullet\bullet))))+ (\bullet((\bullet\bullet)(\bullet\bullet)))\\&\qquad + ((\bullet\bullet)(\bullet(\bullet\bullet)))\\
\llbracket (\bullet(\bullet\bullet)) \cdot (\bullet(\bullet\bullet))\rrbracket &= ((\bullet(\bullet\bullet))(\bullet(\bullet\bullet)))+ ((\bullet\bullet)((\bullet\bullet)(\bullet\bullet)))\\&\qquad + ((\bullet((\bullet\bullet)(\bullet\bullet)))\bullet)+ ((\bullet\bullet)(\bullet(\bullet(\bullet\bullet))))\\&\qquad + (\bullet(\bullet(\bullet(\bullet(\bullet\bullet)))))+ (\bullet((\bullet\bullet)(\bullet(\bullet\bullet))))\\
\end{align*}
Type $\bullet$
\begin{align*}
\llbracket 1 \cdot 1\rrbracket &= \bullet\\
\llbracket 1 \cdot (1\bullet)\rrbracket &= (\bullet\bullet)\\
\llbracket 1 \cdot (1(\bullet\bullet))\rrbracket &= \textstyle\tfrac{1}{3}(\bullet(\bullet\bullet))\\
\llbracket 1 \cdot (\bullet(1\bullet))\rrbracket &= \textstyle\tfrac{2}{3}(\bullet(\bullet\bullet))\\
\llbracket (1\bullet) \cdot (1\bullet)\rrbracket &= (\bullet(\bullet\bullet))\\
\llbracket (1\bullet) \cdot (1(\bullet\bullet))\rrbracket &= \textstyle\tfrac{1}{3}((\bullet\bullet)(\bullet\bullet))+ \textstyle\tfrac{1}{3}(\bullet(\bullet(\bullet\bullet)))\\
\llbracket (1\bullet) \cdot (\bullet(1\bullet))\rrbracket &= \textstyle\tfrac{2}{3}((\bullet\bullet)(\bullet\bullet))+ \textstyle\tfrac{2}{3}(\bullet(\bullet(\bullet\bullet)))\\
\llbracket (1(\bullet\bullet)) \cdot (1(\bullet\bullet))\rrbracket &= \textstyle\tfrac{1}{5}(\bullet(\bullet(\bullet(\bullet\bullet))))+ \textstyle\tfrac{1}{5}(\bullet((\bullet\bullet)(\bullet\bullet)))\\&\qquad + \textstyle\tfrac{1}{15}((\bullet\bullet)(\bullet(\bullet\bullet)))\\
\llbracket (1(\bullet\bullet)) \cdot (\bullet(1\bullet))\rrbracket &= \textstyle\tfrac{2}{15}(\bullet(\bullet(\bullet(\bullet\bullet))))+ \textstyle\tfrac{2}{15}(\bullet((\bullet\bullet)(\bullet\bullet)))\\&\qquad + \textstyle\tfrac{4}{15}((\bullet\bullet)(\bullet(\bullet\bullet)))\\
\llbracket (\bullet(1\bullet)) \cdot (\bullet(1\bullet))\rrbracket &= \textstyle\tfrac{8}{15}(\bullet(\bullet(\bullet(\bullet\bullet))))+ \textstyle\tfrac{8}{15}(\bullet((\bullet\bullet)(\bullet\bullet)))\\&\qquad + \textstyle\tfrac{2}{5}((\bullet\bullet)(\bullet(\bullet\bullet)))\\
\end{align*}
Type \resizebox{.06\textwidth}{!}{\begin{forest}
[, treeNodeRoot [, treeNode] [, treeNode] ]
\end{forest}
}
\begin{align*}
\llbracket (12) \cdot (12)\rrbracket &= (\bullet\bullet)\\
\llbracket (12) \cdot (1(2\bullet))\rrbracket &= \textstyle\tfrac{1}{3}(\bullet(\bullet\bullet))\\
\llbracket (12) \cdot (\bullet(12))\rrbracket &= \textstyle\tfrac{1}{3}(\bullet(\bullet\bullet))\\
\llbracket (12) \cdot ((12)(\bullet\bullet))\rrbracket &= \textstyle\tfrac{1}{3}((\bullet\bullet)(\bullet\bullet))\\
\llbracket (12) \cdot ((1\bullet)(2\bullet))\rrbracket &= \textstyle\tfrac{2}{3}((\bullet\bullet)(\bullet\bullet))\\
\llbracket (12) \cdot (1(2(\bullet\bullet)))\rrbracket &= \textstyle\tfrac{1}{12}(\bullet(\bullet(\bullet\bullet)))\\
\llbracket (12) \cdot (1(\bullet(2\bullet)))\rrbracket &= \textstyle\tfrac{1}{6}(\bullet(\bullet(\bullet\bullet)))\\
\llbracket (12) \cdot (\bullet(1(2\bullet)))\rrbracket &= \textstyle\tfrac{1}{6}(\bullet(\bullet(\bullet\bullet)))\\
\llbracket (12) \cdot (\bullet(\bullet(12)))\rrbracket &= \textstyle\tfrac{1}{6}(\bullet(\bullet(\bullet\bullet)))\\
\llbracket (1(2\bullet)) \cdot (1(2\bullet))\rrbracket &= \textstyle\tfrac{1}{4}(\bullet(\bullet(\bullet\bullet)))\\
\llbracket (1(2\bullet)) \cdot (2(1\bullet))\rrbracket &= \textstyle\tfrac{1}{3}((\bullet\bullet)(\bullet\bullet))\\
\llbracket (1(2\bullet)) \cdot (\bullet(12))\rrbracket &= \textstyle\tfrac{1}{12}(\bullet(\bullet(\bullet\bullet)))\\
\llbracket (1(2\bullet)) \cdot ((12)(\bullet\bullet))\rrbracket &= \textstyle\tfrac{1}{30}((\bullet\bullet)(\bullet(\bullet\bullet)))\\
\llbracket (1(2\bullet)) \cdot ((1\bullet)(2\bullet))\rrbracket &= \textstyle\tfrac{1}{5}((\bullet\bullet)(\bullet(\bullet\bullet)))\\
\llbracket (1(2\bullet)) \cdot (1(2(\bullet\bullet)))\rrbracket &= \textstyle\tfrac{1}{15}(\bullet(\bullet(\bullet(\bullet\bullet))))+ \textstyle\tfrac{1}{15}(\bullet((\bullet\bullet)(\bullet\bullet)))\\
\llbracket (1(2\bullet)) \cdot (1(\bullet(2\bullet)))\rrbracket &= \textstyle\tfrac{2}{15}(\bullet(\bullet(\bullet(\bullet\bullet))))+ \textstyle\tfrac{2}{15}(\bullet((\bullet\bullet)(\bullet\bullet)))\\
\llbracket (1(2\bullet)) \cdot (2(1(\bullet\bullet)))\rrbracket &= \textstyle\tfrac{1}{30}((\bullet\bullet)(\bullet(\bullet\bullet)))\\
\llbracket (1(2\bullet)) \cdot (2(\bullet(1\bullet)))\rrbracket &= \textstyle\tfrac{1}{15}((\bullet\bullet)(\bullet(\bullet\bullet)))\\
\llbracket (1(2\bullet)) \cdot (\bullet(1(2\bullet)))\rrbracket &= \textstyle\tfrac{1}{10}(\bullet(\bullet(\bullet(\bullet\bullet))))\\
\llbracket (1(2\bullet)) \cdot (\bullet(2(1\bullet)))\rrbracket &= \textstyle\tfrac{2}{15}(\bullet((\bullet\bullet)(\bullet\bullet)))\\
\llbracket (1(2\bullet)) \cdot (\bullet(\bullet(12)))\rrbracket &= \textstyle\tfrac{1}{30}(\bullet(\bullet(\bullet(\bullet\bullet))))\\
\llbracket (\bullet(12)) \cdot (\bullet(12))\rrbracket &= \textstyle\tfrac{1}{6}(\bullet(\bullet(\bullet\bullet)))+ \textstyle\tfrac{1}{3}((\bullet\bullet)(\bullet\bullet))\\
\llbracket (\bullet(12)) \cdot ((12)(\bullet\bullet))\rrbracket &= \textstyle\tfrac{1}{15}(\bullet((\bullet\bullet)(\bullet\bullet)))+ \textstyle\tfrac{2}{15}((\bullet\bullet)(\bullet(\bullet\bullet)))\\
\llbracket (\bullet(12)) \cdot ((1\bullet)(2\bullet))\rrbracket &= \textstyle\tfrac{2}{15}(\bullet((\bullet\bullet)(\bullet\bullet)))\\
\llbracket (\bullet(12)) \cdot (1(2(\bullet\bullet)))\rrbracket &= \textstyle\tfrac{1}{60}(\bullet(\bullet(\bullet(\bullet\bullet))))\\
\llbracket (\bullet(12)) \cdot (1(\bullet(2\bullet)))\rrbracket &= \textstyle\tfrac{1}{30}(\bullet(\bullet(\bullet(\bullet\bullet))))\\
\llbracket (\bullet(12)) \cdot (\bullet(1(2\bullet)))\rrbracket &= \textstyle\tfrac{1}{15}(\bullet(\bullet(\bullet(\bullet\bullet))))+ \textstyle\tfrac{1}{15}((\bullet\bullet)(\bullet(\bullet\bullet)))\\
\llbracket (\bullet(12)) \cdot (\bullet(\bullet(12)))\rrbracket &= \textstyle\tfrac{1}{10}(\bullet(\bullet(\bullet(\bullet\bullet))))+ \textstyle\tfrac{2}{15}(\bullet((\bullet\bullet)(\bullet\bullet)))\\&\qquad + \textstyle\tfrac{1}{15}((\bullet\bullet)(\bullet(\bullet\bullet)))\\
\llbracket ((12)(\bullet\bullet)) \cdot ((12)(\bullet\bullet))\rrbracket &= \textstyle\tfrac{1}{9}((\bullet\bullet)((\bullet\bullet)(\bullet\bullet)))+ \textstyle\tfrac{1}{15}((\bullet\bullet)(\bullet(\bullet(\bullet\bullet))))\\
\llbracket ((12)(\bullet\bullet)) \cdot ((1\bullet)(2\bullet))\rrbracket &= \textstyle\tfrac{2}{45}((\bullet\bullet)((\bullet\bullet)(\bullet\bullet)))\\
\llbracket ((12)(\bullet\bullet)) \cdot (1(2(\bullet\bullet)))\rrbracket &= \textstyle\tfrac{1}{180}((\bullet\bullet)(\bullet(\bullet(\bullet\bullet))))\\
\llbracket ((12)(\bullet\bullet)) \cdot (1(\bullet(2\bullet)))\rrbracket &= \textstyle\tfrac{1}{90}((\bullet\bullet)(\bullet(\bullet(\bullet\bullet))))\\
\llbracket ((12)(\bullet\bullet)) \cdot (\bullet(1(2\bullet)))\rrbracket &= \textstyle\tfrac{1}{15}((\bullet(\bullet\bullet))(\bullet(\bullet\bullet)))+ \textstyle\tfrac{1}{90}((\bullet\bullet)(\bullet(\bullet(\bullet\bullet))))\\&\qquad + \textstyle\tfrac{1}{90}(\bullet((\bullet\bullet)(\bullet(\bullet\bullet))))\\
\llbracket ((12)(\bullet\bullet)) \cdot (\bullet(\bullet(12)))\rrbracket &= \textstyle\tfrac{1}{15}((\bullet(\bullet\bullet))(\bullet(\bullet\bullet)))+ \textstyle\tfrac{1}{45}((\bullet((\bullet\bullet)(\bullet\bullet)))\bullet)\\&\qquad + \textstyle\tfrac{1}{90}((\bullet\bullet)(\bullet(\bullet(\bullet\bullet))))+ \textstyle\tfrac{2}{45}(\bullet((\bullet\bullet)(\bullet(\bullet\bullet))))\\
\llbracket ((1\bullet)(2\bullet)) \cdot ((1\bullet)(2\bullet))\rrbracket &= \textstyle\tfrac{2}{5}((\bullet(\bullet\bullet))(\bullet(\bullet\bullet)))\\
\llbracket ((1\bullet)(2\bullet)) \cdot (1(2(\bullet\bullet)))\rrbracket &= \textstyle\tfrac{2}{45}((\bullet\bullet)((\bullet\bullet)(\bullet\bullet)))+ \textstyle\tfrac{2}{45}((\bullet\bullet)(\bullet(\bullet(\bullet\bullet))))\\
\llbracket ((1\bullet)(2\bullet)) \cdot (1(\bullet(2\bullet)))\rrbracket &= \textstyle\tfrac{4}{45}((\bullet\bullet)((\bullet\bullet)(\bullet\bullet)))+ \textstyle\tfrac{4}{45}((\bullet\bullet)(\bullet(\bullet(\bullet\bullet))))\\
\llbracket ((1\bullet)(2\bullet)) \cdot (\bullet(1(2\bullet)))\rrbracket &= \textstyle\tfrac{1}{15}(\bullet((\bullet\bullet)(\bullet(\bullet\bullet))))\\
\llbracket ((1\bullet)(2\bullet)) \cdot (\bullet(\bullet(12)))\rrbracket &= \textstyle\tfrac{2}{45}((\bullet((\bullet\bullet)(\bullet\bullet)))\bullet)\\
\llbracket (1(2(\bullet\bullet))) \cdot (1(2(\bullet\bullet)))\rrbracket &= \textstyle\tfrac{1}{30}((\bullet((\bullet\bullet)(\bullet\bullet)))\bullet)+ \textstyle\tfrac{1}{30}(\bullet(\bullet(\bullet(\bullet(\bullet\bullet)))))\\&\qquad + \textstyle\tfrac{1}{90}(\bullet((\bullet\bullet)(\bullet(\bullet\bullet))))\\
\llbracket (1(2(\bullet\bullet))) \cdot (1(\bullet(2\bullet)))\rrbracket &= \textstyle\tfrac{1}{45}((\bullet((\bullet\bullet)(\bullet\bullet)))\bullet)+ \textstyle\tfrac{1}{45}(\bullet(\bullet(\bullet(\bullet(\bullet\bullet)))))\\&\qquad + \textstyle\tfrac{2}{45}(\bullet((\bullet\bullet)(\bullet(\bullet\bullet))))\\
\llbracket (1(2(\bullet\bullet))) \cdot (2(1(\bullet\bullet)))\rrbracket &= \textstyle\tfrac{1}{90}((\bullet(\bullet\bullet))(\bullet(\bullet\bullet)))\\
\llbracket (1(2(\bullet\bullet))) \cdot (2(\bullet(1\bullet)))\rrbracket &= \textstyle\tfrac{1}{45}((\bullet(\bullet\bullet))(\bullet(\bullet\bullet)))\\
\llbracket (1(2(\bullet\bullet))) \cdot (\bullet(1(2\bullet)))\rrbracket &= \textstyle\tfrac{1}{45}((\bullet((\bullet\bullet)(\bullet\bullet)))\bullet)+ \textstyle\tfrac{1}{45}(\bullet(\bullet(\bullet(\bullet(\bullet\bullet)))))\\
\llbracket (1(2(\bullet\bullet))) \cdot (\bullet(2(1\bullet)))\rrbracket &= \textstyle\tfrac{1}{90}(\bullet((\bullet\bullet)(\bullet(\bullet\bullet))))\\
\llbracket (1(2(\bullet\bullet))) \cdot (\bullet(\bullet(12)))\rrbracket &= \textstyle\tfrac{1}{180}(\bullet(\bullet(\bullet(\bullet(\bullet\bullet)))))\\
\llbracket (1(\bullet(2\bullet))) \cdot (1(\bullet(2\bullet)))\rrbracket &= \textstyle\tfrac{4}{45}((\bullet((\bullet\bullet)(\bullet\bullet)))\bullet)+ \textstyle\tfrac{4}{45}(\bullet(\bullet(\bullet(\bullet(\bullet\bullet)))))\\&\qquad + \textstyle\tfrac{1}{15}(\bullet((\bullet\bullet)(\bullet(\bullet\bullet))))\\
\llbracket (1(\bullet(2\bullet))) \cdot (2(\bullet(1\bullet)))\rrbracket &= \textstyle\tfrac{2}{45}((\bullet(\bullet\bullet))(\bullet(\bullet\bullet)))\\
\llbracket (1(\bullet(2\bullet))) \cdot (\bullet(1(2\bullet)))\rrbracket &= \textstyle\tfrac{2}{45}((\bullet((\bullet\bullet)(\bullet\bullet)))\bullet)+ \textstyle\tfrac{2}{45}(\bullet(\bullet(\bullet(\bullet(\bullet\bullet)))))\\
\llbracket (1(\bullet(2\bullet))) \cdot (\bullet(2(1\bullet)))\rrbracket &= \textstyle\tfrac{1}{45}(\bullet((\bullet\bullet)(\bullet(\bullet\bullet))))\\
\llbracket (1(\bullet(2\bullet))) \cdot (\bullet(\bullet(12)))\rrbracket &= \textstyle\tfrac{1}{90}(\bullet(\bullet(\bullet(\bullet(\bullet\bullet)))))\\
\llbracket (\bullet(1(2\bullet))) \cdot (\bullet(1(2\bullet)))\rrbracket &= \textstyle\tfrac{1}{15}((\bullet\bullet)(\bullet(\bullet(\bullet\bullet))))+ \textstyle\tfrac{1}{15}(\bullet(\bullet(\bullet(\bullet(\bullet\bullet)))))\\
\llbracket (\bullet(1(2\bullet))) \cdot (\bullet(2(1\bullet)))\rrbracket &= \textstyle\tfrac{4}{45}((\bullet\bullet)((\bullet\bullet)(\bullet\bullet)))+ \textstyle\tfrac{4}{45}((\bullet((\bullet\bullet)(\bullet\bullet)))\bullet)\\
\llbracket (\bullet(1(2\bullet))) \cdot (\bullet(\bullet(12)))\rrbracket &= \textstyle\tfrac{1}{45}((\bullet\bullet)(\bullet(\bullet(\bullet\bullet))))+ \textstyle\tfrac{1}{30}(\bullet(\bullet(\bullet(\bullet(\bullet\bullet)))))\\&\qquad + \textstyle\tfrac{1}{45}(\bullet((\bullet\bullet)(\bullet(\bullet\bullet))))\\
\llbracket (\bullet(\bullet(12))) \cdot (\bullet(\bullet(12)))\rrbracket &= \textstyle\tfrac{4}{45}((\bullet\bullet)((\bullet\bullet)(\bullet\bullet)))+ \textstyle\tfrac{4}{45}((\bullet((\bullet\bullet)(\bullet\bullet)))\bullet)\\&\qquad + \textstyle\tfrac{2}{45}((\bullet\bullet)(\bullet(\bullet(\bullet\bullet))))+ \textstyle\tfrac{1}{15}(\bullet(\bullet(\bullet(\bullet(\bullet\bullet)))))\\&\qquad + \textstyle\tfrac{2}{45}(\bullet((\bullet\bullet)(\bullet(\bullet\bullet))))\\
\end{align*}
Type \resizebox{.06\textwidth}{!}{\begin{forest}
[, treeNodeRoot [, treeNode] [, treeNodeInner [, treeNode] [, treeNode] ] ]
\end{forest}
}
\begin{align*}
\llbracket (1(23)) \cdot (1(23))\rrbracket &= \textstyle\tfrac{1}{3}(\bullet(\bullet\bullet))\\
\llbracket (1(23)) \cdot ((1\bullet)(23))\rrbracket &= \textstyle\tfrac{1}{3}((\bullet\bullet)(\bullet\bullet))\\
\llbracket (1(23)) \cdot (1(2(3\bullet)))\rrbracket &= \textstyle\tfrac{1}{12}(\bullet(\bullet(\bullet\bullet)))\\
\llbracket (1(23)) \cdot (1(\bullet(23)))\rrbracket &= \textstyle\tfrac{1}{12}(\bullet(\bullet(\bullet\bullet)))\\
\llbracket (1(23)) \cdot (\bullet(1(23)))\rrbracket &= \textstyle\tfrac{1}{12}(\bullet(\bullet(\bullet\bullet)))\\
\llbracket ((1\bullet)(23)) \cdot ((1\bullet)(23))\rrbracket &= \textstyle\tfrac{1}{10}((\bullet\bullet)(\bullet(\bullet\bullet)))\\
\llbracket ((1\bullet)(23)) \cdot (1(2(3\bullet)))\rrbracket &= \textstyle\tfrac{1}{30}((\bullet\bullet)(\bullet(\bullet\bullet)))\\
\llbracket ((1\bullet)(23)) \cdot (1(\bullet(23)))\rrbracket &= \textstyle\tfrac{1}{30}((\bullet\bullet)(\bullet(\bullet\bullet)))\\
\llbracket ((1\bullet)(23)) \cdot (\bullet(1(23)))\rrbracket &= \textstyle\tfrac{1}{15}(\bullet((\bullet\bullet)(\bullet\bullet)))\\
\llbracket (1(2(3\bullet))) \cdot (1(2(3\bullet)))\rrbracket &= \textstyle\tfrac{1}{20}(\bullet(\bullet(\bullet(\bullet\bullet))))\\
\llbracket (1(2(3\bullet))) \cdot (1(3(2\bullet)))\rrbracket &= \textstyle\tfrac{1}{15}(\bullet((\bullet\bullet)(\bullet\bullet)))\\
\llbracket (1(2(3\bullet))) \cdot (1(\bullet(23)))\rrbracket &= \textstyle\tfrac{1}{60}(\bullet(\bullet(\bullet(\bullet\bullet))))\\
\llbracket (1(2(3\bullet))) \cdot (\bullet(1(23)))\rrbracket &= \textstyle\tfrac{1}{60}(\bullet(\bullet(\bullet(\bullet\bullet))))\\
\llbracket (1(\bullet(23))) \cdot (1(\bullet(23)))\rrbracket &= \textstyle\tfrac{1}{30}(\bullet(\bullet(\bullet(\bullet\bullet))))+ \textstyle\tfrac{1}{15}(\bullet((\bullet\bullet)(\bullet\bullet)))\\
\llbracket (1(\bullet(23))) \cdot (\bullet(1(23)))\rrbracket &= \textstyle\tfrac{1}{60}(\bullet(\bullet(\bullet(\bullet\bullet))))\\
\llbracket (\bullet(1(23))) \cdot (\bullet(1(23)))\rrbracket &= \textstyle\tfrac{1}{30}(\bullet(\bullet(\bullet(\bullet\bullet))))+ \textstyle\tfrac{1}{30}((\bullet\bullet)(\bullet(\bullet\bullet)))\\
\end{align*}
Type \resizebox{.06\textwidth}{!}{\begin{forest}
[, treeNodeRoot [, treeNodeInner [, treeNode] [, treeNode] ] [, treeNodeInner [, treeNode] [, treeNode] ] ]
\end{forest}
}
\begin{align*}
\llbracket ((12)(34)) \cdot ((12)(34))\rrbracket &= \textstyle\tfrac{1}{3}((\bullet\bullet)(\bullet\bullet))\\
\llbracket ((12)(34)) \cdot ((12)(3(4\bullet)))\rrbracket &= \textstyle\tfrac{1}{30}((\bullet\bullet)(\bullet(\bullet\bullet)))\\
\llbracket ((12)(34)) \cdot ((12)(\bullet(34)))\rrbracket &= \textstyle\tfrac{1}{30}((\bullet\bullet)(\bullet(\bullet\bullet)))\\
\llbracket ((12)(34)) \cdot (\bullet((12)(34)))\rrbracket &= \textstyle\tfrac{1}{15}(\bullet((\bullet\bullet)(\bullet\bullet)))\\
\llbracket ((12)(3(4\bullet))) \cdot ((12)(3(4\bullet)))\rrbracket &= \textstyle\tfrac{1}{60}((\bullet\bullet)(\bullet(\bullet(\bullet\bullet))))\\
\llbracket ((12)(3(4\bullet))) \cdot ((12)(4(3\bullet)))\rrbracket &= \textstyle\tfrac{1}{45}((\bullet\bullet)((\bullet\bullet)(\bullet\bullet)))\\
\llbracket ((12)(3(4\bullet))) \cdot ((12)(\bullet(34)))\rrbracket &= \textstyle\tfrac{1}{180}((\bullet\bullet)(\bullet(\bullet(\bullet\bullet))))\\
\llbracket ((12)(3(4\bullet))) \cdot ((34)(1(2\bullet)))\rrbracket &= \textstyle\tfrac{1}{90}((\bullet(\bullet\bullet))(\bullet(\bullet\bullet)))\\
\llbracket ((12)(3(4\bullet))) \cdot ((34)(\bullet(12)))\rrbracket &= \textstyle\tfrac{1}{90}((\bullet(\bullet\bullet))(\bullet(\bullet\bullet)))\\
\llbracket ((12)(3(4\bullet))) \cdot (\bullet((12)(34)))\rrbracket &= \textstyle\tfrac{1}{180}(\bullet((\bullet\bullet)(\bullet(\bullet\bullet))))\\
\llbracket ((12)(\bullet(34))) \cdot ((12)(\bullet(34)))\rrbracket &= \textstyle\tfrac{1}{45}((\bullet\bullet)((\bullet\bullet)(\bullet\bullet)))+ \textstyle\tfrac{1}{90}((\bullet\bullet)(\bullet(\bullet(\bullet\bullet))))\\
\llbracket ((12)(\bullet(34))) \cdot ((34)(\bullet(12)))\rrbracket &= \textstyle\tfrac{1}{90}((\bullet(\bullet\bullet))(\bullet(\bullet\bullet)))\\
\llbracket ((12)(\bullet(34))) \cdot (\bullet((12)(34)))\rrbracket &= \textstyle\tfrac{1}{180}(\bullet((\bullet\bullet)(\bullet(\bullet\bullet))))\\
\llbracket (\bullet((12)(34))) \cdot (\bullet((12)(34)))\rrbracket &= \textstyle\tfrac{1}{45}((\bullet\bullet)((\bullet\bullet)(\bullet\bullet)))+ \textstyle\tfrac{1}{45}((\bullet((\bullet\bullet)(\bullet\bullet)))\bullet)\\
\end{align*}
Type \resizebox{.06\textwidth}{!}{\begin{forest}
[, treeNodeRoot [, treeNode] [, treeNodeInner [, treeNode] [, treeNodeInner [, treeNode] [, treeNode] ] ] ]
\end{forest}
}
\begin{align*}
\llbracket (1(2(34))) \cdot (1(2(34)))\rrbracket &= \textstyle\tfrac{1}{12}(\bullet(\bullet(\bullet\bullet)))\\
\llbracket (1(2(34))) \cdot ((1\bullet)(2(34)))\rrbracket &= \textstyle\tfrac{1}{30}((\bullet\bullet)(\bullet(\bullet\bullet)))\\
\llbracket (1(2(34))) \cdot (1((2\bullet)(34)))\rrbracket &= \textstyle\tfrac{1}{15}(\bullet((\bullet\bullet)(\bullet\bullet)))\\
\llbracket (1(2(34))) \cdot (1(2(3(4\bullet))))\rrbracket &= \textstyle\tfrac{1}{60}(\bullet(\bullet(\bullet(\bullet\bullet))))\\
\llbracket (1(2(34))) \cdot (1(2(\bullet(34))))\rrbracket &= \textstyle\tfrac{1}{60}(\bullet(\bullet(\bullet(\bullet\bullet))))\\
\llbracket (1(2(34))) \cdot (1(\bullet(2(34))))\rrbracket &= \textstyle\tfrac{1}{60}(\bullet(\bullet(\bullet(\bullet\bullet))))\\
\llbracket (1(2(34))) \cdot (\bullet(1(2(34))))\rrbracket &= \textstyle\tfrac{1}{60}(\bullet(\bullet(\bullet(\bullet\bullet))))\\
\llbracket ((1\bullet)(2(34))) \cdot ((1\bullet)(2(34)))\rrbracket &= \textstyle\tfrac{1}{30}((\bullet(\bullet\bullet))(\bullet(\bullet\bullet)))\\
\llbracket ((1\bullet)(2(34))) \cdot (1((2\bullet)(34)))\rrbracket &= \textstyle\tfrac{1}{45}((\bullet\bullet)((\bullet\bullet)(\bullet\bullet)))\\
\llbracket ((1\bullet)(2(34))) \cdot (1(2(3(4\bullet))))\rrbracket &= \textstyle\tfrac{1}{180}((\bullet\bullet)(\bullet(\bullet(\bullet\bullet))))\\
\llbracket ((1\bullet)(2(34))) \cdot (1(2(\bullet(34))))\rrbracket &= \textstyle\tfrac{1}{180}((\bullet\bullet)(\bullet(\bullet(\bullet\bullet))))\\
\llbracket ((1\bullet)(2(34))) \cdot (1(\bullet(2(34))))\rrbracket &= \textstyle\tfrac{1}{180}((\bullet\bullet)(\bullet(\bullet(\bullet\bullet))))\\
\llbracket ((1\bullet)(2(34))) \cdot (\bullet(1(2(34))))\rrbracket &= \textstyle\tfrac{1}{180}(\bullet((\bullet\bullet)(\bullet(\bullet\bullet))))\\
\llbracket (1((2\bullet)(34))) \cdot (1((2\bullet)(34)))\rrbracket &= \textstyle\tfrac{1}{60}(\bullet((\bullet\bullet)(\bullet(\bullet\bullet))))\\
\llbracket (1((2\bullet)(34))) \cdot (1(2(3(4\bullet))))\rrbracket &= \textstyle\tfrac{1}{180}(\bullet((\bullet\bullet)(\bullet(\bullet\bullet))))\\
\llbracket (1((2\bullet)(34))) \cdot (1(2(\bullet(34))))\rrbracket &= \textstyle\tfrac{1}{180}(\bullet((\bullet\bullet)(\bullet(\bullet\bullet))))\\
\llbracket (1((2\bullet)(34))) \cdot (1(\bullet(2(34))))\rrbracket &= \textstyle\tfrac{1}{90}((\bullet((\bullet\bullet)(\bullet\bullet)))\bullet)\\
\llbracket (1((2\bullet)(34))) \cdot (\bullet(1(2(34))))\rrbracket &= \textstyle\tfrac{1}{90}((\bullet((\bullet\bullet)(\bullet\bullet)))\bullet)\\
\llbracket (1(2(3(4\bullet)))) \cdot (1(2(3(4\bullet))))\rrbracket &= \textstyle\tfrac{1}{120}(\bullet(\bullet(\bullet(\bullet(\bullet\bullet)))))\\
\llbracket (1(2(3(4\bullet)))) \cdot (1(2(4(3\bullet))))\rrbracket &= \textstyle\tfrac{1}{90}((\bullet((\bullet\bullet)(\bullet\bullet)))\bullet)\\
\llbracket (1(2(3(4\bullet)))) \cdot (1(2(\bullet(34))))\rrbracket &= \textstyle\tfrac{1}{360}(\bullet(\bullet(\bullet(\bullet(\bullet\bullet)))))\\
\llbracket (1(2(3(4\bullet)))) \cdot (1(\bullet(2(34))))\rrbracket &= \textstyle\tfrac{1}{360}(\bullet(\bullet(\bullet(\bullet(\bullet\bullet)))))\\
\llbracket (1(2(3(4\bullet)))) \cdot (\bullet(1(2(34))))\rrbracket &= \textstyle\tfrac{1}{360}(\bullet(\bullet(\bullet(\bullet(\bullet\bullet)))))\\
\llbracket (1(2(\bullet(34)))) \cdot (1(2(\bullet(34))))\rrbracket &= \textstyle\tfrac{1}{90}((\bullet((\bullet\bullet)(\bullet\bullet)))\bullet)+ \textstyle\tfrac{1}{180}(\bullet(\bullet(\bullet(\bullet(\bullet\bullet)))))\\
\llbracket (1(2(\bullet(34)))) \cdot (1(\bullet(2(34))))\rrbracket &= \textstyle\tfrac{1}{360}(\bullet(\bullet(\bullet(\bullet(\bullet\bullet)))))\\
\llbracket (1(2(\bullet(34)))) \cdot (\bullet(1(2(34))))\rrbracket &= \textstyle\tfrac{1}{360}(\bullet(\bullet(\bullet(\bullet(\bullet\bullet)))))\\
\llbracket (1(\bullet(2(34)))) \cdot (1(\bullet(2(34))))\rrbracket &= \textstyle\tfrac{1}{180}(\bullet(\bullet(\bullet(\bullet(\bullet\bullet)))))+ \textstyle\tfrac{1}{180}(\bullet((\bullet\bullet)(\bullet(\bullet\bullet))))\\
\llbracket (1(\bullet(2(34)))) \cdot (\bullet(1(2(34))))\rrbracket &= \textstyle\tfrac{1}{360}(\bullet(\bullet(\bullet(\bullet(\bullet\bullet)))))\\
\llbracket (\bullet(1(2(34)))) \cdot (\bullet(1(2(34))))\rrbracket &= \textstyle\tfrac{1}{180}((\bullet\bullet)(\bullet(\bullet(\bullet\bullet))))+ \textstyle\tfrac{1}{180}(\bullet(\bullet(\bullet(\bullet(\bullet\bullet)))))\\
\end{align*}
Type \resizebox{.06\textwidth}{!}{\begin{forest}
[, treeNodeRoot [, treeNode] [, treeNodeInner [, treeNode] [, treeNodeInner [, treeNode] [, treeNodeInner [, treeNode] [, treeNode] ] ] ] ]
\end{forest}
}
\begin{align*}
\llbracket (1(2(3(45)))) \cdot (1(2(3(45))))\rrbracket &= \textstyle\tfrac{1}{60}(\bullet(\bullet(\bullet(\bullet\bullet))))\\
\end{align*}
Type \resizebox{.06\textwidth}{!}{\begin{forest}
[, treeNodeRoot [, treeNode] [, treeNodeInner [, treeNodeInner [, treeNode] [, treeNode] ] [, treeNodeInner [, treeNode] [, treeNode] ] ] ]
\end{forest}
}
\begin{align*}
\llbracket (1((23)(45))) \cdot (1((23)(45)))\rrbracket &= \textstyle\tfrac{1}{15}(\bullet((\bullet\bullet)(\bullet\bullet)))\\
\end{align*}
Type \resizebox{.06\textwidth}{!}{\begin{forest}
[, treeNodeRoot [, treeNodeInner [, treeNode] [, treeNode] ] [, treeNodeInner [, treeNode] [, treeNodeInner [, treeNode] [, treeNode] ] ] ]
\end{forest}
}
\begin{align*}
\llbracket ((12)(3(45))) \cdot ((12)(3(45)))\rrbracket &= \textstyle\tfrac{1}{30}((\bullet\bullet)(\bullet(\bullet\bullet)))\\
\end{align*}
Type \resizebox{.06\textwidth}{!}{\begin{forest}
[, treeNodeRoot [, treeNodeInner [, treeNode] [, treeNodeInner [, treeNodeInner [, treeNode] [, treeNode] ] [, treeNodeInner [, treeNode] [, treeNode] ] ] ] [, treeNode] ]
\end{forest}
}
\begin{align*}
\llbracket ((1((23)(45)))6) \cdot ((1((23)(45)))6)\rrbracket &= \textstyle\tfrac{1}{90}((\bullet((\bullet\bullet)(\bullet\bullet)))\bullet)\\
\end{align*}
Type \resizebox{.06\textwidth}{!}{\begin{forest}
[, treeNodeRoot [, treeNodeInner [, treeNode] [, treeNode] ] [, treeNodeInner [, treeNode] [, treeNodeInner [, treeNode] [, treeNodeInner [, treeNode] [, treeNode] ] ] ] ]
\end{forest}
}
\begin{align*}
\llbracket ((12)(3(4(56)))) \cdot ((12)(3(4(56))))\rrbracket &= \textstyle\tfrac{1}{180}((\bullet\bullet)(\bullet(\bullet(\bullet\bullet))))\\
\end{align*}
Type \resizebox{.06\textwidth}{!}{\begin{forest}
[, treeNodeRoot [, treeNode] [, treeNodeInner [, treeNodeInner [, treeNode] [, treeNode] ] [, treeNodeInner [, treeNode] [, treeNodeInner [, treeNode] [, treeNode] ] ] ] ]
\end{forest}
}
\begin{align*}
\llbracket (1((23)(4(56)))) \cdot (1((23)(4(56))))\rrbracket &= \textstyle\tfrac{1}{180}(\bullet((\bullet\bullet)(\bullet(\bullet\bullet))))\\
\end{align*}
Type \resizebox{.06\textwidth}{!}{\begin{forest}
[, treeNodeRoot [, treeNodeInner [, treeNode] [, treeNodeInner [, treeNode] [, treeNode] ] ] [, treeNodeInner [, treeNode] [, treeNodeInner [, treeNode] [, treeNode] ] ] ]
\end{forest}
}
\begin{align*}
\llbracket ((1(23))(4(56))) \cdot ((1(23))(4(56)))\rrbracket &= \textstyle\tfrac{1}{90}((\bullet(\bullet\bullet))(\bullet(\bullet\bullet)))\\
\end{align*}
Type \resizebox{.06\textwidth}{!}{\begin{forest}
[, treeNodeRoot [, treeNodeInner [, treeNode] [, treeNode] ] [, treeNodeInner [, treeNodeInner [, treeNode] [, treeNode] ] [, treeNodeInner [, treeNode] [, treeNode] ] ] ]
\end{forest}
}
\begin{align*}
\llbracket ((12)((34)(56))) \cdot ((12)((34)(56)))\rrbracket &= \textstyle\tfrac{1}{45}((\bullet\bullet)((\bullet\bullet)(\bullet\bullet)))\\
\end{align*}
Type \resizebox{.06\textwidth}{!}{\begin{forest}
[, treeNodeRoot [, treeNode] [, treeNodeInner [, treeNode] [, treeNodeInner [, treeNode] [, treeNodeInner [, treeNode] [, treeNodeInner [, treeNode] [, treeNode] ] ] ] ] ]
\end{forest}
}
\begin{align*}
\llbracket (1(2(3(4(56))))) \cdot (1(2(3(4(56)))))\rrbracket &= \textstyle\tfrac{1}{360}(\bullet(\bullet(\bullet(\bullet(\bullet\bullet)))))\\
\end{align*}
Quotient:
\begin{align*}
\varnothing &= \bullet \cdot \varnothing = \bullet\\
\bullet &= \bullet \cdot \bullet = (\bullet\bullet)\\
(\bullet\bullet) &= \bullet \cdot (\bullet\bullet) = (\bullet(\bullet\bullet))\\
(\bullet(\bullet\bullet)) &= \bullet \cdot (\bullet(\bullet\bullet)) = ((\bullet\bullet)(\bullet\bullet))+ (\bullet(\bullet(\bullet\bullet)))\\
(\bullet(\bullet(\bullet\bullet))) &= \bullet \cdot (\bullet(\bullet(\bullet\bullet))) = (\bullet(\bullet(\bullet(\bullet\bullet))))+ \tfrac{4}{5}(\bullet((\bullet\bullet)(\bullet\bullet)))+ \tfrac{2}{5}((\bullet\bullet)(\bullet(\bullet\bullet)))\\
((\bullet\bullet)(\bullet\bullet)) &= \bullet \cdot ((\bullet\bullet)(\bullet\bullet)) = \tfrac{1}{5}(\bullet((\bullet\bullet)(\bullet\bullet)))+ \tfrac{3}{5}((\bullet\bullet)(\bullet(\bullet\bullet)))\\
(\bullet(\bullet(\bullet(\bullet\bullet)))) &= \bullet \cdot (\bullet(\bullet(\bullet(\bullet\bullet)))) = \tfrac{2}{3}((\bullet((\bullet\bullet)(\bullet\bullet)))\bullet)+ \tfrac{1}{3}((\bullet\bullet)(\bullet(\bullet(\bullet\bullet))))\\&\qquad+ \tfrac{1}{3}(\bullet((\bullet\bullet)(\bullet(\bullet\bullet))))+ (\bullet(\bullet(\bullet(\bullet(\bullet\bullet)))))\\
((\bullet\bullet)(\bullet(\bullet\bullet))) &= \bullet \cdot ((\bullet\bullet)(\bullet(\bullet\bullet))) = \tfrac{2}{3}((\bullet\bullet)(\bullet(\bullet(\bullet\bullet))))+ ((\bullet(\bullet\bullet))(\bullet(\bullet\bullet)))\\&\qquad+ \tfrac{1}{6}(\bullet((\bullet\bullet)(\bullet(\bullet\bullet))))+ \tfrac{2}{3}((\bullet\bullet)((\bullet\bullet)(\bullet\bullet)))\\
(\bullet((\bullet\bullet)(\bullet\bullet))) &= \bullet \cdot (\bullet((\bullet\bullet)(\bullet\bullet))) = \tfrac{1}{3}((\bullet((\bullet\bullet)(\bullet\bullet)))\bullet)+ \tfrac{1}{2}(\bullet((\bullet\bullet)(\bullet(\bullet\bullet))))\\&\qquad+ \tfrac{1}{3}((\bullet\bullet)((\bullet\bullet)(\bullet\bullet)))\\
\end{align*}
}

\end{multicols}

\section{Bounds on the inducibility of trees}
\label{app:bounds}

\revise{We provide here} the full list of bounds for the inducibility of trees with up to $11$ leaves. The bounds are given in the form $T\leq I_{11}(T)\leq I_{10}(T)$. If $T$ has more than $10$ leaves, we only give the bound $I_{11}(T)$. The bounds were obtained by numerically solving the SDPs of Section \ref{sec:resultsInd}, and then rounding their numerical certificates to rigorous rational certificates using the procedure described in Section \ref{sec:comp}. Due to the number of digits in the fractions of the resulting rational bounds, we here only give their first few decimals, rounded up. In some cases the bound at level $11$ is slightly worse than the bound at level $10$ after rounding. In those cases the second inequality is marked in red, and should be interpreted as $T\leq \min (I_{11}(T), I_{10}(T))$.
The known inducibilites \revisedelete{from} that we recover (up to precision $10^{-5}$), and the known bounds that we improve are marked with a star.
\begin{multicols}{2}
	{\tiny\allowdisplaybreaks\begin{align*}
    \bullet&\leq 1.0\leq 1.0^*\\
    (\bullet\bullet)&\leq 1.0\leq 1.0^*\\
    (\bullet(\bullet\bullet))&\leq 1.0\leq 1.0^*\\
    ((\bullet\bullet)(\bullet\bullet))&\leq 0.4285724\red{\leq} 0.4285723^*\\
    (\bullet(\bullet(\bullet\bullet)))&\leq 1.0000023\red{\leq} 1.000001^*\\
    ((\bullet\bullet)(\bullet(\bullet\bullet)))&\leq 0.6666669\leq 0.6666692^*\\
    (\bullet((\bullet\bullet)(\bullet\bullet)))&\leq 0.2471566\leq 0.2471585^*\\
    (\bullet(\bullet(\bullet(\bullet\bullet))))&\leq 1.0000001\leq 1.0000003^*\\
    ((\bullet(\bullet\bullet))(\bullet(\bullet\bullet)))&\leq 0.3225817\red{\leq} 0.3225814^*\\
    ((\bullet\bullet)((\bullet\bullet)(\bullet\bullet)))&\leq 0.2073743\red{\leq} 0.2073739\\
    ((\bullet\bullet)(\bullet(\bullet(\bullet\bullet))))&\leq 0.4687508\red{\leq} 0.4687506\\
    (\bullet((\bullet\bullet)(\bullet(\bullet\bullet))))&\leq 0.3411657\leq 0.3411696\\
    (\bullet(\bullet((\bullet\bullet)(\bullet\bullet))))&\leq 0.1914539\leq 0.1914929\\
    (\bullet(\bullet(\bullet(\bullet(\bullet\bullet)))))&\leq 1.0000036\red{\leq} 1.000001^*\\
    ((\bullet(\bullet\bullet))((\bullet\bullet)(\bullet\bullet)))&\leq 0.2380974\red{\leq} 0.2380958^*\\
    ((\bullet(\bullet\bullet))(\bullet(\bullet(\bullet\bullet))))&\leq 0.5468753\leq 0.5468759\\
    ((\bullet\bullet)((\bullet\bullet)(\bullet(\bullet\bullet))))&\leq 0.2472173\leq 0.2472174\\
    ((\bullet\bullet)(\bullet((\bullet\bullet)(\bullet\bullet))))&\leq 0.0880895\leq 0.0881718\\
    ((\bullet\bullet)(\bullet(\bullet(\bullet(\bullet\bullet)))))&\leq 0.3456807\red{\leq} 0.3456797\\
    (\bullet((\bullet(\bullet\bullet))(\bullet(\bullet\bullet))))&\leq 0.1440113\leq 0.1440145\\
    (\bullet((\bullet\bullet)((\bullet\bullet)(\bullet\bullet))))&\leq 0.1048657\leq 0.1048672\\
    (\bullet((\bullet\bullet)(\bullet(\bullet(\bullet\bullet)))))&\leq 0.2086242\leq 0.2086268\\
    (\bullet(\bullet((\bullet\bullet)(\bullet(\bullet\bullet)))))&\leq 0.2557124\leq 0.2557527\\
    (\bullet(\bullet(\bullet((\bullet\bullet)(\bullet\bullet)))))&\leq 0.1689918\leq 0.1690628\\
    (\bullet(\bullet(\bullet(\bullet(\bullet(\bullet\bullet))))))&\leq 1.0000001\leq 1.0000008^*\\
    (((\bullet\bullet)(\bullet\bullet))((\bullet\bullet)(\bullet\bullet)))&\leq 0.050619\red{\leq} 0.0506189^*\\
    (((\bullet\bullet)(\bullet\bullet))(\bullet(\bullet(\bullet\bullet))))&\leq 0.1349838\leq 0.134984\\
    ((\bullet(\bullet(\bullet\bullet)))(\bullet(\bullet(\bullet\bullet))))&\leq 0.2734403\red{\leq} 0.2734386\\
    ((\bullet(\bullet\bullet))((\bullet\bullet)(\bullet(\bullet\bullet))))&\leq 0.2939654\red{\leq} 0.2939636\\
    ((\bullet(\bullet\bullet))(\bullet((\bullet\bullet)(\bullet\bullet))))&\leq 0.1083039\leq 0.1083669\\
    ((\bullet(\bullet\bullet))(\bullet(\bullet(\bullet(\bullet\bullet)))))&\leq 0.4375042\red{\leq} 0.437501\\
    ((\bullet\bullet)((\bullet(\bullet\bullet))(\bullet(\bullet\bullet))))&\leq 0.1092112\leq 0.1092115\\
    ((\bullet\bullet)((\bullet\bullet)((\bullet\bullet)(\bullet\bullet))))&\leq 0.0702087\leq 0.0702093\\
    ((\bullet\bullet)((\bullet\bullet)(\bullet(\bullet(\bullet\bullet)))))&\leq 0.1479366\red{\leq} 0.1479356\\
    ((\bullet\bullet)(\bullet((\bullet\bullet)(\bullet(\bullet\bullet)))))&\leq 0.1117998\leq 0.1118649\\
    ((\bullet\bullet)(\bullet(\bullet((\bullet\bullet)(\bullet\bullet)))))&\leq 0.0607014\leq 0.0608194\\
    ((\bullet\bullet)(\bullet(\bullet(\bullet(\bullet(\bullet\bullet))))))&\leq 0.315595\leq 0.3155953\\
    (\bullet((\bullet(\bullet\bullet))((\bullet\bullet)(\bullet\bullet))))&\leq 0.1006769\leq 0.10068\\
    (\bullet((\bullet(\bullet\bullet))(\bullet(\bullet(\bullet\bullet)))))&\leq 0.218171\leq 0.2181722\\
    (\bullet((\bullet\bullet)((\bullet\bullet)(\bullet(\bullet\bullet)))))&\leq 0.1311494\leq 0.131184\\
    (\bullet((\bullet\bullet)(\bullet((\bullet\bullet)(\bullet\bullet)))))&\leq 0.0445721\leq 0.0449494\\
    (\bullet((\bullet\bullet)(\bullet(\bullet(\bullet(\bullet\bullet))))))&\leq 0.1566008\leq 0.1567041\\
    (\bullet(\bullet((\bullet(\bullet\bullet))(\bullet(\bullet\bullet)))))&\leq 0.110595\leq 0.1105952\\
    (\bullet(\bullet((\bullet\bullet)((\bullet\bullet)(\bullet\bullet)))))&\leq 0.0781389\leq 0.0781907\\
    (\bullet(\bullet((\bullet\bullet)(\bullet(\bullet(\bullet\bullet))))))&\leq 0.1532779\leq 0.1533452\\
    (\bullet(\bullet(\bullet((\bullet\bullet)(\bullet(\bullet\bullet))))))&\leq 0.2231186\leq 0.2232376\\
    (\bullet(\bullet(\bullet(\bullet((\bullet\bullet)(\bullet\bullet))))))&\leq 0.156537\leq 0.1568862\\
    (\bullet(\bullet(\bullet(\bullet(\bullet(\bullet(\bullet\bullet)))))))&\leq 1.0\leq 1.0000012^*\\
    (((\bullet\bullet)(\bullet\bullet))((\bullet\bullet)(\bullet(\bullet\bullet))))&\leq 0.1411782\leq 0.1411783^*\\
    (((\bullet\bullet)(\bullet\bullet))(\bullet((\bullet\bullet)(\bullet\bullet))))&\leq 0.0369633\leq 0.0370008\\
    (((\bullet\bullet)(\bullet\bullet))(\bullet(\bullet(\bullet(\bullet\bullet)))))&\leq 0.1115373\leq 0.1115422\\
    ((\bullet(\bullet(\bullet\bullet)))((\bullet\bullet)(\bullet(\bullet\bullet))))&\leq 0.1922642\leq 0.1922674\\
    ((\bullet(\bullet(\bullet\bullet)))(\bullet((\bullet\bullet)(\bullet\bullet))))&\leq 0.0951453\leq 0.0952781\\
    ((\bullet(\bullet(\bullet\bullet)))(\bullet(\bullet(\bullet(\bullet\bullet)))))&\leq 0.4921905\red{\leq} 0.4921887\\
    ((\bullet(\bullet\bullet))((\bullet(\bullet\bullet))(\bullet(\bullet\bullet))))&\leq 0.1063606\leq 0.106361\\
    ((\bullet(\bullet\bullet))((\bullet\bullet)((\bullet\bullet)(\bullet\bullet))))&\leq 0.0683775\leq 0.0683789\\
    ((\bullet(\bullet\bullet))((\bullet\bullet)(\bullet(\bullet(\bullet\bullet)))))&\leq 0.1538092\red{\leq} 0.1538091\\
    ((\bullet(\bullet\bullet))(\bullet((\bullet\bullet)(\bullet(\bullet\bullet)))))&\leq 0.1120877\leq 0.1124398\\
    ((\bullet(\bullet\bullet))(\bullet(\bullet((\bullet\bullet)(\bullet\bullet)))))&\leq 0.0630144\leq 0.0634229\\
    ((\bullet(\bullet\bullet))(\bullet(\bullet(\bullet(\bullet(\bullet\bullet))))))&\leq 0.3281263\red{\leq} 0.3281259\\
    ((\bullet\bullet)((\bullet(\bullet\bullet))((\bullet\bullet)(\bullet\bullet))))&\leq 0.0767213\leq 0.0767225\\
    ((\bullet\bullet)((\bullet(\bullet\bullet))(\bullet(\bullet(\bullet\bullet)))))&\leq 0.1677236\leq 0.167724\\
    ((\bullet\bullet)((\bullet\bullet)((\bullet\bullet)(\bullet(\bullet\bullet)))))&\leq 0.0835942\leq 0.0835972\\
    ((\bullet\bullet)((\bullet\bullet)(\bullet((\bullet\bullet)(\bullet\bullet)))))&\leq 0.0285892\leq 0.0287976\\
    ((\bullet\bullet)((\bullet\bullet)(\bullet(\bullet(\bullet(\bullet\bullet))))))&\leq 0.1060242\leq 0.1060431\\
    ((\bullet\bullet)(\bullet((\bullet(\bullet\bullet))(\bullet(\bullet\bullet)))))&\leq 0.0456309\leq 0.0457598\\
    ((\bullet\bullet)(\bullet((\bullet\bullet)((\bullet\bullet)(\bullet\bullet)))))&\leq 0.0334688\leq 0.0334714\\
    ((\bullet\bullet)(\bullet((\bullet\bullet)(\bullet(\bullet(\bullet\bullet))))))&\leq 0.0657211\leq 0.0658082\\
    ((\bullet\bullet)(\bullet(\bullet((\bullet\bullet)(\bullet(\bullet\bullet))))))&\leq 0.0786251\leq 0.0787129\\
    ((\bullet\bullet)(\bullet(\bullet(\bullet((\bullet\bullet)(\bullet\bullet))))))&\leq 0.0519508\leq 0.0521146\\
    ((\bullet\bullet)(\bullet(\bullet(\bullet(\bullet(\bullet(\bullet\bullet)))))))&\leq 0.3066941\leq 0.3066951\\
    (\bullet(((\bullet\bullet)(\bullet\bullet))((\bullet\bullet)(\bullet\bullet))))&\leq 0.0204504\leq 0.0204549\\
    (\bullet(((\bullet\bullet)(\bullet\bullet))(\bullet(\bullet(\bullet\bullet)))))&\leq 0.0545345\leq 0.0545396\\
    (\bullet((\bullet(\bullet(\bullet\bullet)))(\bullet(\bullet(\bullet\bullet)))))&\leq 0.1069244\leq 0.1069273\\
    (\bullet((\bullet(\bullet\bullet))((\bullet\bullet)(\bullet(\bullet\bullet)))))&\leq 0.1214086\leq 0.1214095\\
    (\bullet((\bullet(\bullet\bullet))(\bullet((\bullet\bullet)(\bullet\bullet)))))&\leq 0.0430079\leq 0.0434258\\
    (\bullet((\bullet(\bullet\bullet))(\bullet(\bullet(\bullet(\bullet\bullet))))))&\leq 0.1723559\leq 0.172359\\
    (\bullet((\bullet\bullet)((\bullet(\bullet\bullet))(\bullet(\bullet\bullet)))))&\leq 0.0539278\leq 0.0539904\\
    (\bullet((\bullet\bullet)((\bullet\bullet)((\bullet\bullet)(\bullet\bullet)))))&\leq 0.0376985\leq 0.0377181\\
    (\bullet((\bullet\bullet)((\bullet\bullet)(\bullet(\bullet(\bullet\bullet))))))&\leq 0.0701662\leq 0.0703284\\
    (\bullet((\bullet\bullet)(\bullet((\bullet\bullet)(\bullet(\bullet\bullet))))))&\leq 0.057187\leq 0.057908\\
    (\bullet((\bullet\bullet)(\bullet(\bullet((\bullet\bullet)(\bullet\bullet))))))&\leq 0.0297153\leq 0.0305371\\
    (\bullet((\bullet\bullet)(\bullet(\bullet(\bullet(\bullet(\bullet\bullet)))))))&\leq 0.1470463\leq 0.147387\\
    (\bullet(\bullet((\bullet(\bullet\bullet))((\bullet\bullet)(\bullet\bullet)))))&\leq 0.0770303\leq 0.0770328\\
    (\bullet(\bullet((\bullet(\bullet\bullet))(\bullet(\bullet(\bullet\bullet))))))&\leq 0.1674006\leq 0.1674007\\
    (\bullet(\bullet((\bullet\bullet)((\bullet\bullet)(\bullet(\bullet\bullet))))))&\leq 0.0976269\leq 0.097758\\
    (\bullet(\bullet((\bullet\bullet)(\bullet((\bullet\bullet)(\bullet\bullet))))))&\leq 0.0347112\leq 0.0358522\\
    (\bullet(\bullet((\bullet\bullet)(\bullet(\bullet(\bullet(\bullet\bullet)))))))&\leq 0.1170884\leq 0.1180212\\
    (\bullet(\bullet(\bullet((\bullet(\bullet\bullet))(\bullet(\bullet\bullet))))))&\leq 0.0969838\leq 0.0969874\\
    (\bullet(\bullet(\bullet((\bullet\bullet)((\bullet\bullet)(\bullet\bullet))))))&\leq 0.0676847\leq 0.0678276\\
    (\bullet(\bullet(\bullet((\bullet\bullet)(\bullet(\bullet(\bullet\bullet)))))))&\leq 0.1320994\leq 0.132266\\
    (\bullet(\bullet(\bullet(\bullet((\bullet\bullet)(\bullet(\bullet\bullet)))))))&\leq 0.2052418\leq 0.205465\\
    (\bullet(\bullet(\bullet(\bullet(\bullet((\bullet\bullet)(\bullet\bullet)))))))&\leq 0.1486806\leq 0.1493087\\
    (\bullet(\bullet(\bullet(\bullet(\bullet(\bullet(\bullet(\bullet\bullet))))))))&\leq 1.000006\red{\leq} 1.0000011^*\\
    (((\bullet\bullet)(\bullet(\bullet\bullet)))((\bullet\bullet)(\bullet(\bullet\bullet))))&\leq 0.1095907\leq 0.1095919^*\\
    (((\bullet\bullet)(\bullet(\bullet\bullet)))(\bullet((\bullet\bullet)(\bullet\bullet))))&\leq 0.0545185\leq 0.054681\\
    (((\bullet\bullet)(\bullet(\bullet\bullet)))(\bullet(\bullet(\bullet(\bullet\bullet)))))&\leq 0.1641412\leq 0.1641494\\
    (((\bullet\bullet)(\bullet\bullet))((\bullet(\bullet\bullet))(\bullet(\bullet\bullet))))&\leq 0.0568157\leq 0.0568163\\
    (((\bullet\bullet)(\bullet\bullet))((\bullet\bullet)((\bullet\bullet)(\bullet\bullet))))&\leq 0.0365269\leq 0.0365296\\
    (((\bullet\bullet)(\bullet\bullet))((\bullet\bullet)(\bullet(\bullet(\bullet\bullet)))))&\leq 0.0720995\leq 0.0721028\\
    (((\bullet\bullet)(\bullet\bullet))(\bullet((\bullet\bullet)(\bullet(\bullet\bullet)))))&\leq 0.0417168\red{\leq} 0.0417147\\
    (((\bullet\bullet)(\bullet\bullet))(\bullet(\bullet((\bullet\bullet)(\bullet\bullet)))))&\leq 0.020993\leq 0.0212709\\
    (((\bullet\bullet)(\bullet\bullet))(\bullet(\bullet(\bullet(\bullet(\bullet\bullet))))))&\leq 0.1075\leq 0.1075031\\
    ((\bullet((\bullet\bullet)(\bullet\bullet)))(\bullet((\bullet\bullet)(\bullet\bullet))))&\leq 0.0150657\leq 0.0151785\\
    ((\bullet((\bullet\bullet)(\bullet\bullet)))(\bullet(\bullet(\bullet(\bullet\bullet)))))&\leq 0.0660705\leq 0.0677592\\
    ((\bullet(\bullet(\bullet(\bullet\bullet))))(\bullet(\bullet(\bullet(\bullet\bullet)))))&\leq 0.2460954\red{\leq} 0.2460945\\
    ((\bullet(\bullet(\bullet\bullet)))((\bullet(\bullet\bullet))(\bullet(\bullet\bullet))))&\leq 0.0813881\leq 0.081396\\
    ((\bullet(\bullet(\bullet\bullet)))((\bullet\bullet)((\bullet\bullet)(\bullet\bullet))))&\leq 0.0523329\red{\leq} 0.0523192\\
    ((\bullet(\bullet(\bullet\bullet)))((\bullet\bullet)(\bullet(\bullet(\bullet\bullet)))))&\leq 0.1212067\leq 0.1212832\\
    ((\bullet(\bullet(\bullet\bullet)))(\bullet((\bullet\bullet)(\bullet(\bullet\bullet)))))&\leq 0.1075205\leq 0.107755\\
    ((\bullet(\bullet(\bullet\bullet)))(\bullet(\bullet((\bullet\bullet)(\bullet\bullet)))))&\leq 0.070611\leq 0.071158\\
    ((\bullet(\bullet(\bullet\bullet)))(\bullet(\bullet(\bullet(\bullet(\bullet\bullet))))))&\leq 0.4101601\red{\leq} 0.4101573\\
    ((\bullet(\bullet\bullet))((\bullet(\bullet\bullet))((\bullet\bullet)(\bullet\bullet))))&\leq 0.0673361\red{\leq} 0.067335\\
    ((\bullet(\bullet\bullet))((\bullet(\bullet\bullet))(\bullet(\bullet(\bullet\bullet)))))&\leq 0.151672\leq 0.151672\\
    ((\bullet(\bullet\bullet))((\bullet\bullet)((\bullet\bullet)(\bullet(\bullet\bullet)))))&\leq 0.0700246\leq 0.0700337\\
    ((\bullet(\bullet\bullet))((\bullet\bullet)(\bullet((\bullet\bullet)(\bullet\bullet)))))&\leq 0.024991\leq 0.0260398\\
    ((\bullet(\bullet\bullet))((\bullet\bullet)(\bullet(\bullet(\bullet(\bullet\bullet))))))&\leq 0.0958826\leq 0.0959416\\
    ((\bullet(\bullet\bullet))(\bullet((\bullet(\bullet\bullet))(\bullet(\bullet\bullet)))))&\leq 0.0401573\leq 0.0403922\\
    ((\bullet(\bullet\bullet))(\bullet((\bullet\bullet)((\bullet\bullet)(\bullet\bullet)))))&\leq 0.0291906\leq 0.0292522\\
    ((\bullet(\bullet\bullet))(\bullet((\bullet\bullet)(\bullet(\bullet(\bullet\bullet))))))&\leq 0.0580703\leq 0.0582923\\
    ((\bullet(\bullet\bullet))(\bullet(\bullet((\bullet\bullet)(\bullet(\bullet\bullet))))))&\leq 0.0710672\leq 0.0713126\\
    ((\bullet(\bullet\bullet))(\bullet(\bullet(\bullet((\bullet\bullet)(\bullet\bullet))))))&\leq 0.047099\leq 0.0475895\\
    ((\bullet(\bullet\bullet))(\bullet(\bullet(\bullet(\bullet(\bullet(\bullet\bullet)))))))&\leq 0.2773396\leq 0.27734\\
    ((\bullet\bullet)(((\bullet\bullet)(\bullet\bullet))((\bullet\bullet)(\bullet\bullet))))&\leq 0.0157936\leq 0.0158019\\
    ((\bullet\bullet)(((\bullet\bullet)(\bullet\bullet))(\bullet(\bullet(\bullet\bullet)))))&\leq 0.0421139\red{\leq} 0.0421076\\
    ((\bullet\bullet)((\bullet(\bullet(\bullet\bullet)))(\bullet(\bullet(\bullet\bullet)))))&\leq 0.0825958\leq 0.0825965\\
    ((\bullet\bullet)((\bullet(\bullet\bullet))((\bullet\bullet)(\bullet(\bullet\bullet)))))&\leq 0.0916893\leq 0.0916936\\
    ((\bullet\bullet)((\bullet(\bullet\bullet))(\bullet((\bullet\bullet)(\bullet\bullet)))))&\leq 0.0330448\leq 0.0333551\\
    ((\bullet\bullet)((\bullet(\bullet\bullet))(\bullet(\bullet(\bullet(\bullet\bullet))))))&\leq 0.1321533\leq 0.1321539\\
    ((\bullet\bullet)((\bullet\bullet)((\bullet(\bullet\bullet))(\bullet(\bullet\bullet)))))&\leq 0.0355028\leq 0.0355186\\
    ((\bullet\bullet)((\bullet\bullet)((\bullet\bullet)((\bullet\bullet)(\bullet\bullet)))))&\leq 0.0230787\red{\leq} 0.0230724\\
    ((\bullet\bullet)((\bullet\bullet)((\bullet\bullet)(\bullet(\bullet(\bullet\bullet))))))&\leq 0.0447242\leq 0.04474\\
    ((\bullet\bullet)((\bullet\bullet)(\bullet((\bullet\bullet)(\bullet(\bullet\bullet))))))&\leq 0.0357857\leq 0.0360811\\
    ((\bullet\bullet)((\bullet\bullet)(\bullet(\bullet((\bullet\bullet)(\bullet\bullet))))))&\leq 0.0188528\leq 0.0196592\\
    ((\bullet\bullet)((\bullet\bullet)(\bullet(\bullet(\bullet(\bullet(\bullet\bullet)))))))&\leq 0.09544\leq 0.0954771\\
    ((\bullet\bullet)(\bullet((\bullet(\bullet\bullet))((\bullet\bullet)(\bullet\bullet)))))&\leq 0.031395\leq 0.0314173\\
    ((\bullet\bullet)(\bullet((\bullet(\bullet\bullet))(\bullet(\bullet(\bullet\bullet))))))&\leq 0.0679941\leq 0.068037\\
    ((\bullet\bullet)(\bullet((\bullet\bullet)((\bullet\bullet)(\bullet(\bullet\bullet))))))&\leq 0.0413824\leq 0.0415641\\
    ((\bullet\bullet)(\bullet((\bullet\bullet)(\bullet((\bullet\bullet)(\bullet\bullet))))))&\leq 0.0144759\leq 0.0165141\\
    ((\bullet\bullet)(\bullet((\bullet\bullet)(\bullet(\bullet(\bullet(\bullet\bullet)))))))&\leq 0.0490942\leq 0.050696\\
    ((\bullet\bullet)(\bullet(\bullet((\bullet(\bullet\bullet))(\bullet(\bullet\bullet))))))&\leq 0.0334156\leq 0.0334726\\
    ((\bullet\bullet)(\bullet(\bullet((\bullet\bullet)((\bullet\bullet)(\bullet\bullet))))))&\leq 0.0236896\leq 0.0239339\\
    ((\bullet\bullet)(\bullet(\bullet((\bullet\bullet)(\bullet(\bullet(\bullet\bullet)))))))&\leq 0.0465471\leq 0.0470627\\
    ((\bullet\bullet)(\bullet(\bullet(\bullet((\bullet\bullet)(\bullet(\bullet\bullet)))))))&\leq 0.0674928\leq 0.0677231\\
    ((\bullet\bullet)(\bullet(\bullet(\bullet(\bullet((\bullet\bullet)(\bullet\bullet)))))))&\leq 0.0474298\leq 0.0478606\\
    ((\bullet\bullet)(\bullet(\bullet(\bullet(\bullet(\bullet(\bullet(\bullet\bullet))))))))&\leq 0.3020641\leq 0.3020658\\
    (\bullet(((\bullet\bullet)(\bullet\bullet))((\bullet\bullet)(\bullet(\bullet\bullet)))))&\leq 0.0558554\leq 0.0558602\\
    (\bullet(((\bullet\bullet)(\bullet\bullet))(\bullet((\bullet\bullet)(\bullet\bullet)))))&\leq 0.0145222\leq 0.014862\\
    (\bullet(((\bullet\bullet)(\bullet\bullet))(\bullet(\bullet(\bullet(\bullet\bullet))))))&\leq 0.043377\leq 0.0434295\\
    (\bullet((\bullet(\bullet(\bullet\bullet)))((\bullet\bullet)(\bullet(\bullet\bullet)))))&\leq 0.0756829\leq 0.0757179\\
    (\bullet((\bullet(\bullet(\bullet\bullet)))(\bullet((\bullet\bullet)(\bullet\bullet)))))&\leq 0.0371911\leq 0.0377163\\
    (\bullet((\bullet(\bullet(\bullet\bullet)))(\bullet(\bullet(\bullet(\bullet\bullet))))))&\leq 0.1909558\leq 0.1909597\\
    (\bullet((\bullet(\bullet\bullet))((\bullet(\bullet\bullet))(\bullet(\bullet\bullet)))))&\leq 0.0441847\leq 0.0442173\\
    (\bullet((\bullet(\bullet\bullet))((\bullet\bullet)((\bullet\bullet)(\bullet\bullet)))))&\leq 0.0285515\red{\leq} 0.0285476\\
    (\bullet((\bullet(\bullet\bullet))((\bullet\bullet)(\bullet(\bullet(\bullet\bullet))))))&\leq 0.0613023\leq 0.0613409\\
    (\bullet((\bullet(\bullet\bullet))(\bullet((\bullet\bullet)(\bullet(\bullet\bullet))))))&\leq 0.0442962\leq 0.0457139\\
    (\bullet((\bullet(\bullet\bullet))(\bullet(\bullet((\bullet\bullet)(\bullet\bullet))))))&\leq 0.0248663\leq 0.0260477\\
    (\bullet((\bullet(\bullet\bullet))(\bullet(\bullet(\bullet(\bullet(\bullet\bullet)))))))&\leq 0.1282351\leq 0.128252\\
    (\bullet((\bullet\bullet)((\bullet(\bullet\bullet))((\bullet\bullet)(\bullet\bullet)))))&\leq 0.036057\leq 0.0360871\\
    (\bullet((\bullet\bullet)((\bullet(\bullet\bullet))(\bullet(\bullet(\bullet\bullet))))))&\leq 0.0725516\leq 0.0725724\\
    (\bullet((\bullet\bullet)((\bullet\bullet)((\bullet\bullet)(\bullet(\bullet\bullet))))))&\leq 0.0453039\leq 0.0455103\\
    (\bullet((\bullet\bullet)((\bullet\bullet)(\bullet((\bullet\bullet)(\bullet\bullet))))))&\leq 0.0151569\leq 0.0168906\\
    (\bullet((\bullet\bullet)((\bullet\bullet)(\bullet(\bullet(\bullet(\bullet\bullet)))))))&\leq 0.0521377\leq 0.0530988\\
    (\bullet((\bullet\bullet)(\bullet((\bullet(\bullet\bullet))(\bullet(\bullet\bullet))))))&\leq 0.0227812\leq 0.023975\\
    (\bullet((\bullet\bullet)(\bullet((\bullet\bullet)((\bullet\bullet)(\bullet\bullet))))))&\leq 0.0171237\leq 0.017879\\
    (\bullet((\bullet\bullet)(\bullet((\bullet\bullet)(\bullet(\bullet(\bullet\bullet)))))))&\leq 0.0337422\leq 0.0367598\\
    (\bullet((\bullet\bullet)(\bullet(\bullet((\bullet\bullet)(\bullet(\bullet\bullet)))))))&\leq 0.0384892\leq 0.0401641\\
    (\bullet((\bullet\bullet)(\bullet(\bullet(\bullet((\bullet\bullet)(\bullet\bullet)))))))&\leq 0.0247273\leq 0.0267558\\
    (\bullet((\bullet\bullet)(\bullet(\bullet(\bullet(\bullet(\bullet(\bullet\bullet))))))))&\leq 0.1430699\leq 0.1438289\\
    (\bullet(\bullet(((\bullet\bullet)(\bullet\bullet))((\bullet\bullet)(\bullet\bullet)))))&\leq 0.015781\leq 0.0157837\\
    (\bullet(\bullet(((\bullet\bullet)(\bullet\bullet))(\bullet(\bullet(\bullet\bullet))))))&\leq 0.0420831\red{\leq} 0.042071\\
    (\bullet(\bullet((\bullet(\bullet(\bullet\bullet)))(\bullet(\bullet(\bullet\bullet))))))&\leq 0.0825765\leq 0.0825765\\
    (\bullet(\bullet((\bullet(\bullet\bullet))((\bullet\bullet)(\bullet(\bullet\bullet))))))&\leq 0.0916172\leq 0.0916186\\
    (\bullet(\bullet((\bullet(\bullet\bullet))(\bullet((\bullet\bullet)(\bullet\bullet))))))&\leq 0.0330789\leq 0.0335285\\
    (\bullet(\bullet((\bullet(\bullet\bullet))(\bullet(\bullet(\bullet(\bullet\bullet)))))))&\leq 0.1321212\leq 0.1321234\\
    (\bullet(\bullet((\bullet\bullet)((\bullet(\bullet\bullet))(\bullet(\bullet\bullet))))))&\leq 0.0395341\leq 0.0397238\\
    (\bullet(\bullet((\bullet\bullet)((\bullet\bullet)((\bullet\bullet)(\bullet\bullet))))))&\leq 0.0281547\leq 0.0282735\\
    (\bullet(\bullet((\bullet\bullet)((\bullet\bullet)(\bullet(\bullet(\bullet\bullet)))))))&\leq 0.0512426\leq 0.0519989\\
    (\bullet(\bullet((\bullet\bullet)(\bullet((\bullet\bullet)(\bullet(\bullet\bullet)))))))&\leq 0.0444578\leq 0.0464682\\
    (\bullet(\bullet((\bullet\bullet)(\bullet(\bullet((\bullet\bullet)(\bullet\bullet)))))))&\leq 0.0229626\leq 0.0258078\\
    (\bullet(\bullet((\bullet\bullet)(\bullet(\bullet(\bullet(\bullet(\bullet\bullet))))))))&\leq 0.1109042\leq 0.1121893\\
    (\bullet(\bullet(\bullet((\bullet(\bullet\bullet))((\bullet\bullet)(\bullet\bullet))))))&\leq 0.0671513\red{\leq} 0.0671494\\
    (\bullet(\bullet(\bullet((\bullet(\bullet\bullet))(\bullet(\bullet(\bullet\bullet)))))))&\leq 0.1459225\leq 0.1459252\\
    (\bullet(\bullet(\bullet((\bullet\bullet)((\bullet\bullet)(\bullet(\bullet\bullet)))))))&\leq 0.084083\leq 0.0845377\\
    (\bullet(\bullet(\bullet((\bullet\bullet)(\bullet((\bullet\bullet)(\bullet\bullet)))))))&\leq 0.0305436\leq 0.0325929\\
    (\bullet(\bullet(\bullet((\bullet\bullet)(\bullet(\bullet(\bullet(\bullet\bullet))))))))&\leq 0.1015992\leq 0.1036481\\
    (\bullet(\bullet(\bullet(\bullet((\bullet(\bullet\bullet))(\bullet(\bullet\bullet)))))))&\leq 0.089068\leq 0.0890797\\
    (\bullet(\bullet(\bullet(\bullet((\bullet\bullet)((\bullet\bullet)(\bullet\bullet)))))))&\leq 0.0618614\leq 0.0623212\\
    (\bullet(\bullet(\bullet(\bullet((\bullet\bullet)(\bullet(\bullet(\bullet\bullet))))))))&\leq 0.1204158\leq 0.1208088\\
    (\bullet(\bullet(\bullet(\bullet(\bullet((\bullet\bullet)(\bullet(\bullet\bullet))))))))&\leq 0.1938207\leq 0.1940385\\
    (\bullet(\bullet(\bullet(\bullet(\bullet(\bullet((\bullet\bullet)(\bullet\bullet))))))))&\leq 0.1432937\leq 0.1442329\\
    (\bullet(\bullet(\bullet(\bullet(\bullet(\bullet(\bullet(\bullet(\bullet\bullet)))))))))&\leq 1.0000037\red{\leq} 1.0000015^*\\
    (((\bullet\bullet)(\bullet(\bullet\bullet)))((\bullet(\bullet\bullet))(\bullet(\bullet\bullet))))&\leq 0.0971236^*\\
    (((\bullet\bullet)(\bullet(\bullet\bullet)))((\bullet\bullet)((\bullet\bullet)(\bullet\bullet))))&\leq 0.062443\\
    (((\bullet\bullet)(\bullet(\bullet\bullet)))((\bullet\bullet)(\bullet(\bullet(\bullet\bullet)))))&\leq 0.1321834\\
    (((\bullet\bullet)(\bullet(\bullet\bullet)))(\bullet((\bullet\bullet)(\bullet(\bullet\bullet)))))&\leq 0.0656563\\
    (((\bullet\bullet)(\bullet(\bullet\bullet)))(\bullet(\bullet((\bullet\bullet)(\bullet\bullet)))))&\leq 0.0309013\\
    (((\bullet\bullet)(\bullet(\bullet\bullet)))(\bullet(\bullet(\bullet(\bullet(\bullet\bullet))))))&\leq 0.1574119\\
    (((\bullet\bullet)(\bullet\bullet))((\bullet(\bullet\bullet))((\bullet\bullet)(\bullet\bullet))))&\leq 0.0329195\\
    (((\bullet\bullet)(\bullet\bullet))((\bullet(\bullet\bullet))(\bullet(\bullet(\bullet\bullet)))))&\leq 0.0660954\\
    (((\bullet\bullet)(\bullet\bullet))((\bullet\bullet)((\bullet\bullet)(\bullet(\bullet\bullet)))))&\leq 0.0312717\\
    (((\bullet\bullet)(\bullet\bullet))((\bullet\bullet)(\bullet((\bullet\bullet)(\bullet\bullet)))))&\leq 0.0112434\\
    (((\bullet\bullet)(\bullet\bullet))((\bullet\bullet)(\bullet(\bullet(\bullet(\bullet\bullet))))))&\leq 0.0396759\\
    (((\bullet\bullet)(\bullet\bullet))(\bullet((\bullet(\bullet\bullet))(\bullet(\bullet\bullet)))))&\leq 0.0155952\\
    (((\bullet\bullet)(\bullet\bullet))(\bullet((\bullet\bullet)((\bullet\bullet)(\bullet\bullet)))))&\leq 0.0112518\\
    (((\bullet\bullet)(\bullet\bullet))(\bullet((\bullet\bullet)(\bullet(\bullet(\bullet\bullet))))))&\leq 0.022193\\
    (((\bullet\bullet)(\bullet\bullet))(\bullet(\bullet((\bullet\bullet)(\bullet(\bullet\bullet))))))&\leq 0.0268531\\
    (((\bullet\bullet)(\bullet\bullet))(\bullet(\bullet(\bullet((\bullet\bullet)(\bullet\bullet))))))&\leq 0.0177505\\
    (((\bullet\bullet)(\bullet\bullet))(\bullet(\bullet(\bullet(\bullet(\bullet(\bullet\bullet)))))))&\leq 0.1045095\\
    ((\bullet((\bullet\bullet)(\bullet\bullet)))((\bullet(\bullet\bullet))(\bullet(\bullet\bullet))))&\leq 0.0226776\\
    ((\bullet((\bullet\bullet)(\bullet\bullet)))((\bullet\bullet)((\bullet\bullet)(\bullet\bullet))))&\leq 0.016968\\
    ((\bullet((\bullet\bullet)(\bullet\bullet)))((\bullet\bullet)(\bullet(\bullet(\bullet\bullet)))))&\leq 0.0280677\\
    ((\bullet((\bullet\bullet)(\bullet\bullet)))(\bullet((\bullet\bullet)(\bullet(\bullet\bullet)))))&\leq 0.0380457\\
    ((\bullet((\bullet\bullet)(\bullet\bullet)))(\bullet(\bullet((\bullet\bullet)(\bullet\bullet)))))&\leq 0.018366\\
    ((\bullet((\bullet\bullet)(\bullet\bullet)))(\bullet(\bullet(\bullet(\bullet(\bullet\bullet))))))&\leq 0.0585165\\
    ((\bullet(\bullet(\bullet(\bullet\bullet))))((\bullet(\bullet\bullet))(\bullet(\bullet\bullet))))&\leq 0.0762186\\
    ((\bullet(\bullet(\bullet(\bullet\bullet))))((\bullet\bullet)((\bullet\bullet)(\bullet\bullet))))&\leq 0.0490018\\
    ((\bullet(\bullet(\bullet(\bullet\bullet))))((\bullet\bullet)(\bullet(\bullet(\bullet\bullet)))))&\leq 0.1106769\\
    ((\bullet(\bullet(\bullet(\bullet\bullet))))(\bullet((\bullet\bullet)(\bullet(\bullet\bullet)))))&\leq 0.0808743\\
    ((\bullet(\bullet(\bullet(\bullet\bullet))))(\bullet(\bullet((\bullet\bullet)(\bullet\bullet)))))&\leq 0.0636934\\
    ((\bullet(\bullet(\bullet(\bullet\bullet))))(\bullet(\bullet(\bullet(\bullet(\bullet\bullet))))))&\leq 0.4511739\\
    ((\bullet(\bullet(\bullet\bullet)))((\bullet(\bullet\bullet))((\bullet\bullet)(\bullet\bullet))))&\leq 0.0583844\\
    ((\bullet(\bullet(\bullet\bullet)))((\bullet(\bullet\bullet))(\bullet(\bullet(\bullet\bullet)))))&\leq 0.1333566\\
    ((\bullet(\bullet(\bullet\bullet)))((\bullet\bullet)((\bullet\bullet)(\bullet(\bullet\bullet)))))&\leq 0.0605023\\
    ((\bullet(\bullet(\bullet\bullet)))((\bullet\bullet)(\bullet((\bullet\bullet)(\bullet\bullet)))))&\leq 0.021838\\
    ((\bullet(\bullet(\bullet\bullet)))((\bullet\bullet)(\bullet(\bullet(\bullet(\bullet\bullet))))))&\leq 0.0847542\\
    ((\bullet(\bullet(\bullet\bullet)))(\bullet((\bullet(\bullet\bullet))(\bullet(\bullet\bullet)))))&\leq 0.0352903\\
    ((\bullet(\bullet(\bullet\bullet)))(\bullet((\bullet\bullet)((\bullet\bullet)(\bullet\bullet)))))&\leq 0.0261751\\
    ((\bullet(\bullet(\bullet\bullet)))(\bullet((\bullet\bullet)(\bullet(\bullet(\bullet\bullet))))))&\leq 0.0533813\\
    ((\bullet(\bullet(\bullet\bullet)))(\bullet(\bullet((\bullet\bullet)(\bullet(\bullet\bullet))))))&\leq 0.072829\\
    ((\bullet(\bullet(\bullet\bullet)))(\bullet(\bullet(\bullet((\bullet\bullet)(\bullet\bullet))))))&\leq 0.0517956\\
    ((\bullet(\bullet(\bullet\bullet)))(\bullet(\bullet(\bullet(\bullet(\bullet(\bullet\bullet)))))))&\leq 0.3222706\\
    ((\bullet(\bullet\bullet))(((\bullet\bullet)(\bullet\bullet))((\bullet\bullet)(\bullet\bullet))))&\leq 0.0136125\\
    ((\bullet(\bullet\bullet))(((\bullet\bullet)(\bullet\bullet))(\bullet(\bullet(\bullet\bullet)))))&\leq 0.0362962\\
    ((\bullet(\bullet\bullet))((\bullet(\bullet(\bullet\bullet)))(\bullet(\bullet(\bullet\bullet)))))&\leq 0.072188\\
    ((\bullet(\bullet\bullet))((\bullet(\bullet\bullet))((\bullet\bullet)(\bullet(\bullet\bullet)))))&\leq 0.0790186\\
    ((\bullet(\bullet\bullet))((\bullet(\bullet\bullet))(\bullet((\bullet\bullet)(\bullet\bullet)))))&\leq 0.0288268\\
    ((\bullet(\bullet\bullet))((\bullet(\bullet\bullet))(\bullet(\bullet(\bullet(\bullet\bullet))))))&\leq 0.1155022\\
    ((\bullet(\bullet\bullet))((\bullet\bullet)((\bullet(\bullet\bullet))(\bullet(\bullet\bullet)))))&\leq 0.0293984\\
    ((\bullet(\bullet\bullet))((\bullet\bullet)((\bullet\bullet)((\bullet\bullet)(\bullet\bullet)))))&\leq 0.0189046\\
    ((\bullet(\bullet\bullet))((\bullet\bullet)((\bullet\bullet)(\bullet(\bullet(\bullet\bullet))))))&\leq 0.0390967\\
    ((\bullet(\bullet\bullet))((\bullet\bullet)(\bullet((\bullet\bullet)(\bullet(\bullet\bullet))))))&\leq 0.0300035\\
    ((\bullet(\bullet\bullet))((\bullet\bullet)(\bullet(\bullet((\bullet\bullet)(\bullet\bullet))))))&\leq 0.0162861\\
    ((\bullet(\bullet\bullet))((\bullet\bullet)(\bullet(\bullet(\bullet(\bullet(\bullet\bullet)))))))&\leq 0.0833444\\
    ((\bullet(\bullet\bullet))(\bullet((\bullet(\bullet\bullet))((\bullet\bullet)(\bullet\bullet)))))&\leq 0.0266453\\
    ((\bullet(\bullet\bullet))(\bullet((\bullet(\bullet\bullet))(\bullet(\bullet(\bullet\bullet))))))&\leq 0.0576713\\
    ((\bullet(\bullet\bullet))(\bullet((\bullet\bullet)((\bullet\bullet)(\bullet(\bullet\bullet))))))&\leq 0.0348056\\
    ((\bullet(\bullet\bullet))(\bullet((\bullet\bullet)(\bullet((\bullet\bullet)(\bullet\bullet))))))&\leq 0.01262\\
    ((\bullet(\bullet\bullet))(\bullet((\bullet\bullet)(\bullet(\bullet(\bullet(\bullet\bullet)))))))&\leq 0.0417166\\
    ((\bullet(\bullet\bullet))(\bullet(\bullet((\bullet(\bullet\bullet))(\bullet(\bullet\bullet))))))&\leq 0.0292303\\
    ((\bullet(\bullet\bullet))(\bullet(\bullet((\bullet\bullet)((\bullet\bullet)(\bullet\bullet))))))&\leq 0.0207011\\
    ((\bullet(\bullet\bullet))(\bullet(\bullet((\bullet\bullet)(\bullet(\bullet(\bullet\bullet)))))))&\leq 0.0406586\\
    ((\bullet(\bullet\bullet))(\bullet(\bullet(\bullet((\bullet\bullet)(\bullet(\bullet\bullet)))))))&\leq 0.0590266\\
    ((\bullet(\bullet\bullet))(\bullet(\bullet(\bullet(\bullet((\bullet\bullet)(\bullet\bullet)))))))&\leq 0.0415539\\
    ((\bullet(\bullet\bullet))(\bullet(\bullet(\bullet(\bullet(\bullet(\bullet(\bullet\bullet))))))))&\leq 0.2640024\\
    ((\bullet\bullet)(((\bullet\bullet)(\bullet\bullet))((\bullet\bullet)(\bullet(\bullet\bullet)))))&\leq 0.0430452\\
    ((\bullet\bullet)(((\bullet\bullet)(\bullet\bullet))(\bullet((\bullet\bullet)(\bullet\bullet)))))&\leq 0.0111632\\
    ((\bullet\bullet)(((\bullet\bullet)(\bullet\bullet))(\bullet(\bullet(\bullet(\bullet\bullet))))))&\leq 0.0334361\\
    ((\bullet\bullet)((\bullet(\bullet(\bullet\bullet)))((\bullet\bullet)(\bullet(\bullet\bullet)))))&\leq 0.0581597\\
    ((\bullet\bullet)((\bullet(\bullet(\bullet\bullet)))(\bullet((\bullet\bullet)(\bullet\bullet)))))&\leq 0.0285765\\
    ((\bullet\bullet)((\bullet(\bullet(\bullet\bullet)))(\bullet(\bullet(\bullet(\bullet\bullet))))))&\leq 0.147039\\
    ((\bullet\bullet)((\bullet(\bullet\bullet))((\bullet(\bullet\bullet))(\bullet(\bullet\bullet)))))&\leq 0.0328586\\
    ((\bullet\bullet)((\bullet(\bullet\bullet))((\bullet\bullet)((\bullet\bullet)(\bullet\bullet)))))&\leq 0.0211295\\
    ((\bullet\bullet)((\bullet(\bullet\bullet))((\bullet\bullet)(\bullet(\bullet(\bullet\bullet))))))&\leq 0.0459686\\
    ((\bullet\bullet)((\bullet(\bullet\bullet))(\bullet((\bullet\bullet)(\bullet(\bullet\bullet))))))&\leq 0.0337921\\
    ((\bullet\bullet)((\bullet(\bullet\bullet))(\bullet(\bullet((\bullet\bullet)(\bullet\bullet))))))&\leq 0.0190919\\
    ((\bullet\bullet)((\bullet(\bullet\bullet))(\bullet(\bullet(\bullet(\bullet(\bullet\bullet)))))))&\leq 0.0980288\\
    ((\bullet\bullet)((\bullet\bullet)((\bullet(\bullet\bullet))((\bullet\bullet)(\bullet\bullet)))))&\leq 0.0240934\\
    ((\bullet\bullet)((\bullet\bullet)((\bullet(\bullet\bullet))(\bullet(\bullet(\bullet\bullet))))))&\leq 0.0501322\\
    ((\bullet\bullet)((\bullet\bullet)((\bullet\bullet)((\bullet\bullet)(\bullet(\bullet\bullet))))))&\leq 0.0276363\\
    ((\bullet\bullet)((\bullet\bullet)((\bullet\bullet)(\bullet((\bullet\bullet)(\bullet\bullet))))))&\leq 0.0094195\\
    ((\bullet\bullet)((\bullet\bullet)((\bullet\bullet)(\bullet(\bullet(\bullet(\bullet\bullet)))))))&\leq 0.0323844\\
    ((\bullet\bullet)((\bullet\bullet)(\bullet((\bullet(\bullet\bullet))(\bullet(\bullet\bullet))))))&\leq 0.0144613\\
    ((\bullet\bullet)((\bullet\bullet)(\bullet((\bullet\bullet)((\bullet\bullet)(\bullet\bullet))))))&\leq 0.0106603\\
    ((\bullet\bullet)((\bullet\bullet)(\bullet((\bullet\bullet)(\bullet(\bullet(\bullet\bullet)))))))&\leq 0.0207861\\
    ((\bullet\bullet)((\bullet\bullet)(\bullet(\bullet((\bullet\bullet)(\bullet(\bullet\bullet)))))))&\leq 0.0242244\\
    ((\bullet\bullet)((\bullet\bullet)(\bullet(\bullet(\bullet((\bullet\bullet)(\bullet\bullet)))))))&\leq 0.0158174\\
    ((\bullet\bullet)((\bullet\bullet)(\bullet(\bullet(\bullet(\bullet(\bullet(\bullet\bullet))))))))&\leq 0.0917822\\
    ((\bullet\bullet)(\bullet(((\bullet\bullet)(\bullet\bullet))((\bullet\bullet)(\bullet\bullet)))))&\leq 0.0063735\\
    ((\bullet\bullet)(\bullet(((\bullet\bullet)(\bullet\bullet))(\bullet(\bullet(\bullet\bullet))))))&\leq 0.0169653\\
    ((\bullet\bullet)(\bullet((\bullet(\bullet(\bullet\bullet)))(\bullet(\bullet(\bullet\bullet))))))&\leq 0.0332716\\
    ((\bullet\bullet)(\bullet((\bullet(\bullet\bullet))((\bullet\bullet)(\bullet(\bullet\bullet))))))&\leq 0.0373028\\
    ((\bullet\bullet)(\bullet((\bullet(\bullet\bullet))(\bullet((\bullet\bullet)(\bullet\bullet))))))&\leq 0.0135754\\
    ((\bullet\bullet)(\bullet((\bullet(\bullet\bullet))(\bullet(\bullet(\bullet(\bullet\bullet)))))))&\leq 0.0533437\\
    ((\bullet\bullet)(\bullet((\bullet\bullet)((\bullet(\bullet\bullet))(\bullet(\bullet\bullet))))))&\leq 0.0167977\\
    ((\bullet\bullet)(\bullet((\bullet\bullet)((\bullet\bullet)((\bullet\bullet)(\bullet\bullet))))))&\leq 0.0118462\\
    ((\bullet\bullet)(\bullet((\bullet\bullet)((\bullet\bullet)(\bullet(\bullet(\bullet\bullet)))))))&\leq 0.0218306\\
    ((\bullet\bullet)(\bullet((\bullet\bullet)(\bullet((\bullet\bullet)(\bullet(\bullet\bullet)))))))&\leq 0.0185545\\
    ((\bullet\bullet)(\bullet((\bullet\bullet)(\bullet(\bullet((\bullet\bullet)(\bullet\bullet)))))))&\leq 0.01\\
    ((\bullet\bullet)(\bullet((\bullet\bullet)(\bullet(\bullet(\bullet(\bullet(\bullet\bullet))))))))&\leq 0.0460062\\
    ((\bullet\bullet)(\bullet(\bullet((\bullet(\bullet\bullet))((\bullet\bullet)(\bullet\bullet))))))&\leq 0.0230307\\
    ((\bullet\bullet)(\bullet(\bullet((\bullet(\bullet\bullet))(\bullet(\bullet(\bullet\bullet)))))))&\leq 0.0500705\\
    ((\bullet\bullet)(\bullet(\bullet((\bullet\bullet)((\bullet\bullet)(\bullet(\bullet\bullet)))))))&\leq 0.0293239\\
    ((\bullet\bullet)(\bullet(\bullet((\bullet\bullet)(\bullet((\bullet\bullet)(\bullet\bullet)))))))&\leq 0.0111047\\
    ((\bullet\bullet)(\bullet(\bullet((\bullet\bullet)(\bullet(\bullet(\bullet(\bullet\bullet))))))))&\leq 0.0357995\\
    ((\bullet\bullet)(\bullet(\bullet(\bullet((\bullet(\bullet\bullet))(\bullet(\bullet\bullet)))))))&\leq 0.0290179\\
    ((\bullet\bullet)(\bullet(\bullet(\bullet((\bullet\bullet)((\bullet\bullet)(\bullet\bullet)))))))&\leq 0.020317\\
    ((\bullet\bullet)(\bullet(\bullet(\bullet((\bullet\bullet)(\bullet(\bullet(\bullet\bullet))))))))&\leq 0.0396144\\
    ((\bullet\bullet)(\bullet(\bullet(\bullet(\bullet((\bullet\bullet)(\bullet(\bullet\bullet))))))))&\leq 0.0614394\\
    ((\bullet\bullet)(\bullet(\bullet(\bullet(\bullet(\bullet((\bullet\bullet)(\bullet\bullet))))))))&\leq 0.0446234\\
    ((\bullet\bullet)(\bullet(\bullet(\bullet(\bullet(\bullet(\bullet(\bullet(\bullet\bullet)))))))))&\leq 0.298744\\
    (\bullet(((\bullet\bullet)(\bullet(\bullet\bullet)))((\bullet\bullet)(\bullet(\bullet\bullet)))))&\leq 0.0427034\\
    (\bullet(((\bullet\bullet)(\bullet(\bullet\bullet)))(\bullet((\bullet\bullet)(\bullet\bullet)))))&\leq 0.0212429\\
    (\bullet(((\bullet\bullet)(\bullet(\bullet\bullet)))(\bullet(\bullet(\bullet(\bullet\bullet))))))&\leq 0.063534\\
    (\bullet(((\bullet\bullet)(\bullet\bullet))((\bullet(\bullet\bullet))(\bullet(\bullet\bullet)))))&\leq 0.0222191\\
    (\bullet(((\bullet\bullet)(\bullet\bullet))((\bullet\bullet)((\bullet\bullet)(\bullet\bullet)))))&\leq 0.0142899\\
    (\bullet(((\bullet\bullet)(\bullet\bullet))((\bullet\bullet)(\bullet(\bullet(\bullet\bullet))))))&\leq 0.0279132\\
    (\bullet(((\bullet\bullet)(\bullet\bullet))(\bullet((\bullet\bullet)(\bullet(\bullet\bullet))))))&\leq 0.0162353\\
    (\bullet(((\bullet\bullet)(\bullet\bullet))(\bullet(\bullet((\bullet\bullet)(\bullet\bullet))))))&\leq 0.0083416\\
    (\bullet(((\bullet\bullet)(\bullet\bullet))(\bullet(\bullet(\bullet(\bullet(\bullet\bullet)))))))&\leq 0.0415536\\
    (\bullet((\bullet((\bullet\bullet)(\bullet\bullet)))(\bullet((\bullet\bullet)(\bullet\bullet)))))&\leq 0.0059523\\
    (\bullet((\bullet((\bullet\bullet)(\bullet\bullet)))(\bullet(\bullet(\bullet(\bullet\bullet))))))&\leq 0.0258215\\
    (\bullet((\bullet(\bullet(\bullet(\bullet\bullet))))(\bullet(\bullet(\bullet(\bullet\bullet))))))&\leq 0.0949084\\
    (\bullet((\bullet(\bullet(\bullet\bullet)))((\bullet(\bullet\bullet))(\bullet(\bullet\bullet)))))&\leq 0.0324341\\
    (\bullet((\bullet(\bullet(\bullet\bullet)))((\bullet\bullet)((\bullet\bullet)(\bullet\bullet)))))&\leq 0.0207928\\
    (\bullet((\bullet(\bullet(\bullet\bullet)))((\bullet\bullet)(\bullet(\bullet(\bullet\bullet))))))&\leq 0.0470275\\
    (\bullet((\bullet(\bullet(\bullet\bullet)))(\bullet((\bullet\bullet)(\bullet(\bullet\bullet))))))&\leq 0.0417464\\
    (\bullet((\bullet(\bullet(\bullet\bullet)))(\bullet(\bullet((\bullet\bullet)(\bullet\bullet))))))&\leq 0.0275483\\
    (\bullet((\bullet(\bullet(\bullet\bullet)))(\bullet(\bullet(\bullet(\bullet(\bullet\bullet)))))))&\leq 0.1582911\\
    (\bullet((\bullet(\bullet\bullet))((\bullet(\bullet\bullet))((\bullet\bullet)(\bullet\bullet)))))&\leq 0.0286278\\
    (\bullet((\bullet(\bullet\bullet))((\bullet(\bullet\bullet))(\bullet(\bullet(\bullet\bullet))))))&\leq 0.0599812\\
    (\bullet((\bullet(\bullet\bullet))((\bullet\bullet)((\bullet\bullet)(\bullet(\bullet\bullet))))))&\leq 0.0303138\\
    (\bullet((\bullet(\bullet\bullet))((\bullet\bullet)(\bullet((\bullet\bullet)(\bullet\bullet))))))&\leq 0.0111186\\
    (\bullet((\bullet(\bullet\bullet))((\bullet\bullet)(\bullet(\bullet(\bullet(\bullet\bullet)))))))&\leq 0.0383284\\
    (\bullet((\bullet(\bullet\bullet))(\bullet((\bullet(\bullet\bullet))(\bullet(\bullet\bullet))))))&\leq 0.016262\\
    (\bullet((\bullet(\bullet\bullet))(\bullet((\bullet\bullet)((\bullet\bullet)(\bullet\bullet))))))&\leq 0.0116694\\
    (\bullet((\bullet(\bullet\bullet))(\bullet((\bullet\bullet)(\bullet(\bullet(\bullet\bullet)))))))&\leq 0.0235487\\
    (\bullet((\bullet(\bullet\bullet))(\bullet(\bullet((\bullet\bullet)(\bullet(\bullet\bullet)))))))&\leq 0.0280953\\
    (\bullet((\bullet(\bullet\bullet))(\bullet(\bullet(\bullet((\bullet\bullet)(\bullet\bullet)))))))&\leq 0.0188036\\
    (\bullet((\bullet(\bullet\bullet))(\bullet(\bullet(\bullet(\bullet(\bullet(\bullet\bullet))))))))&\leq 0.1085368\\
    (\bullet((\bullet\bullet)(((\bullet\bullet)(\bullet\bullet))((\bullet\bullet)(\bullet\bullet)))))&\leq 0.0069791\\
    (\bullet((\bullet\bullet)(((\bullet\bullet)(\bullet\bullet))(\bullet(\bullet(\bullet\bullet))))))&\leq 0.0185711\\
    (\bullet((\bullet\bullet)((\bullet(\bullet(\bullet\bullet)))(\bullet(\bullet(\bullet\bullet))))))&\leq 0.0342061\\
    (\bullet((\bullet\bullet)((\bullet(\bullet\bullet))((\bullet\bullet)(\bullet(\bullet\bullet))))))&\leq 0.0425295\\
    (\bullet((\bullet\bullet)((\bullet(\bullet\bullet))(\bullet((\bullet\bullet)(\bullet\bullet))))))&\leq 0.014193\\
    (\bullet((\bullet\bullet)((\bullet(\bullet\bullet))(\bullet(\bullet(\bullet(\bullet\bullet)))))))&\leq 0.0554156\\
    (\bullet((\bullet\bullet)((\bullet\bullet)((\bullet(\bullet\bullet))(\bullet(\bullet\bullet))))))&\leq 0.0181066\\
    (\bullet((\bullet\bullet)((\bullet\bullet)((\bullet\bullet)((\bullet\bullet)(\bullet\bullet))))))&\leq 0.0128573\\
    (\bullet((\bullet\bullet)((\bullet\bullet)((\bullet\bullet)(\bullet(\bullet(\bullet\bullet)))))))&\leq 0.0233441\\
    (\bullet((\bullet\bullet)((\bullet\bullet)(\bullet((\bullet\bullet)(\bullet(\bullet\bullet)))))))&\leq 0.019107\\
    (\bullet((\bullet\bullet)((\bullet\bullet)(\bullet(\bullet((\bullet\bullet)(\bullet\bullet)))))))&\leq 0.0103594\\
    (\bullet((\bullet\bullet)((\bullet\bullet)(\bullet(\bullet(\bullet(\bullet(\bullet\bullet))))))))&\leq 0.048504\\
    (\bullet((\bullet\bullet)(\bullet((\bullet(\bullet\bullet))((\bullet\bullet)(\bullet\bullet))))))&\leq 0.015401\\
    (\bullet((\bullet\bullet)(\bullet((\bullet(\bullet\bullet))(\bullet(\bullet(\bullet\bullet)))))))&\leq 0.0333184\\
    (\bullet((\bullet\bullet)(\bullet((\bullet\bullet)((\bullet\bullet)(\bullet(\bullet\bullet)))))))&\leq 0.0212784\\
    (\bullet((\bullet\bullet)(\bullet((\bullet\bullet)(\bullet((\bullet\bullet)(\bullet\bullet)))))))&\leq 0.009088\\
    (\bullet((\bullet\bullet)(\bullet((\bullet\bullet)(\bullet(\bullet(\bullet(\bullet\bullet))))))))&\leq 0.0268518\\
    (\bullet((\bullet\bullet)(\bullet(\bullet((\bullet(\bullet\bullet))(\bullet(\bullet\bullet)))))))&\leq 0.0161156\\
    (\bullet((\bullet\bullet)(\bullet(\bullet((\bullet\bullet)((\bullet\bullet)(\bullet\bullet)))))))&\leq 0.0118507\\
    (\bullet((\bullet\bullet)(\bullet(\bullet((\bullet\bullet)(\bullet(\bullet(\bullet\bullet))))))))&\leq 0.023627\\
    (\bullet((\bullet\bullet)(\bullet(\bullet(\bullet((\bullet\bullet)(\bullet(\bullet\bullet))))))))&\leq 0.0321444\\
    (\bullet((\bullet\bullet)(\bullet(\bullet(\bullet(\bullet((\bullet\bullet)(\bullet\bullet))))))))&\leq 0.0226679\\
    (\bullet((\bullet\bullet)(\bullet(\bullet(\bullet(\bullet(\bullet(\bullet(\bullet\bullet)))))))))&\leq 0.140389\\
    (\bullet(\bullet(((\bullet\bullet)(\bullet\bullet))((\bullet\bullet)(\bullet(\bullet\bullet))))))&\leq 0.0429334\\
    (\bullet(\bullet(((\bullet\bullet)(\bullet\bullet))(\bullet((\bullet\bullet)(\bullet\bullet))))))&\leq 0.0111578\\
    (\bullet(\bullet(((\bullet\bullet)(\bullet\bullet))(\bullet(\bullet(\bullet(\bullet\bullet)))))))&\leq 0.0334201\\
    (\bullet(\bullet((\bullet(\bullet(\bullet\bullet)))((\bullet\bullet)(\bullet(\bullet\bullet))))))&\leq 0.058034\\
    (\bullet(\bullet((\bullet(\bullet(\bullet\bullet)))(\bullet((\bullet\bullet)(\bullet\bullet))))))&\leq 0.0285983\\
    (\bullet(\bullet((\bullet(\bullet(\bullet\bullet)))(\bullet(\bullet(\bullet(\bullet\bullet)))))))&\leq 0.1470347\\
    (\bullet(\bullet((\bullet(\bullet\bullet))((\bullet(\bullet\bullet))(\bullet(\bullet\bullet))))))&\leq 0.0330857\\
    (\bullet(\bullet((\bullet(\bullet\bullet))((\bullet\bullet)((\bullet\bullet)(\bullet\bullet))))))&\leq 0.0212749\\
    (\bullet(\bullet((\bullet(\bullet\bullet))((\bullet\bullet)(\bullet(\bullet(\bullet\bullet)))))))&\leq 0.0459748\\
    (\bullet(\bullet((\bullet(\bullet\bullet))(\bullet((\bullet\bullet)(\bullet(\bullet\bullet)))))))&\leq 0.0338421\\
    (\bullet(\bullet((\bullet(\bullet\bullet))(\bullet(\bullet((\bullet\bullet)(\bullet\bullet)))))))&\leq 0.0191234\\
    (\bullet(\bullet((\bullet(\bullet\bullet))(\bullet(\bullet(\bullet(\bullet(\bullet\bullet))))))))&\leq 0.0980258\\
    (\bullet(\bullet((\bullet\bullet)((\bullet(\bullet\bullet))((\bullet\bullet)(\bullet\bullet))))))&\leq 0.0260933\\
    (\bullet(\bullet((\bullet\bullet)((\bullet(\bullet\bullet))(\bullet(\bullet(\bullet\bullet)))))))&\leq 0.0517623\\
    (\bullet(\bullet((\bullet\bullet)((\bullet\bullet)((\bullet\bullet)(\bullet(\bullet\bullet)))))))&\leq 0.0337782\\
    (\bullet(\bullet((\bullet\bullet)((\bullet\bullet)(\bullet((\bullet\bullet)(\bullet\bullet)))))))&\leq 0.0118\\
    (\bullet(\bullet((\bullet\bullet)((\bullet\bullet)(\bullet(\bullet(\bullet(\bullet\bullet))))))))&\leq 0.038892\\
    (\bullet(\bullet((\bullet\bullet)(\bullet((\bullet(\bullet\bullet))(\bullet(\bullet\bullet)))))))&\leq 0.0174256\\
    (\bullet(\bullet((\bullet\bullet)(\bullet((\bullet\bullet)((\bullet\bullet)(\bullet\bullet)))))))&\leq 0.0131765\\
    (\bullet(\bullet((\bullet\bullet)(\bullet((\bullet\bullet)(\bullet(\bullet(\bullet\bullet))))))))&\leq 0.0270876\\
    (\bullet(\bullet((\bullet\bullet)(\bullet(\bullet((\bullet\bullet)(\bullet(\bullet\bullet))))))))&\leq 0.0297785\\
    (\bullet(\bullet((\bullet\bullet)(\bullet(\bullet(\bullet((\bullet\bullet)(\bullet\bullet))))))))&\leq 0.0192499\\
    (\bullet(\bullet((\bullet\bullet)(\bullet(\bullet(\bullet(\bullet(\bullet(\bullet\bullet)))))))))&\leq 0.1077698\\
    (\bullet(\bullet(\bullet(((\bullet\bullet)(\bullet\bullet))((\bullet\bullet)(\bullet\bullet))))))&\leq 0.0136964\\
    (\bullet(\bullet(\bullet(((\bullet\bullet)(\bullet\bullet))(\bullet(\bullet(\bullet\bullet)))))))&\leq 0.0364981\\
    (\bullet(\bullet(\bullet((\bullet(\bullet(\bullet\bullet)))(\bullet(\bullet(\bullet\bullet)))))))&\leq 0.0716329\\
    (\bullet(\bullet(\bullet((\bullet(\bullet\bullet))((\bullet\bullet)(\bullet(\bullet\bullet)))))))&\leq 0.0794783\\
    (\bullet(\bullet(\bullet((\bullet(\bullet\bullet))(\bullet((\bullet\bullet)(\bullet\bullet)))))))&\leq 0.0287851\\
    (\bullet(\bullet(\bullet((\bullet(\bullet\bullet))(\bullet(\bullet(\bullet(\bullet\bullet))))))))&\leq 0.1146128\\
    (\bullet(\bullet(\bullet((\bullet\bullet)((\bullet(\bullet\bullet))(\bullet(\bullet\bullet)))))))&\leq 0.0339684\\
    (\bullet(\bullet(\bullet((\bullet\bullet)((\bullet\bullet)((\bullet\bullet)(\bullet\bullet)))))))&\leq 0.0241595\\
    (\bullet(\bullet(\bullet((\bullet\bullet)((\bullet\bullet)(\bullet(\bullet(\bullet\bullet))))))))&\leq 0.0438161\\
    (\bullet(\bullet(\bullet((\bullet\bullet)(\bullet((\bullet\bullet)(\bullet(\bullet\bullet))))))))&\leq 0.0391523\\
    (\bullet(\bullet(\bullet((\bullet\bullet)(\bullet(\bullet((\bullet\bullet)(\bullet\bullet))))))))&\leq 0.0207505\\
    (\bullet(\bullet(\bullet((\bullet\bullet)(\bullet(\bullet(\bullet(\bullet(\bullet\bullet)))))))))&\leq 0.0960874\\
    (\bullet(\bullet(\bullet(\bullet((\bullet(\bullet\bullet))((\bullet\bullet)(\bullet\bullet)))))))&\leq 0.0613767\\
    (\bullet(\bullet(\bullet(\bullet((\bullet(\bullet\bullet))(\bullet(\bullet(\bullet\bullet))))))))&\leq 0.1333636\\
    (\bullet(\bullet(\bullet(\bullet((\bullet\bullet)((\bullet\bullet)(\bullet(\bullet\bullet))))))))&\leq 0.0765637\\
    (\bullet(\bullet(\bullet(\bullet((\bullet\bullet)(\bullet((\bullet\bullet)(\bullet\bullet))))))))&\leq 0.0284654\\
    (\bullet(\bullet(\bullet(\bullet((\bullet\bullet)(\bullet(\bullet(\bullet(\bullet\bullet)))))))))&\leq 0.0931223\\
    (\bullet(\bullet(\bullet(\bullet(\bullet((\bullet(\bullet\bullet))(\bullet(\bullet\bullet))))))))&\leq 0.08385\\
    (\bullet(\bullet(\bullet(\bullet(\bullet((\bullet\bullet)((\bullet\bullet)(\bullet\bullet))))))))&\leq 0.0581878\\
    (\bullet(\bullet(\bullet(\bullet(\bullet((\bullet\bullet)(\bullet(\bullet(\bullet\bullet)))))))))&\leq 0.1129824\\
    (\bullet(\bullet(\bullet(\bullet(\bullet(\bullet((\bullet\bullet)(\bullet(\bullet\bullet)))))))))&\leq 0.1857998\\
    (\bullet(\bullet(\bullet(\bullet(\bullet(\bullet(\bullet((\bullet\bullet)(\bullet\bullet)))))))))&\leq 0.1394359\\
    (\bullet(\bullet(\bullet(\bullet(\bullet(\bullet(\bullet(\bullet(\bullet(\bullet\bullet))))))))))&\leq 1.0000029^*\\
    \end{align*}}
    
\end{multicols}
\section{Profiles of trees}
\label{app:profiles}

We provide here the outer approximations of the tree-profiles of eight pairs of trees between $4$ and $6$ leaves. The approximations obtained from levels $6$ to $11$ of the hierarchy are represented: the approximations become obviously tighter as the level increases.

\begin{center}
	\begin{figure}[ht!]
		\includegraphics[width=.9\linewidth]{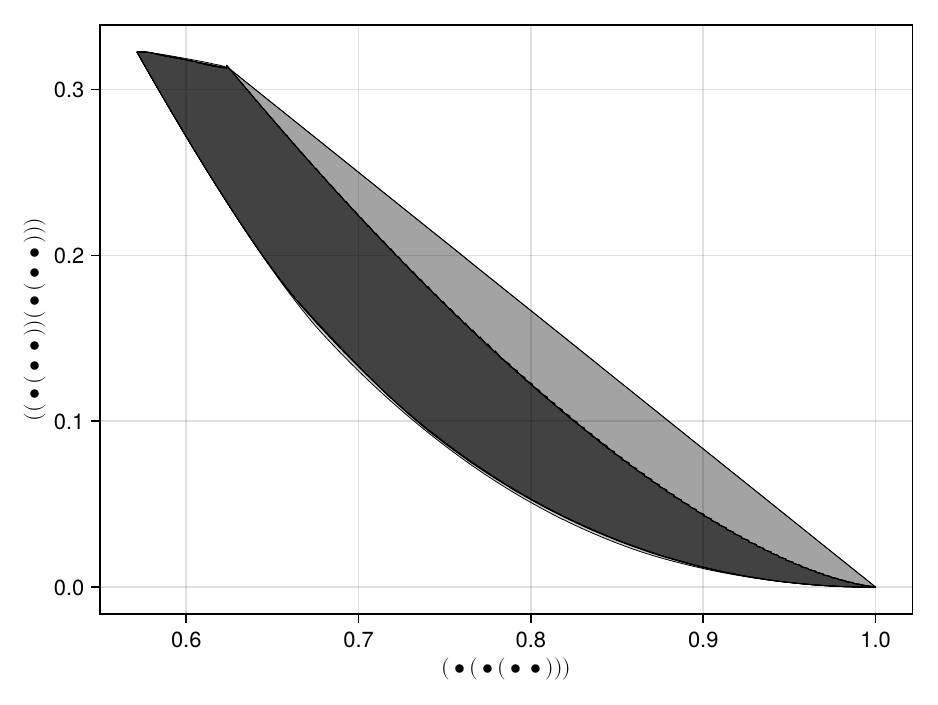}
		\vspace{-.45cm}
		\caption{Outer approximations of the tree-profile of  \begin{minipage}{0.06\textwidth}
				\resizebox{.8\textwidth}{!}{%
					\begin{forest}
						[, treeNodeRoot
								[, treeNode]
								[, treeNodeInner[, treeNode][, treeNodeInner[, treeNode],[, treeNode]]]]
					\end{forest}}\end{minipage} \hspace{-.2cm}
			and \begin{minipage}{0.09\textwidth}
				\resizebox{.7\textwidth}{!}{%
					\begin{forest}
						[, treeNodeRoot
								[, treeNodeInner[, treeNode][, treeNodeInner[, treeNode][, treeNode]]]
								[, treeNodeInner[, treeNode][, treeNodeInner[, treeNode][, treeNode]]]]
					\end{forest}}\end{minipage}}
		\label{fig:treeProfile1}
	\end{figure}
	\begin{figure}[ht!]
		\includegraphics[width=.9\linewidth]{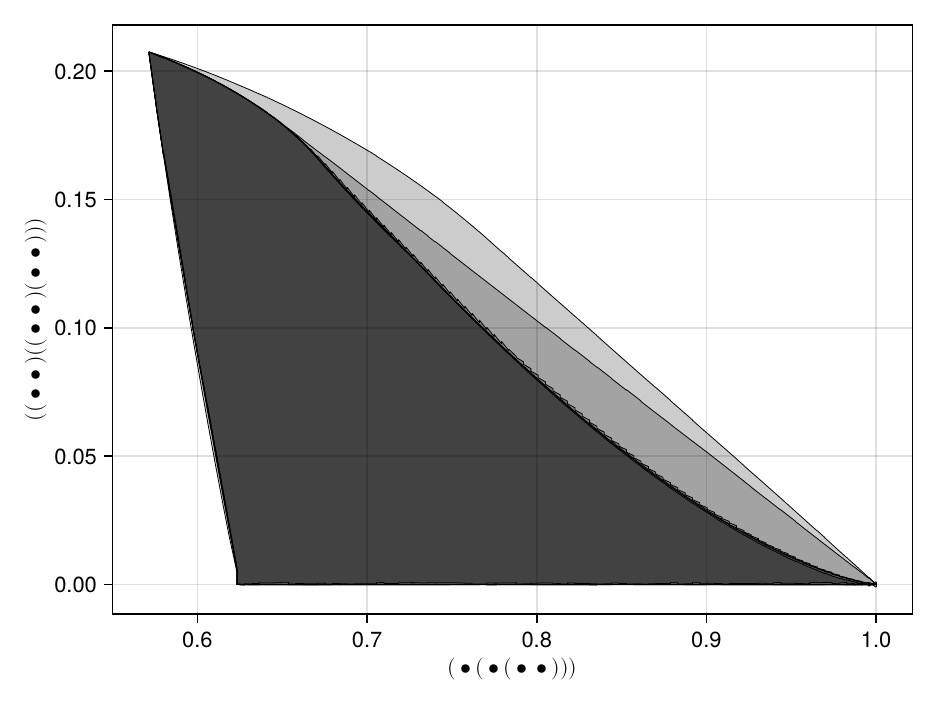}
		\vspace{-.45cm}
		\caption{Outer approximations of the tree-profile of \begin{minipage}{0.06\textwidth}
				\resizebox{.8\textwidth}{!}{%
					\begin{forest}
						[, treeNodeRoot
								[, treeNode]
								[, treeNodeInner[, treeNode][, treeNodeInner[, treeNode],[, treeNode]]]]
					\end{forest}}\end{minipage} \hspace{-.2cm}
			and \begin{minipage}{0.09\textwidth}
				\resizebox{.7\textwidth}{!}{%
					\begin{forest}
						[, treeNodeRoot
								[, treeNodeInner[, treeNode][, treeNode]]
								[, treeNodeInner[, treeNodeInner[, treeNode][, treeNode]][, treeNodeInner[, treeNode][, treeNode]]]]
					\end{forest}}\end{minipage}}
		\label{fig:treeProfile2}
	\end{figure}
	\begin{figure}[ht!]
		\includegraphics[width=.9\linewidth]{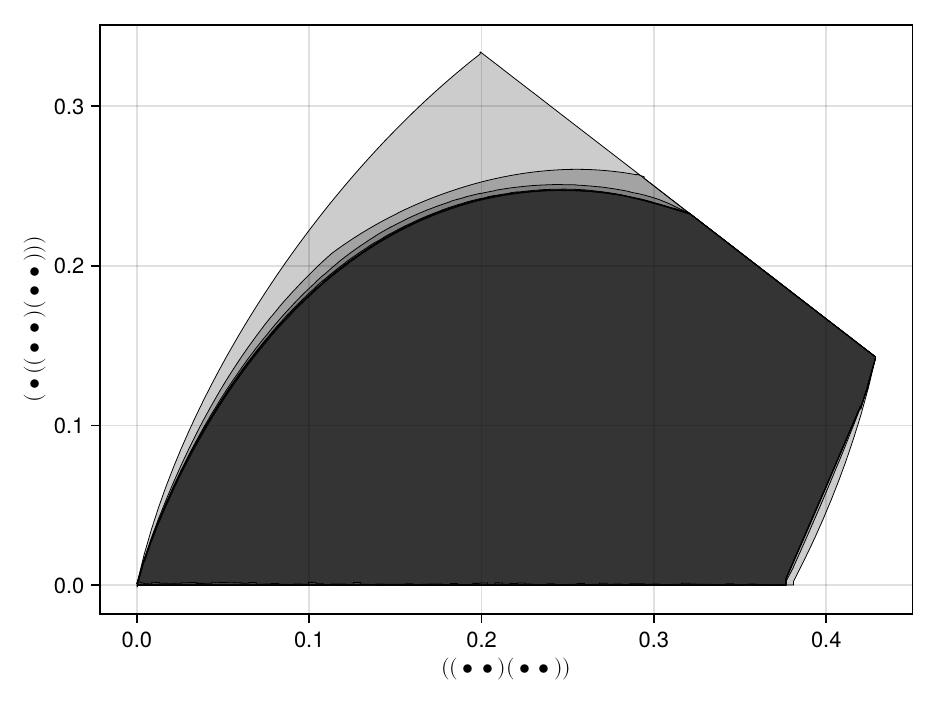}
		\vspace{-.45cm}
		\caption{Outer approximations of the tree-profile of \begin{minipage}{0.06\textwidth}
				\resizebox{\textwidth}{!}{%
					\begin{forest}
						[, treeNodeRoot
								[, treeNodeInner[, treeNode],[, treeNode]]
								[, treeNodeInner[, treeNode],[, treeNode]]]
					\end{forest}}\end{minipage}
			and \begin{minipage}{0.06\textwidth}
				\centering
				\resizebox{\textwidth}{!}{%
					\begin{forest}
						[, treeNodeRoot
								[, treeNode]
								[, treeNodeInner[, treeNodeInner[, treeNode][, treeNode]]
										[, treeNodeInner[, treeNode][, treeNode]]]]
					\end{forest}}\end{minipage}}
		\label{fig:treeProfile3}
	\end{figure}
	\begin{figure}[ht!]
		\includegraphics[width=.9\linewidth]{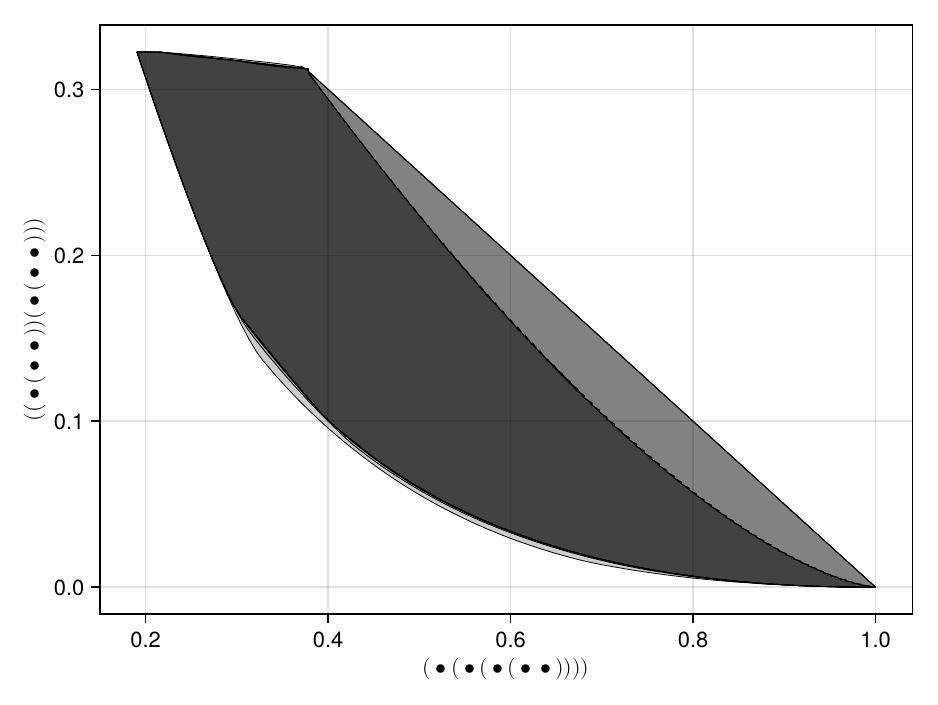}
		\vspace{-.45cm}
		\caption{Outer approximations of the tree-profile of \begin{minipage}{0.1\textwidth}
				\resizebox{.6\textwidth}{!}{%
					\begin{forest}
						[, treeNodeRoot
								[, treeNode]
								[, treeNodeInner[, treeNode][, treeNodeInner[, treeNode],[, treeNodeInner[, treeNode],[, treeNode]]]]]
					\end{forest}}\end{minipage} \hspace{-.5cm} and \begin{minipage}{0.1\textwidth}
				\resizebox{.7\textwidth}{!}{%
					\begin{forest}
						[, treeNodeRoot
								[, treeNodeInner[, treeNode][, treeNodeInner[, treeNode][, treeNode]]]
								[, treeNodeInner[, treeNode][, treeNodeInner[, treeNode][, treeNode]]]]
					\end{forest}}\end{minipage}}
		\label{fig:treeProfile4}
	\end{figure}
	\begin{figure}[ht!]
		\includegraphics[width=.9\linewidth]{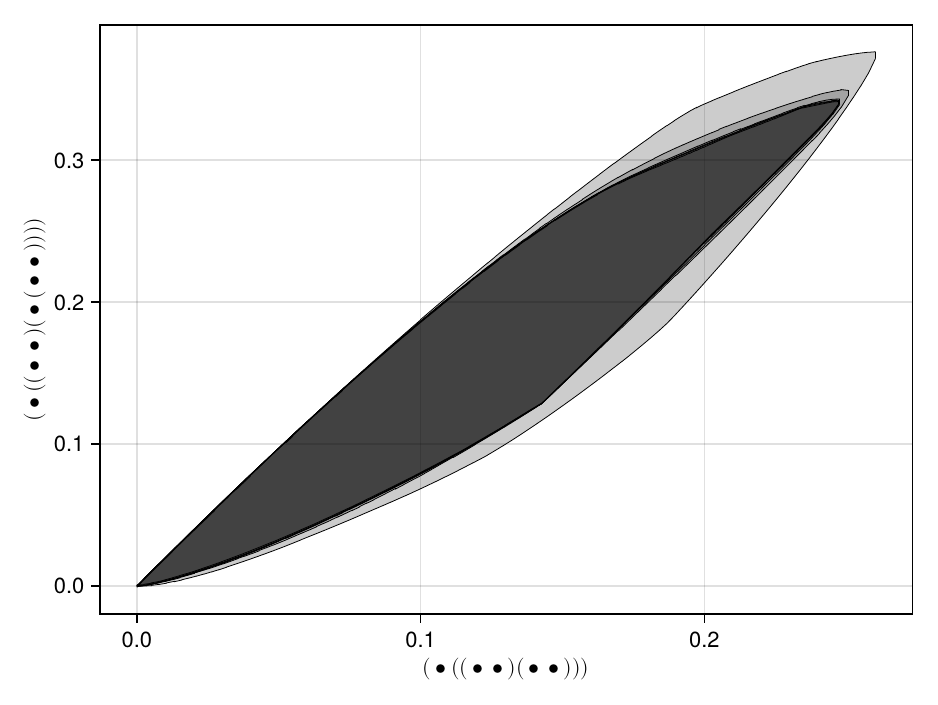}
		\vspace{-.45cm}
		\caption{Outer approximations of the tree-profile of \begin{minipage}{0.05\textwidth}
				\centering
				\resizebox{\textwidth}{!}{%
					\begin{forest}
						[, treeNodeRoot
								[, treeNode]
								[, treeNodeInner[, treeNodeInner[, treeNode][, treeNode]]
										[, treeNodeInner[, treeNode][, treeNode]]]]
					\end{forest}}\end{minipage} and \begin{minipage}{0.1\textwidth}
				\resizebox{.7\textwidth}{!}{%
					\begin{forest}
						[, treeNodeRoot
								[, treeNode]
								[, treeNodeInner[, treeNodeInner[, treeNode],[, treeNode]]
										[, treeNodeInner[, treeNode],[, treeNodeInner[, treeNode],[, treeNode]]]]]
					\end{forest}}\end{minipage}}
		\label{fig:treeProfile5}
	\end{figure}
	\begin{figure}[ht!]
		\includegraphics[width=.9\linewidth]{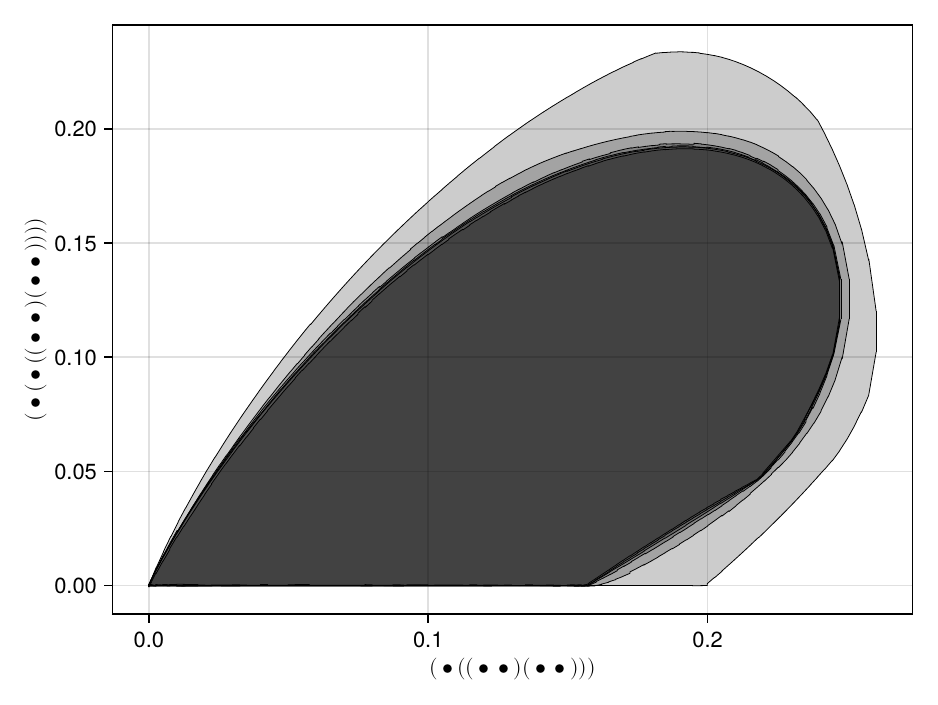}
		\vspace{-.45cm}
		\caption{Outer approximations of the tree-profile of \begin{minipage}{0.05\textwidth}
				\centering
				\resizebox{\textwidth}{!}{%
					\begin{forest}
						[, treeNodeRoot
								[, treeNode]
								[, treeNodeInner[, treeNodeInner[, treeNode][, treeNode]]
										[, treeNodeInner[, treeNode][, treeNode]]]]
					\end{forest}}\end{minipage} and  \begin{minipage}{0.1\textwidth}
				\resizebox{.6\textwidth}{!}{%
					\begin{forest}
						[, treeNodeRoot
								[, treeNode]
								[, treeNodeInner[, treeNode][, treeNodeInner[, treeNodeInner[, treeNode],[, treeNode]][, treeNodeInner[, treeNode],[, treeNode]]]]] \end{forest}}\end{minipage}}
		\label{fig:treeProfile6}
	\end{figure}
	\begin{figure}[ht!]
		\includegraphics[width=.9\linewidth]{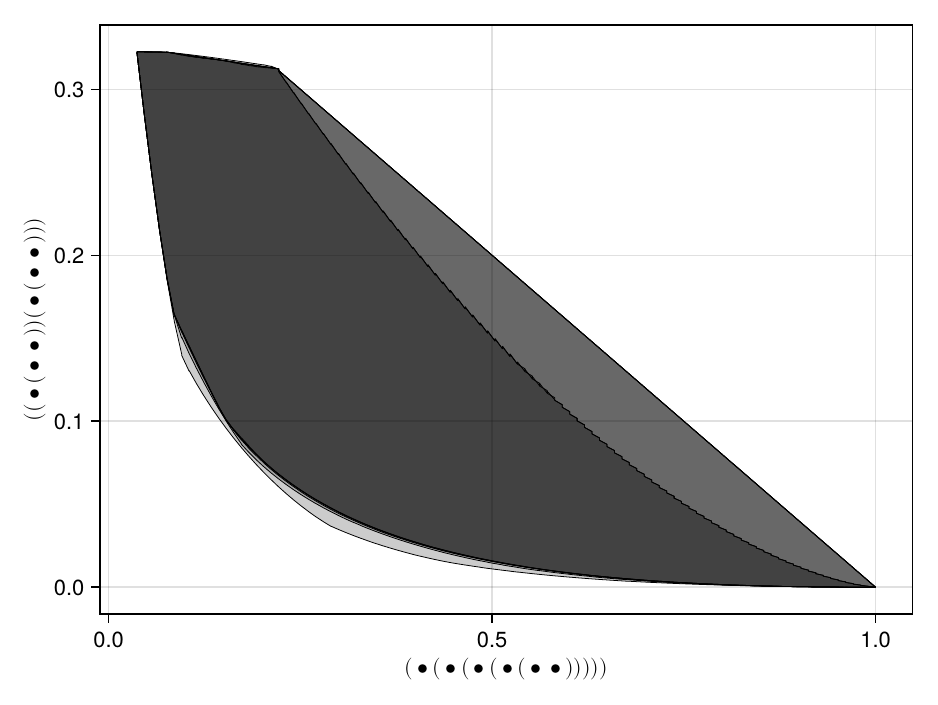}
		\vspace{-.45cm}
		\caption{Outer approximations of the tree-profile of \begin{minipage}{0.1\textwidth}
				\resizebox{.7\textwidth}{!}{%
					\begin{forest}
						[, treeNodeRoot
								[, treeNode]
								[, treeNodeInner[, treeNode][, treeNodeInner[, treeNode],[, treeNodeInner[, treeNode],[, treeNodeInner[, treeNode],[, treeNode]]]]]]
					\end{forest}}\end{minipage} \hspace{-.5cm} and \begin{minipage}{0.1\textwidth}
				\resizebox{.7\textwidth}{!}{%
					\begin{forest}
						[, treeNodeRoot
								[, treeNodeInner[, treeNode][, treeNodeInner[, treeNode][, treeNode]]]
								[, treeNodeInner[, treeNode][, treeNodeInner[, treeNode][, treeNode]]]]
					\end{forest}}\end{minipage}}
		\label{fig:treeProfile7}
	\end{figure}
	\begin{figure}[ht!]
		\includegraphics[width=.9\linewidth]{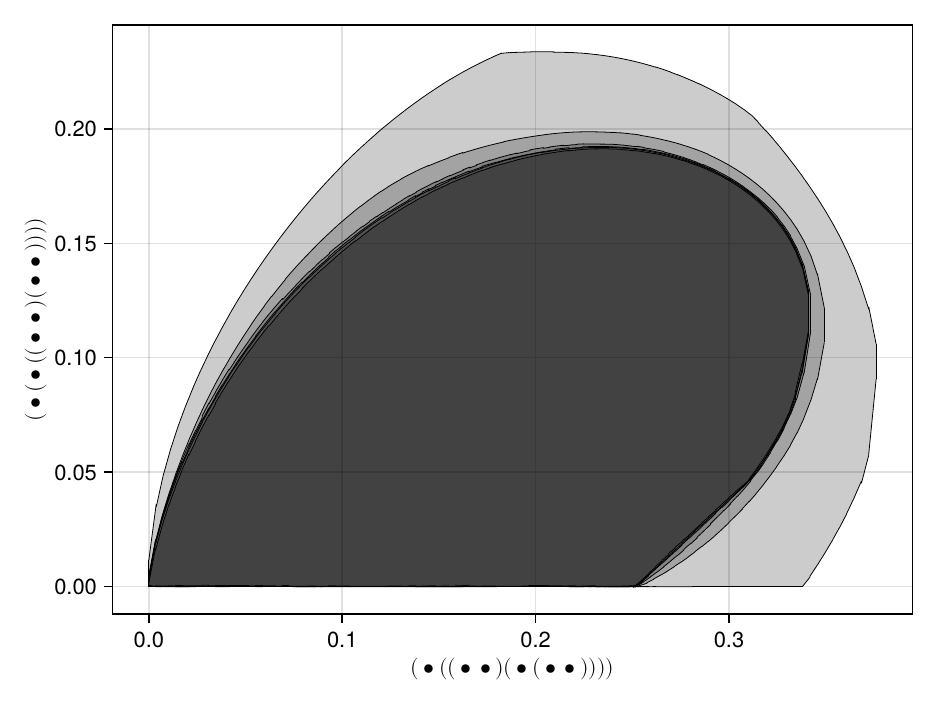}
		\vspace{-.45cm}
		\caption{Outer approximations of the tree-profile of \begin{minipage}{0.1\textwidth}
				\resizebox{.75\textwidth}{!}{%
					\begin{forest}
						[, treeNodeRoot
								[, treeNode]
								[, treeNodeInner[, treeNodeInner[, treeNode],[, treeNode]]
										[, treeNodeInner[, treeNode],[, treeNodeInner[, treeNode],[, treeNode]]]]]
					\end{forest}}\end{minipage} \hspace{-.5cm} and  \begin{minipage}{0.1\textwidth}
				\resizebox{.6\textwidth}{!}{%
					\begin{forest}
						[, treeNodeRoot
								[, treeNode]
								[, treeNodeInner[, treeNode][, treeNodeInner[, treeNodeInner[, treeNode],[, treeNode]][, treeNodeInner[, treeNode],[, treeNode]]]]] \end{forest}}\end{minipage}}
		\label{fig:treeProfile8}
	\end{figure}
\end{center}

\end{document}